\documentclass[10pt,twoside]{article}
\usepackage{mathrsfs}
\usepackage{amsmath}
\usepackage{amssymb}
\usepackage{fancyhdr}
\usepackage{latexsym}
\usepackage{bbding}
\usepackage{mathrsfs}
\usepackage{wasysym}
\usepackage{cite}

\usepackage{multicol,graphics}

\newtheorem{theorem}{Theorem}[section]
\newtheorem{lemma}[theorem]{Lemma}

\newtheorem{definition}[theorem]{Definition}
\newtheorem{remark}[theorem]{Remark}
\newtheorem{proposition}[theorem]{Proposition}
\numberwithin{equation}{section}
\newenvironment{proof}[1][Proof]{\noindent\textbf{#1.} }{\hfill $\Box$}
\allowdisplaybreaks

 \makeatletter
\begin{document}
\title{{Global well-posedness of  the  2D nonhomogeneous incompressible nematic liquid crystal flows with vacuum}
\thanks{Qiao Liu is partially supported by the National Natural Science Foundation of China (11401202),  the Scientific Research Fund of Hunan Provincial
Education Department (14B117), and the China Postdoctoral Science Foundation (2015M570053).}
}
\author{
{\small  
 Qiao Liu$^{1,2}$ \thanks{\text{Corresponding author. E-mail address}: liuqao2005@163.com.
}}
{\small \quad  
Shengquan Liu$^{2}$ \thanks{\text{E-mail address}: shquanliu@163.com.
}}
{\small \quad  
Wenke Tan$^{1,3}$ \thanks{\text{E-mail address}: tanwenkeybfq@163.com.
}}
{\small \quad and\quad
Xin Zhong$^{3}$ \thanks{\text{E-mail address}: xzhong1014@amss.ac.cn.
}}
\medskip
\\
{\small  $^{1}$\textit{Department of Mathematics, Hunan Normal University, Changsha, Hunan 410081, P. R. China}}\\
{\small  $^{2}$\textit{Institute of Applied Physics and Computational Mathematics, Beijing 100088, P. R. China}}\\
{\small  $^{3}$\textit{Institute of Applied Mathematics, AMSS, Chinese Academy of Sciences, Beijing 100190, P. R. China}}
}
\date{}
\maketitle

\begin{abstract}
This paper concerns the Cauchy problem of the two-dimensional (2D) nonhomogeneous incompressible nematic liquid crystal flows on the
whole space $\mathbb{R}^{2}$ with vacuum as far field density. It is proved that the 2D nonhomogeneous incompressible nematic liquid crystal
flows admits a unique global strong solution provided the initial data density and the gradient of orientation decay not too slow at infinity,
 and the initial orientation satisfies a geometric condition (see \eqref{eq1.3}).
In particular, the initial data can be arbitrarily large and the initial density  may contain vacuum states and even have compact support. Furthermore, the large time behavior of the solution is also obtained.

\medskip


\end{abstract}

\section{Introduction}\label{Int}

\noindent

Liquid crystals can form and remain in an intermediate
phase of matter between liquids and solids. When a solid melts, if
the energy gain is enough to overcome the positional order but the
shape of the molecules prevents the immediate collapse of
orientational order, liquid crystals are formed.  The lack of
positional order is a shared property of liquid crystals and
liquids; on the other hand, liquid crystals are anisotropic (like
solids).  
The nematic liquid crystals
exhibit long¡ªrange ordering in the sense that their rigid rod-like
molecules arrange themselves with their long axes parallel to each
other. Their molecules float around as in a liquid, but have the
tendency to align along a preferred direction due to their
orientation. The hydrodynamic theory of the nematic liquid crystals was first derived by
 Ericksen and Leslie  during the period of 1958
through 1968 (see \cite{ER,LE}). A
brief account of the Ericksen-Leslie theory on nematic liquid
crystal flows and the derivations of several approximate systems can
be found in the appendix of \cite{LL1}. For more details on the
hydrodynamic continuum theory of liquid crystals, we refer the
readers to the book of Stewart \cite{IWS}.

In this paper, we  investigate  the global existence of
solutions to the following two dimensional (2D) simplified version of
nematic liquid crystal flow in the whole space $\mathbb{R}^{2}$, which describes the
motion of a nonhomogeneous incompressible flow of
 nematic liquid crystals:
(see, e.g., \cite{L,LLW}):
\begin{align}
   \label{eq1.1}
 \begin{cases}
 \varrho_{t}+\operatorname{div} (\varrho u) =0,\\
\varrho u_{t}+\varrho u\cdot\nabla u-\nu\Delta u +\nabla{P}=-\lambda\operatorname{div}(\nabla d \odot\nabla d),\\
d_{t}+(u\cdot\nabla)d=\gamma(\Delta d+|\nabla d|^{2}d),\\
\quad\operatorname{div} u=0,\quad\quad |d|=1,
\end{cases}
\end{align}
where $\varrho(x,t): \mathbb{R}^{2}\times (0,+\infty)\rightarrow \mathbb{R}$ is the density, $u(x,t):\mathbb{R}^{2}\times (0,+\infty)\rightarrow
\mathbb{R}^{2}$ is the unknown velocity field of the flow,
$P(x,t):\mathbb{R}^{2}\times (0,+\infty)\rightarrow \mathbb{R}$ is
the scalar pressure and $d:\mathbb{R}^{2}\times
(0,+\infty)\rightarrow \mathbb{S}^{2}$, the unit sphere in
$\mathbb{R}^{3}$, is the unknown (averaged) macroscopic/continuum
molecule orientation of the nematic liquid crystal flow, $\operatorname{div} u=0$ represents the incompressible condition,  and  $\nu$, $\lambda$ and $\gamma$ are positive numbers
associated to the properties of the material: $\nu$ is the kinematic
viscosity, $\lambda$ is the competition between kinetic energy and
potential energy, and $\gamma$ is the microscopic elastic relaxation
time for the molecular orientation field. The notation $\nabla
d\odot\nabla d$ denotes the $2\times 2$ matrix whose $(i,j)$-th
entry is given by $\partial_{i}d\cdot
\partial_{j}d$ ($1\leq i,j\leq 2$).

We consider the Cauchy problem for \eqref{eq1.1}  with 
the initial
conditions
\begin{align}\label{eq1.2}
 &\varrho(x,0)=\varrho_{0}(x),\quad u(x,0)=u_{0}(x), \quad d(x,0)=d_{0}(x),\quad |d_{0}(x)|=1, \quad\text{ in }\mathbb{R}^{2}
\end{align}
for given initial data $\varrho_{0}$, $u_{0}$, and $d_{0}$. Since the concrete values of $\nu$, $\lambda$ and $\gamma$ do
not play a special role in our discussion, for simplicity, we assume
that they are all equal to one throughout this paper.

The above system \eqref{eq1.1}--\eqref{eq1.2}  is a simplified
version of the Ericksen-Leslie model \cite{ER,LE}, but it still
retains most important  mathematical structures as well as most of
the essential difficulties of the original Ericksen-Leslie model.  Mathematically, system \eqref{eq1.1}--\eqref{eq1.2} is a strongly coupled system between
the nonhomogeneous incompressible Navier-Stokes equations (see, e.g., \cite{ZL,LXZ,Lions1,Temam})
 and the transported
heat flows of harmonic map (see, e.g., \cite{CDY,LLZ,W}), and thus, its mathematical analysis is full
of challenges.

There is a huge literature on the studies about global existence and behaviors of solutions to \eqref{eq1.1}. The important progress on the global existence of strong or weak solutions in multi-dimension has been made by many authors, refer to \cite{DHX,LJK,LJK2,FLN,WD} and references therein.
However, since the system \eqref{eq1.1} contains the incompressible Navier-Stokes equations as a subsystem, one cannot expect, in general, any better results than those for the Navier-Stokes equations. In the homogeneous case, i.e., $\varrho\equiv \text{constant}$, Lin-Lin-Wang in \cite{LLW}
established that there exists global Leray--Hopf type weak solutions
to the initial boundary value problem for system
\eqref{eq1.1}--\eqref{eq1.2} on bounded domains in two space
dimensions (see also \cite{HMC}). The uniqueness of such weak
solutions is proved by Lin-Wang \cite{LW}, see also Xu-Zhang \cite{XZ} for related works. When the space dimension is three, Lin-Wang \cite{LW1}
obtained the existence of global weak solutions very recently when the initial data $(u_{0},d_{0})\in L^{2}\times H^{1}$ with the initial director field $d_{0}$ maps to the upper hemisphere $\mathbb{S}_{+}^{2}$.  The global existence of weak solutions to homogeneous  system \eqref{eq1.1}--\eqref{eq1.2} with general initial data in dimension three is still an open problem. There are also some studies on the homogeneous incompressible nematic liquid crystal flows, we refer the readers  to \cite{LL1,LL2,LZZ,JHW} and references therein.

In the non-homogeneous case, i.e., the density dependent case, when the initial data away from vacuum, recently, Li \cite{LJK2} established the global existence of strong and weak solutions to the two-dimensional system \eqref{eq1.1}--\eqref{eq1.2} provided that the initial  orientation $d_{0}=(d_{01},d_{02},d_{03})$ satisfies a geometric condition
\begin{equation}\label{eq1.3}
d_{03}\geq\varepsilon_{0}\ \text{for some positive}\ \varepsilon_{0}>0.
\end{equation}
In the presence of vacuum,
if the initial data is small (in some sense) and satisfies the following  compatibility conditions
\begin{align}\label{eq1.4}
-\mu\Delta u_{0}+\nabla P_{0}+\lambda\operatorname{div}(\nabla d_{0}\odot\nabla d_{0})=\varrho_{0}^{\frac{1}{2}} g_{0}
\end{align}
in a bounded smooth domain $\Omega\subseteq\mathbb{R}^2$, and $(P_{0},g_{0})\in H^{1}(\Omega)\times L^{2}(\Omega)$, Wen-Ding \cite{WD} obtained the global existence and uniqueness of the strong solutions to the system \eqref{eq1.1}--\eqref{eq1.2}, see also \cite{FLN,LJK} for related works. It should be emphasized the possible appearance of vacuum is one
of the main difficulties, which indeed leads to the singular behaviors of solutions in the presence of vacuum, such as the finite time blow-up of smooth solutions \cite{HW1}.

It is not known in general about the existence of global strong solutions to the
problem \eqref{eq1.1}--\eqref{eq1.2} in two-dimension without the geometric condition \eqref{eq1.3}
or the compatibility condition \eqref{eq1.4} imposed on the initial data. This problem is rather interesting
 and hard to investigate. Indeed, it should be noted that the previous studies on the heat flow of a
 harmonic map \cite{CDY} indicate that the strong solution of a harmonic map can be blow-up in finite time.
 In our case, since the system \eqref{eq1.1} contains the heat flow of a harmonic map as a subsystem,
 we cannot expect that \eqref{eq1.1} have a global strong solution with general initial data.
 This makes the analysis rather delicate and difficult.



It should be noticed that when $d$ is a constant vector and $|d|=1$, the system \eqref{eq1.1} reduces to the nonhomogeneous incompressible Navier--Stokes equations, which have been discussed in
numerous studies \cite{CK2003,Danchin03,AK,Lions1,LSZ} and so on. In the case when $\varrho_{0}$ is bounded away from zero,  Kazhikov \cite{AK} established the global existence of weak solutions. Danchin \cite{Danchin03} obtained the existence of local and unique strong solutions in the so-called critical Besov space. When the initial data may contain vacuum states, Lions \cite{Lions1} obtained the the global existence of weak solutions.  Choe-Kim \cite{CK2003} proposed a compatibility condition and established the local existence of strong solutions under some suitable smallness conditions. The global existence of strong solutions on bounded domains were established by Huang-Wang \cite{HW2014}. Recently, inspired by Li-Liang \cite{LJLZ},  L{\"u}-Shi-Zhong \cite{LSZ} established the global existence of strong solutions to the 2D Cauchy problem of the incompressible Navier-Stokes equations on the whole space $\mathbb{R}^{2}$ with vacuum as far field density. However, the global existence of strong solution with large data to the 2D Cauchy problem \eqref{eq1.1}--\eqref{eq1.2} with  vacuum as far field density is still open. In fact, this is the main aim of the present paper.

Before formulating our main result, let us first  define precisely what we mean by strong solutions.

\begin{definition}\label{def1.1}
If all derivatives involved in \eqref{eq1.1} for $(\varrho, u, P, d)$ are regular distributions, and equations \eqref{eq1.1} hold almost everywhere in $\mathbb{R}^{2}\times (0,T)$, then $(\varrho, u, P, d)$ is called a strong solution to \eqref{eq1.1}.
\end{definition}

Without loss of generality,  we assume that the initial density $\varrho_{0}$ satisfies
\begin{align}\label{eq1.5}
\int_{\mathbb{R}^{2}} \varrho_{0} \text{d}x=1,
\end{align}
which implies that there exists a positive constant $N_{0}$ such that
\begin{align}\label{eq1.6}
\int_{B_{N_{0}}}\varrho_{0}\text{d}x\geq \frac{1}{2}\int_{\mathbb{R}^{2}}\varrho_{0}\text{d}x
= \frac{1}{2},
\end{align}
where $B_{R}\triangleq \{x\in\mathbb{R}^{2} | |x|<R\}$ for all $R>0$.

Now, we state our main result as follows:

\begin{theorem}\label{thm1.2}
For constants $q>2$, $a>1$, assume that the initial data $(\varrho_{0}, u_{0}, d_{0})$ satisfies \eqref{eq1.3}, \eqref{eq1.5}, \eqref{eq1.6}, and
\begin{align}\label{eq1.7}
\begin{cases}
\varrho_{0}\geq 0, \quad \varrho_{0}\bar{x}^{a}\in L^{1}(\mathbb{R}^{2})\cap H^{1}(\mathbb{R}^{2})\cap W^{1,p}(\mathbb{R}^{2}), \quad
\varrho^{\frac{1}{2}}_{0}u_{0}\in L^{2}(\mathbb{R}^{2}), \quad \nabla u_{0}\in L^{2}(\mathbb{R}^{2}),\\
\operatorname{div} u_{0}=0,\quad d_{0}\in L^{2}(\mathbb{R}^{2}),\quad  \nabla d_{0} \bar{x}^{\frac{a}{2}}\in L^{2}(\mathbb{R}^{2}),\quad \nabla^{2} d_{0}\in L^{2}(\mathbb{R}^{2}),
\quad |d_{0}|=1,
\end{cases}
\end{align}
where
\begin{align*}
\bar{x}\triangleq (e+|x|^{2})^{\frac{1}{2}}\ln^{2} (e+|x|^{2}).
\end{align*}
Then the Cauchy problem \eqref{eq1.1}--\eqref{eq1.2} has a unique global strong solution $(\varrho, u, P, d)$ satisfying that for any $0<T<\infty$,
\begin{align}\label{eq1.8}
\begin{cases}
0\leq \varrho\in C([0,T]; L^{1}(\mathbb{R}^{2})\cap H^{1}(\mathbb{R}^{2})\cap W^{1,q}(\mathbb{R}^{2})),\\
  \varrho \bar{x}^{a}\in L^{\infty}(0,T; L^{1}(\mathbb{R}^{2})\cap H^{1}(\mathbb{R}^{2})\cap W^{1,q}(\mathbb{R}^{2}))\\
  \varrho^{\frac{1}{2}} u,  \nabla u,   t^{\frac{1}{2}}\nabla u,  t^{\frac{1}{2}}\varrho^{\frac{1}{2}} u,
    t^{\frac{1}{2}}\nabla P, t\nabla P, t^{\frac{1}{2}}\nabla^{2} u, t \nabla^{2} u\in L^{\infty}(0,T; L^{2}(\mathbb{R}^{2})),\\
  \nabla d, \nabla d \bar{x}^{\frac{a}{2}}, \nabla^{2} d, t^{\frac{1}{2}}\nabla^{2} d, t^{\frac{1}{2}}\nabla d_{t}, t^{\frac{1}{2}}\nabla^{3}d
  \in L^{\infty}(0,T; L^{2}(\mathbb{R}^{2})),\\
  \nabla u\in L^{2}(0,T; H^{1}(\mathbb{R}^{2}))\cap L^{\frac{q+1}{q}}(0,T; W^{1,q}(\mathbb{R}^{2})),\\
  \nabla P\in L^{2}(0,T; L^{2}(\mathbb{R}^{2}))\cap L^{\frac{q+1}{q}}(0,T; L^{q}(\mathbb{R}^{2})),\\
  \nabla^{2} d\in L^{2}(0,T; H^{1}(\mathbb{R}^{2})),\, \nabla d_{t}, \nabla^{2} d \bar{x}^{\frac{a}{2}}\in L^{2}(\mathbb{R}^{2}\times (0,T)),\\
  \varrho^{\frac{1}{2}} u_{t},  t^{\frac{1}{2}}\nabla u_{t}, t^{\frac{1}{2}}\nabla d_{t}, t^{\frac{1}{2}}\nabla^{2} d_{t}
  \in L^{2}(\mathbb{R}^{2}\times (0,T)),
  \\
  t^{\frac{1}{2}} \nabla u \in L^{2}(0,T; W^{1,q}(\mathbb{R}^{2})),
\end{cases}
\end{align}
and
\begin{align}\label{eq1.9}
\inf_{0\leq t\leq T} \int_{B_{N_{1}}} \varrho(x,t)\text{d}x\geq \frac{1}{4}
\end{align}
for some positive constant $N_{1}$ depending only on $\|\varrho^{\frac{1}{2}}_{0} u_{0}\|_{L^{2}(\mathbb{R}^{2})}$, $N_{0}$, and $T$. Moreover, the solution
$(\varrho, u, P, d)$ has the following temporal  decay rates, i.e., for all $t\geq 1$,
\begin{align}\label{eq1.10}
\begin{cases}
\|\varrho^{\frac{1}{2}} u_{t}(\cdot,t)\|_{L^{2}(\mathbb{R}^{2})}+\|d_{t}(\cdot,t)\|_{H^{1}(\mathbb{R}^{2})}+\|\nabla u(\cdot,t)\|_{L^{2}(\mathbb{R}^{2})} +\|\nabla^{2} d(\cdot, t)\|_{L^{2}(\mathbb{R}^{2})}\leq C t^{-\frac{1}{2}},\\
 \|\nabla^{2} u(\cdot,t)\|_{L^{2}(\mathbb{R}^{2})} +\||\nabla d||\nabla^{2} d|(\cdot, t)\|_{L^{2}(\mathbb{R}^{2})}
+\|\nabla P(\cdot,t)\|_{L^{2}(\mathbb{R}^{2})}\leq C t^{-1},
\end{cases}
\end{align}
where $C$ depends only on initial datums.
\end{theorem}

\begin{remark}
When $d$ is a constant vector and $|d|=1$, the system \eqref{eq1.1} turns to be the nonhomogeneous incompressible Navier--Stokes equations, and Theorem \ref{thm1.2} is similar to the results of \cite{LSZ}. Roughly speaking, we generalize the results of \cite{LSZ} to the incompressible nematic liquid crystal flows.
\end{remark}

\begin{remark}
Our Theorem \ref{thm1.2} holds for arbitrarily large data which is in sharp contrast to \cite{FLN,LJK,WD} where the smallness conditions on the initial data is needed in order to obtain the global existence of strong solutions. Moreover, the compatibility condition \eqref{eq1.4} on the initial data is also needed in \cite{FLN,LJK,WD}. It seems more involved to show the global existence of strong solutions with general initial data. This is the main reason for us to add an additional geometric condition \eqref{eq1.3}.
\end{remark}

\begin{remark}
Compared  with \cite{LJK2}, there is no need to impose the absence of vacuum for the initial density. Furthermore, it should be pointed out that the large time asymptotic decay with rates of the global strong solution in \eqref{eq1.10} is completely new for the 2D nonhomogeneous nematic liquid crystal flows.
\end{remark}

We now make some comments on the analysis of the present paper. Using some key ideas due to \cite{ZL,LXZ}, where the authors deal with the 2D incompressible Navier--Stokes and MHD equations, respectively, we first establish that if $(\varrho_{0}, u_{0},d_{0})$ satisfies \eqref{eq1.5}--\eqref{eq1.7}, then there exists a small $T_{0}>0$ such that the Cauchy problem \eqref{eq1.1}--\eqref{eq1.2} admits a unique strong solution  $(\varrho,u,P,d)$ in $\mathbb{R}^{2}\times (0,T_{0}]$ satisfying \eqref{eq1.8} and \eqref{eq1.9} (see Theorem \ref{thm3.1}). Thus, to prove Theorem \ref{thm1.2}, we only need to give some global a priori estimates on the strong solutions to system \eqref{eq1.1}--\eqref{eq1.2} in suitable higher norms.
It should be pointed out that the crucial techniques of proofs in \cite{LW,DHX} cannot be adapted to the situation treated here. One reason is that when $\Omega\subseteq \mathbb{R}^{2}$ becomes unbounded, the standard Sobolev embedding inequality is critical, and it seems difficult to bound the $L^{p}$-norm ($p>2$) of the velocity $u$ just in terms of $\|\varrho^{\frac{1}{2}} u\|_{L^{2}(\mathbb{R}^{2})}$ and $\|\nabla u\|_{L^{2}(\mathbb{R}^{2})}$. Moreover, compared with \cite{LSZ,LXZ}, for system \eqref{eq1.1}--\eqref{eq1.2} treated here, the strong coupling terms and strong nonlinear terms, such as $\operatorname{div}(\nabla d\odot\nabla d)$,  $u\cdot\nabla d$ and $|\nabla d|^{2}d$, will bring out some new difficulties.

To overcome these difficulties mentioned above, some new ideas are needed. To deal with the difficulty  caused by the lack of the Sobolev inequality, we observe that, in  equations \eqref{eq1.2}, the velocity $u$ is always accompanied by $\varrho$. Motivated by \cite{LJLZ,LZ,ZL}, by introducing a weighted function to the density, as well as the Hardy-type inequality in \cite{Lions1} by Lions, the $\|\varrho^{\alpha} u\|_{L^{r}(\Omega)}$ ($r>2,\alpha>0$ and $\Omega=B_{R}$ or $\mathbb{R}^{2}$) is controlled in terms of $\|\varrho^{\frac{1}{2}} u\|_{L^{2}(\Omega)}$ and $\|\nabla u\|_{L^{2}(\Omega)}$ (see Lemma \ref{lem2.4}).   After some spatial estimates on $\nabla d$ (i.e., $\nabla d \bar{x}^{\frac{a}{2}}$, see \eqref{eq3.5}), and suitable a priori estimates, we then construct approximate solutions to \eqref{eq1.1}, that is, for density strictly away from vacuum initially, we consider a initial boundary value problem of \eqref{eq1.1} in any bounded ball $B_{R}$ with radius $R>0$.
Combining all these ideas stated above with those due to
{\cite{ZL,LXZ}, we derive some desired bounds on the gradients of the velocity and the spatial weighted ones on both the density and its gradients where all these bounds are independent of both the radius of the balls $B_{R}$ and the lower bound of the initial density, and then obtain the local existence and uniqueness of solution (see Subsection 3.2).

Base on the local existence result (see Theorem \ref{thm3.1}), we try to give some a priori estimates which are needed to obtain the global  existence  of strong solutions.  When we derive the estimates on the $L^{\infty}(0,T;L^{2}(\mathbb{R}^{2}))$-norm of $\|\nabla u\|_{L^{2}(\mathbb{R}^{2})}$ and $\|\nabla^{2} d\|_{L^{2}(\mathbb{R}^{2})}$. On the one hand, motivated by \cite{LZ,LSZ}, we use material derivatives $\dot{u}\triangleq u_{t}+u\cdot\nabla u$ instead of the usual $u_{t}$, and use some facts on Hardy and BMO spaces (see Lemma \ref{lem2.6}) to bound the key term $\int_{\mathbb{R}^{2}} |P||\nabla u|^{2}\text{d}x$ (see the estimates of $I_{2}$ of \eqref{eq4.5}). On the other hand, the usual $L^{2}(\mathbb{R}^{2}\times (0,T))$-norm of $\nabla d_{t}$ cannot be directly estimated due to the strong coupled term $u\cdot\nabla d$ and the strong nonlinear term $|\nabla d|^{2}d$. Motivated by \cite{LXZ1}, multiplying $\eqref{eq1.1}_{3}$ by $\Delta\nabla d$ instead of the usual $\nabla d_{t}$, and the nonlinear terms $u\cdot\nabla d$ and $|\nabla d|^{2}d$ can be controlled after integration by parts (see \eqref{eq4.09}). Using the structure of the 2D heat flows of harmonic maps, we multiply \eqref{eq3.8} by $\nabla d\Delta |\nabla d|^{2}$ and thus obtain some useful a priori estimates on $\||\nabla d||\nabla^{2}d|\|_{L^{2}(\mathbb{R}^{2})}$ and $\||\nabla d||\Delta\nabla d|\|_{L^{2}(\mathbb{R}^{2})}$ (see \eqref{eq4.15}), which are crucial in deriving the time-independent estimates on both the $L^{\infty}(0,T; L^{2}(\mathbb{R}^{2}))$-norm of $t^{\frac{1}{2}}\varrho^{\frac{1}{2}}\dot{u}$ and the $L^{2}(\mathbb{R}^{2}\times (0,T))$-norm of $t^{\frac{1}{2}} \nabla \dot{u}$ (see \eqref{eq4.11}).
Next, after some careful analysis, we derive the desired $L^{1}(0,T;L^{\infty}(\mathbb{R}^{2}))$ bound of the gradient of the velocity $\nabla u$ (see \eqref{eq4.28}), which in particular implies the bound on the $L^{\infty}(0,T;L^{q}(\mathbb{R}^{2}))$-norm ($q>2$) of the gradient of the density. Moreover, some useful spatial weighted estimates on $\varrho, \nabla d, \nabla^{2}d$ are derived (see Lemma \ref{lem3.6}). With the a priori estimates stated above in hand, we can estimate the higher order derivatives of the solution $(\varrho,u,P,d)$ (see \eqref{eq4.33}) by using the similar arguments as those in \cite{LSZ,LXZ1} to obtain the desired results.

The remaining parts of the present paper are organized as follows.
In Section 2,  we shall give some elementary facts and inequalities which will be needed in later analysis. In Section 3,  we prove the local existence and uniqueness of solutions to the Cauchy problems of \eqref{eq1.1}--\eqref{eq1.2}.
Finally, in Section 4, we first establish some a priori estimates on strong solutions, then the proof of Theorem \ref{thm1.2} is given.
Throughout the paper, we denote by $C$ the positive
constant, which may depend on $a$, $\varepsilon_{0}$ and the initial data, and its value may change from line to line.
We also $C(\alpha,\beta,\cdots)$ to emphasize that the constant $C$ depends on $\alpha,\beta,\cdots$.

\section{Preliminaries}

In this section, we shall give some known results and elementary inequalities which will be used frequently later.  We first list the following local existence theory on bounded balls, where the initial density is strictly away from vacuum, whose proof can be shown by similar arguments as in \cite{HWW12,LJK,WD}.

\begin{lemma}\label{lem2.1}
For $R>0$ and $B_{R}=\{x\in\mathbb{R}^{2}| |x|\leq R\}$, assume that $(\varrho_{0}, u_{0}, d_{0})$ satisfies
\begin{align}\label{eq2.1}
\varrho_{0}\!\in \!H^{1}\!(B_{\!R})\cap L^{\infty}\!(B_{\!R}), \!\quad\! u_{0}\!\in\! H_{0}^{1}\!(B_{\!R})\cap H^{2}\!(B_{\!R}),\!\quad\! d_{0}\!\in\! H^{2}\!(B_{\!R}), \!\quad\! \inf_{x\in B_{\!R}}\! \varrho_{0}(x)\!>\!0,\!\quad\! \operatorname{div} u_{0}=0, \quad\! |d_{0}|=1.
\end{align}
Then there exists a small time $T_{R}>0$ such that the equations \eqref{eq1.1} with the following initial-boundary value conditions
\begin{align}\label{eq2.2}
\begin{cases}
(\varrho, u, d)(x,0)=(\varrho_{0}, u_{0}, d_{0}), \quad x\in B_{R},\\
u(x,t)=0\quad  \frac{ \partial }{\partial\nu }d(x,t)=0,\quad x\in \partial B_{R}, t>0,
\end{cases}
\end{align}
has a unique solution $(\varrho, u, P,d)$ on $B_{R}\times (0,T_{R}]$ satisfying
\begin{align}\label{eq2.3}
\begin{cases}
\varrho\in C((0,T_{R}];H^{2}(B_{R}))\cap L^{\infty}(B_{R}\times(0,T_{R}]), \quad \varrho_{t}\in L^{\infty}(0,T;L^{2}(B_{R})),\\
u\in L^{\infty}(0,T_{R};H^{1}_{0}(B_{R})\cap H^{2}(B_{R})), \quad u_{t}\in L^{2}(0,T_{R};H_{0}^{1}(B_{R})),\\
P\in L^{\infty} (0,T_{R}; H^{1}(B_{R})),\\
d\in L^{\infty}(0,T_{R};H^{3}(B_{R})),\quad d_{t}\in L^{\infty}(0,T_{R};H^{1}(B_{R})).
\end{cases}
\end{align}
Here, $\nu=\frac{x}{R}$ (with $|x|=R$) is the  outer normal vector  on $\partial B_{R}$.
\end{lemma}

Next, for either $\Omega= \mathbb{R}^{2}$ or $\Omega=B_{R}$ with $R\geq 1$, the following weighted $L^{p}$-bounds for elements of the Hilbert space
$\widetilde{D}^{1,2}(\Omega)\triangleq \{v\in H^{1}_{loc}(\Omega) | \nabla u\in L^{2}(\Omega)\}$ can be found in Theorem B.1 of \cite{Lions1}.

\begin{lemma}\label{lem2.2}
For $m\in [2,\infty)$ and $\theta\in (\frac{1+m}{2}, \infty)$, there exists a positive constant $C$ such that for either  $\Omega= \mathbb{R}^{2}$ or $\Omega=B_{R}$ with $R\geq 1$, and for any $v\in \widetilde{D}^{1,2}(\Omega)$,
\begin{align*}
\left(\int_{\Omega}\frac{|v|^{m}}{e+|x|^{2}}(\ln (e+|x|^{2}))^{-\theta}\text{d}x\right)^{\frac{1}{m}}\leq C\|v\|_{L^{2}(B_{1})}+C\|\nabla v\|_{L^{2}(\Omega)}.
\end{align*}
\end{lemma}

A useful consequence of Lemma \ref{lem2.2} is the following weighted bounds for elements of $\widetilde{D}^{1,2}(\Omega)$, which have been proved in \cite{LJLZ,ZL}. It will play a crucial role in our following analysis.

\begin{lemma}\label{lem2.3}
Let $\bar{x}$ be as in Theorem \ref{thm1.2} and $\Omega$ as in Lemma \ref{lem2.2}. Assume that $\varrho\in L^{1}(\Omega)\cap L^{\infty}(\Omega)$
is a non-negative function such that
\begin{align*}
\int_{B_{N_{1}}} \varrho\text{d}x\geq M_{1}, \quad \|\varrho\|_{L^{1}(\Omega)\cap L^{\infty}(\Omega)}\leq M_{2},
\end{align*}
for positive constants $M_{1}$, $M_{2}$, and $N_{1}\geq 1$ with $B_{N_{1}}\subseteq \Omega$. Then for $\varepsilon >0$ and $\eta>0$, there is a positive constant $C$ depending only on $\varepsilon, \eta, M_{1}, M_{2}$, and $ N_{1}$ such that every $v\in \widetilde{D}^{1,2}(\Omega)$ satisfies
\begin{align*}
\|v\bar{x}^{-\eta}\|_{L^{\frac{2+\varepsilon}{\widetilde{\eta}}}(\Omega)}\leq
C\|\varrho^{\frac{1}{2}} v\|_{L^{2}(\Omega)}
+C\|\nabla v\|_{L^{2}(\Omega)}
\end{align*}
with $\widetilde{\eta}=\min\{1,\eta\}$.
\end{lemma}

We shall still need to the following lemma obtained by Liang \cite{ZL}.

\begin{lemma}\label{lem2.4}
Let the assumptions in Lemma \ref{lem2.3} holds. Suppose in addition that $\varrho\bar{x}^{a}\in L^{1}(\Omega)$ with $a>1$. Then there is a constant $C$ depending on \textbf{$M_1,M_2$}, $N_{1}, a, r, \|\varrho\|_{L^{\infty}(\Omega)}$ and $\|\varrho\bar{x}^{a}\|_{L^{1}(\Omega)}$ such that for $\alpha>\frac{2}{r(2+a)}$ and $r\geq 2$
\begin{align}\label{eq2.4}
\|\varrho^{\alpha} v\|_{L^{r}(\Omega)}\leq C(\|\varrho^{\frac{1}{2}} v\|_{L^{2}(\Omega)}+\|\nabla v\|_{L^{2}(\Omega)}).
\end{align}
\end{lemma}

Let us recall the following $L^{p}$-estimate for elliptic systems, whose proof is similar to that of \cite{CCK} (see Lemma 12). It is also a direct result of the combination of the well-known elliptic theory \cite{ADN} and  a standard scaling procedure.

\begin{lemma}\label{lem2.5}
For $p>1$ and $k> 0$, there exists a positive constant $C_{0}$ depending only on $p$ and $k$ such that
\begin{align*}
\|\nabla^{k+2} v\|_{L^{p}(B_{R})} \leq C_{0}\|\Delta v\|_{W^{k,p}(B_{R})}
\end{align*}
 for every $v\in W^{k+2,p}(B_{R})$ satisfying either
 \begin{align}\label{eq2.5}
 v\cdot x =0, \quad \operatorname{rot} v=0,\quad \text{ on }\partial B_{R}
 \end{align}
 or
 \begin{align*}
 n\cdot\nabla v=0, \quad \text{ on }\partial B_{R}.
 \end{align*}
 \end{lemma}

 Let $\mathcal{H}^{1}(\mathbb{R}^{2})$ and $BMO(\mathbb{R}^{2})$ stand for the usual Hardy and BMO spaces (see \cite{Stein}), we end this section by   the following useful lemma, whose proofs can be found in \cite{CLMS93}.

 \begin{lemma}\label{lem2.6}
 (i)\ There is a positive constant $C$ such that
 \begin{align*}
 \|E\cdot B\|_{\mathcal{H}^{1}(\mathbb{R}^{2})}\leq C\|E\|_{L^{2}(\mathbb{R}^{2})} \|B\|_{L^{2}(\mathbb{R}^{2})},
 \end{align*}
 for all $E\in L^{2}(\mathbb{R}^{2})$ and $B\in L^{2}(\mathbb{R}^{2})$ with
 \begin{align*}
 \operatorname{div} E =0, \quad \nabla^{\bot}\cdot B=0\quad \text{in } \mathcal{D}'(\mathbb{R}^{2}).
 \end{align*}
 (ii)\ There is a positive constant $C$ such that  for all $v\in \widetilde{D}^{1,2}(\mathbb{R}^{2})$, it holds
 \begin{align*}
 \|v\|_{BMO(\mathbb{R}^{2})}\leq C \|\nabla v\|_{L^{2}(\mathbb{R}^{2})}.
 \end{align*}
 \end{lemma}

\section{Local existence and uniqueness of solutions}

In this section, we shall prove the following local existence and uniqueness of strong solutions to the Cauchy problem of \eqref{eq1.1}--\eqref{eq1.2}.

\begin{theorem}\label{thm3.1}
Assume that $(\varrho_{0}, u_{0}, d_{0})$ satisfies  \eqref{eq1.5}--\eqref{eq1.7}. Then there exist a small positive time $T_{0}>0$ and a unique strong solution $(\varrho, u, P, d)$ to the Cauchy problem of system \eqref{eq1.1}--\eqref{eq1.2} in $\mathbb{R}^{2}\times (0,T_{0}]$ satisfying \eqref{eq1.8} and \eqref{eq1.9} with $T=T_{0}$.
\end{theorem}

\begin{remark}
We remark that in Theorem \ref{thm3.1}, we do not need the condition \eqref{eq1.3}.  From \cite{LJLZ,LXZ}, it is easy to see that if we replace $\bar{x}$ in Theorem \ref{thm1.2} by
\begin{align*}
\bar{x}\triangleq   (e+|x|^{2})^{\frac{1}{2}} \ln^{1+\eta_{0}} (e+|x|^{2}),
\end{align*}
with $\eta_{0}>0$, then the  result of Theorem \ref{thm3.1} still holds.
\end{remark}

\subsection{A priori estimates}
Throughout this subsection, for $p\in[1,\infty]$ and $k\geq 0$, we denote
\begin{align*}
\int f\text{d}x=\int_{B_{R}} f\text{d}x, \quad L^{p}=L^{p}(B_{R}),\quad W^{1,p}=W^{1,p}(B_{R}), \quad H^{k}=W^{k,2},
\end{align*}
for simplicity. Denote by $\|\cdot\|_{X}$
 the norm of the $X(B_{R})$-functions. Moreover, for $R>4N_{0}\geq 4$, in addition \eqref{eq2.2}, assume that $(\varrho_{0},u_{0},d_{0})$ satisfies
\begin{align}\label{eq3.1}
\frac{1}{2}\leq \int_{B_{N_{0}}}\varrho_{0}(x)\text{d}x\leq \int_{B_{R}} \varrho_{0}(x)\text{d}x\leq 1.
\end{align}
Then, from Lemma \ref{lem2.1}, we find that there exist some $T_{R}>0$ such that the initial-boundary value problem \eqref{eq1.1} and \eqref{eq2.2} has a unique strong solution $(\varrho, u, P, d)$ on $B_{R}\times [0,T_{R}]$ satisfying \eqref{eq2.3}.

For $\bar{x},  a$ and $q$ as in Theorem \ref{thm1.2}, the main goal of this subsection is to derive the following key a priori estimates on $\Phi(t)$ defined by
\begin{align*}
\Phi(t)\triangleq 1+\|\varrho^{\frac{1}{2}} u\|_{L^{2}}^{2}+\|\nabla u\|_{L^{2}}^{2}+\|\nabla d \bar{x}^{\frac{a}{2}}\|_{L^{2}}^{2}
+\|\nabla^{2} d\|_{L^{2}}^{2}+\|\varrho\bar{x}^{a}\|_{L^{1}\cap H^{1}\cap W^{1,q}},
\end{align*}
which are needed for the proof of Theorem \ref{thm3.1}.

\begin{proposition}\label{prop3.1}
Assume that $(\varrho_{0}, u_{0}, d_{0})$ satisfies \eqref{eq2.1} and \eqref{eq3.1}. Let $(\varrho, u, P, d)$ be the solution to the initial-boundary value problem \eqref{eq1.1} and \eqref{eq2.2} on $B_{R}\times [0,T_{R}]$ obtained by Lemma \ref{lem2.1}. Then there exist positive constants $T_{0}$ and $M$ both depending only on $a, q, N_{0}$, and $E_{0}$ such that
\begin{align}\label{eq3.2}
&\sup_{0\leq t\leq T_{0}}\left(\Phi(t)+t\|\varrho^{\frac{1}{2}} u_{t}\|_{L^{2}}^{2}+ t\| d_{t}\|_{H^{1}}^{2}
+t \|\nabla^{2} u\|_{L^{2}}^{2}+t\|\nabla P\|_{L^{2}}^{2}+t\|\nabla^{3} d\|_{L^{2}}^{2}\right)\nonumber\\
&+\int_{0}^{T_{0}}\left( \|\varrho^{\frac{1}{2}} u_{t}\|_{L^{2}}^{2}+\|\nabla^{2} u\|_{L^{2}}^{2}+ \|\nabla d_{t}\|_{H^{1}}^{2}
+\|\nabla^{3} d\|_{L^{2}}^{2} +\|\nabla^{2} d \bar{x}^{\frac{a}{2}}\|_{L^{2}}^{2} \right)\text{d}t\nonumber\\
&+\int_{0}^{T_{0}} \left( \|\nabla^{2} u\|_{L^{q}}^{\frac{q+1}{q}}+\|\nabla P\|_{L^{q}}^{\frac{1+q}{q}}+t\|\nabla^{2} u\|_{L^{q}}^{2}
+t\|\nabla P\|_{L^{q}}^{2}\right)\text{d}t\nonumber\\
&+ \int_{0}^{T_{0}} t\left( \|\nabla u_{t}\|_{L^{2}}^{2}+\|\nabla d_{t}\|_{H^{1}}^{2}\right)\text{d}t\leq M,
\end{align}
where
\begin{align*}
E_{0}\triangleq 1+\|\varrho_{0}^{\frac{1}{2}} u_{0}\|_{L^{2}}+\|\nabla u_{0}\|_{L^{2}}+\|\nabla d_{0} \bar{x}^{\frac{a}{2}}\|_{L^{2}}
+\|\nabla^{2} d_{0}\|_{L^{2}}+\|\varrho_{0}\bar{x}^{a}\|_{L^{1}\cap H^{1}\cap W^{1,p}}.
\end{align*}
\end{proposition}

To prove Proposition \ref{prop3.1}, whose proof will be postponed to the end of this subsection, we need to give the following useful estimates.

\begin{lemma}\label{lem3.2}
Under the assumptions of Proposition \ref{prop3.1}, let $(\varrho, u, P, d)$ be a smooth solution to the initial boundary value problem \eqref{eq1.1} and \eqref{eq2.2}. Then we have for all $t\geq 0$,
\begin{align}
        \label{eq3.3}
&\sup_{0\leq \tau\leq t} \left(
\|\varrho^{\frac{1}{2}} u\|_{L^{2}}^{2}+\|\nabla d\|_{L^{2}}^{2}\right)
+\int_{0}^{t} \|\nabla u\|_{L^{2}}^{2}+\|\Delta d+|\nabla d|^{2} d\|_{L^{2}}^{2}\text{d}\tau
\leq \|\varrho_{0}^{\frac{1}{2}} u_{0}\|_{L^{2}}^{2}+\|\nabla d_{0}\|_{L^{2}}^{2},\\
 \label{eq3.4}
&\qquad\quad\sup_{0\leq \tau\leq t} \left(
\|\varrho\|_{L^{1}\cap L^{\infty}}+\|\nabla d\|_{L^{2}}^{2}\right)
+\int_{0}^{t} \|\nabla^{2} d\|_{L^{2}}^{2}\text{d}\tau
\leq C\exp\left\{C\int_{0}^{t}\Phi(\tau)\text{d}\tau\right\},
\end{align}
and
\begin{align}\label{eq3.5}
\sup_{0\leq \tau\leq t} \left(\|\nabla d \bar{x}^{\frac{a}{2}}\|_{L^{2}}^{2} \right)
+\int_{0}^{t} \|\nabla^{2} d \bar{x}^{\frac{a}{2}}\|_{L^{2}}^{2}\text{d}\tau\leq C\exp\left\{\int_{0}^{t}\Phi(\tau)\text{d}\tau\right\}.
\end{align}
\end{lemma}

\begin{proof}\label{proof of lem3.2}
\eqref{eq3.3} is the standard energy inequality, see, e.g., \cite{LLZ,LLW}. To prove \eqref{eq3.4}, we first notice that from the continuity equation $\eqref{eq1.1}_{1}$ and the divergence free condition $\eqref{eq1.1}_{4}$, it is easy to see (cf, \cite{Lions1}) that
\begin{align}\label{eq3.6}
\|\varrho(t)\|_{L^{p}}=\|\varrho_{0}\|_{L^{p}}, \quad \text{ for all } p\in [1,\infty]\text{ and } t\geq 0.
\end{align}
Next,  multiplying  $\eqref{eq1.1}_{3}$ by $-\Delta d$,   and integrating
over $B_{R}$, it follows that
\begin{align}\label{eq3.7}
\frac{1}{2}\frac{d}{dt}  \|\nabla d\|_{L^{2}}^{2}
 +&\frac{1}{2}\|\nabla^{2} d\|_{L^{2}}^{2}\leq \int |\nabla u||\nabla d||\nabla d|\text{d}x+\int |\nabla d|^{2}|\nabla^{2} d|\text{d}x
 +C\|\nabla d\|_{L^{2}}^{2}\nonumber\\
 \leq& \|\nabla u\|_{L^{2}} \|\nabla d\|_{L^{4}}^{2} +\|\nabla d\|_{L^{4}}^{2}\|\nabla^{2} d\|_{L^{2}}+C\|\nabla d\|_{L^{2}}^{2}\nonumber\\
 \leq& \frac{1}{4}\|\nabla^{2} d\|_{L^{2}}^{2}+C(1+\|\nabla u\|_{L^{2}}^{2}+\|\nabla d\|_{L^{2}}^{2})\|\nabla d\|_{L^{2}}^{2}\nonumber\\
 \leq& \frac{1}{4}\|\nabla^{2} d\|_{L^{2}}^{2}+C\Phi(t) \|\nabla d\|_{L^{2}}^{2},
\end{align}
where we have used the condition $|d|=1$, and the following inequalities
\begin{align*}
&\int |\Delta d|^{2}\text{d}x=\sum_{i,j=1}^{2} \int \partial_{i}^{2}d\cdot\partial_{j}^{2}d\text{d}x
=-\sum_{i,j=1}^{2}\int\partial_{i} d\cdot\partial_{i}\partial_{j}^{2}d\text{d}x\quad (\text{notice that } \frac{\partial}{\partial\nu}d=0 \text{ on } \partial B_{R}.)\nonumber\\
=&-\sum_{i,j=1}^{2}\int_{\partial B_{R}} \partial_{i}d\cdot\partial_{j}\partial_{i}d \nu_{j}\text{d}S+\sum_{i,j=1}^{2}\int |\partial_{i}\partial_{j} d|^{2}\text{d}x\quad \quad (\nu \text{ is defined in Lemma \ref{lem2.1}.})\nonumber\\
= &\|\nabla^{2} d\|_{L^{2}}^{2}+\sum_{i,j=1}^{2}\int_{\partial B_{R}} \partial_{i}d\cdot\partial_{j} d\partial_{i}\nu_{j}\text{d}S
= \|\nabla^{2} d\|_{L^{2}}^{2}+\frac{1}{R}\sum_{i,j=1}^{2}\int_{\partial B_{R}} \partial_{i}d\cdot\partial_{j} d\delta_{ij}\text{d}S\nonumber\\
\geq &\|\nabla^{2} d\|_{L^{2}}^{2}-\|\nabla d\|_{L^{2}(\partial B_{R})}^{2}
\geq \|\nabla^{2} d\|_{L^{2}}^{2}-C\|\nabla d\|_{{H}^{\frac{1}{2}}( B_{R})}^{2} \quad (\text{see, e.g., \cite{OAL} for trace embedding}  )\nonumber\\
\geq & \|\nabla^{2} d\|_{L^{2}}^{2}-C\|\nabla d\|_{L^{2}}\|\nabla^{2}d\|_{L^{2}}
\geq \frac{1}{2}\|\nabla^{2} d\|_{L^{2}}^{2}-C\|\nabla d\|_{L^{2}}^{2}.
\end{align*}
Here $\delta_{ij}=1$ if $i=j$ (or $=0$ if $i\neq j$). The estimate \eqref{eq3.7}  together with Gronwall's inequality ensures
\begin{align*}
\sup_{0\leq \tau\leq t} \left(\|\nabla d\|_{L^{2}}^{2}\right)
+\int_{0}^{t} \|\nabla^{2} d\|_{L^{2}}^{2}\text{d}\tau\leq \exp\left\{C\int_{0}^{t}\Phi(\tau)\text{d}\tau\right\}.
\end{align*}
This combined with \eqref{eq3.6} leads to \eqref{eq3.4}.

In what follows, we are in a position to prove \eqref{eq3.5}. By applying $\nabla $ on $\eqref{eq1.1}_{3}$, we have
\begin{align}\label{eq3.8}
\nabla d_{t}-\Delta\nabla d=-\nabla(u\cdot \nabla d)+\nabla(|\nabla d|^{2}d).
\end{align}
Multiplying \eqref{eq3.8} with $\nabla d \bar{x}^{a}$ and integrating by parts over $B_{R}$ yield
\begin{align}\label{eq3.9}
\frac{1}{2}\frac{d}{dt} \|\nabla d \bar{x}^{\frac{a}{2}}\|_{L^{2}}^{2} +\|\nabla^{2} d \bar{x}^{\frac{a}{2}}\|_{L^{2}}^{2}
\leq& C\!\!\int \!\!|\nabla d| |\nabla^{2} d| \nabla \bar{x}^{a} \text{d}x\! +\!\!\int \! \! |\nabla u||\nabla d|^{2} \bar{x}^{a}
\text{d}x +\!\!\int\! |u||\nabla d|^{2} \nabla \bar{x}^{a}\text{d}x \nonumber\\
&+\!\!\int\! \! |\nabla d|^{2} |\nabla^{2} d| \bar{x}^{a}\text{d}x
\!+\!\!\int\! |\nabla d|^{3} \nabla\bar{x}^{a}\text{d}x\!+\!C\|\nabla d \bar{x}^{\frac{a}{2}}\|_{L^{2}}^{2}\nonumber\\
\triangleq& I_{1}+I_{2}+I_{3}+I_{4}+I_{5}+\!C\|\nabla d \bar{x}^{\frac{a}{2}}\|_{L^{2}}^{2},
\end{align}
where we have used the following inequalities
\begin{align*}
&-\int \nabla\Delta d\cdot\nabla d\bar{x}^{a}\text{d}x=-\sum_{i,j=1}^{2}\int \partial_{i}\partial_{j}^{2}d\cdot\partial_{i}d\bar{x}^{a}\text{d}x\nonumber\\
=& -\sum_{i,j=1}^{2}\int_{\partial B_{R}} \partial_{i}\partial_{j}d\cdot\partial_{i} d\bar{x}^{a} \nu_{j}\text{d}S+ \sum_{i,j=1}^{2} \int |\partial_{i}\partial_{j} d|^{2} \bar{x}^{a}\text{d}x +\sum_{i,j=1}^{2} \int \partial_{i}\partial_{j} d\cdot \partial_{i}d  \partial_{j}\bar{x}^{a}\text{d}x\nonumber\\
=&\|\nabla^{2}d\bar{x}^{\frac{a}{2}}\|_{L^{2}}^{2}+\sum_{i,j=1}^{2} \int \partial_{i}\partial_{j} d\cdot \partial_{i}d  \partial_{j}\bar{x}^{a}\text{d}x
+\sum_{i,j=1}^{2}\int_{\partial B_{R}} \partial_{i} d\cdot\partial_{i} d\bar{x}^{a} \partial_{i}\nu_{j}\text{d}S\nonumber\\
=&\|\nabla^{2}d\bar{x}^{\frac{a}{2}}\|_{L^{2}}^{2}+\sum_{i,j=1}^{2} \int \partial_{i}\partial_{j} d \cdot\partial_{i}d  \partial_{j}\bar{x}^{a}\text{d}x
+\frac{1}{R}\sum_{i,j=1}^{2}\int_{\partial B_{R}} \partial_{i} d\cdot\partial_{i} d\bar{x}^{a} \delta_{ij}\text{d}S\nonumber\\
\geq& \|\nabla^{2}d\bar{x}^{\frac{a}{2}}\|_{L^{2}}^{2} -C\int |\nabla^{2} d| |\nabla d|\nabla \bar{x}^{a}\text{d}x -\|\nabla d\bar{x}^{a}\|_{L^{2}(\partial B_{R})}^{2}\nonumber\\
\geq& \|\nabla^{2}d\bar{x}^{\frac{a}{2}}\|_{L^{2}}^{2} -C\int |\nabla^{2} d| |\nabla d|\nabla \bar{x}^{a}\text{d}x -C\|\nabla d\bar{x}^{a}\|_{{H}^{\frac{1}{2}}( B_{R})}^{2}\nonumber\\
\geq& \frac{1}{2}\|\nabla^{2}d\bar{x}^{\frac{a}{2}}\|_{L^{2}}^{2} -C\int |\nabla^{2} d| |\nabla d|\nabla \bar{x}^{a}\text{d}x -C\|\nabla d\bar{x}^{a}\|_{L^{2}}^{2}.
\end{align*}
We can bound each term on the right-hand side of \eqref{eq3.9} as follows
\begin{align*}
I_{1}\leq& \int |\nabla d| |\nabla^{2} d| \bar{x}^{a}\text{d}x\leq \frac{1}{10} \|\nabla^{2} d \bar{x}^{\frac{a}{2}}\|_{L^{2}}^{2}
+C\|\nabla d \bar{x}^{\frac{a}{2}}\|_{L^{2}}^{2};\\
I_{2}\leq& \|\nabla u\|_{L^{2}} \|\nabla d \bar{x}^{\frac{a}{2}}\|_{L^{4}}^{2}
\leq \|\nabla u\|_{L^{2}}\|\nabla d \bar{x}^{\frac{a}{2}}\|_{L^{2}}(\|\nabla d \bar{x}^{\frac{a}{2}}\|_{L^{2}}+\|\nabla (\nabla d \bar{x}^{\frac{a}{2}})\|_{L^{2}})\\
\leq& \|\nabla u\|_{L^{2}} \|\nabla d \bar{x}^{\frac{a}{2}}\|_{L^{2}} (\|\nabla d \bar{x}^{\frac{a}{2}}\|_{L^{2}}+\|\nabla^{2}d \bar{x}^{\frac{a}{2}}\|_{L^{2}}+\|\nabla d \nabla \bar{x}^{\frac{a}{2}}\|_{L^{2}})\\
\leq & \frac{1}{10} \|\nabla^{2} d \bar{x}^{\frac{a}{2}}\|_{L^{2}}^{2}+C(1+\|\nabla u\|_{L^{2}}^{2})
 \|\nabla d\bar{x}^{\frac{a}{2}}\|_{L^{2}}^{2}
\leq  \frac{1}{10} \|\nabla^{2} d \bar{x}^{\frac{a}{2}}\|_{L^{2}}^{2}+C\Phi(t) \|\nabla d\bar{x}^{\frac{a}{2}}\|_{L^{2}}^{2};\\
I_{3}\leq&\int |u||\nabla d|\bar{x}^{a-\frac{3}{4}}\text{d}x\leq \|\nabla d \bar{x}^{\frac{a}{2}}\|_{L^{4}} \|\nabla d \bar{x}^{\frac{a}{2}}\|_{L^{2}}\|u\bar{x}^{-\frac{3}{4}}\|_{L^{4}}\\
\leq&  C\|\nabla d \bar{x}^{\frac{a}{2}}\|_{L^{4}}^{2} + C(\|\varrho^{\frac{1}{2}} u\|_{L^{2}}^{2}+\|\nabla u\|_{L^{2}}^{2})\|\nabla d \bar{x}^{\frac{a}{2}}\|_{L^{2}}^{2}
\leq  \frac{1}{10} \|\nabla^{2} d \bar{x}^{\frac{a}{2}}\|_{L^{2}}^{2}+C\Phi(t) \|\nabla d\bar{x}^{\frac{a}{2}}\|_{L^{2}}^{2};\\
I_{4}\leq&\|\nabla^{2} d \bar{x}^{\frac{a}{2}}\|_{L^{2}}
\|\nabla d\bar{x}^{\frac{a}{2}}\|_{L^{4}} \|\nabla d\|_{L^{4}}\leq \frac{1}{20}\|\nabla^{2} d \bar{x}^{\frac{a}{2}}\|_{L^{2}}^{2}
+ \|\nabla d\|_{L^{4}}^{2}\|\nabla d\bar{x}^{\frac{a}{2}}\|_{L^{4}}^{2}\\
\leq& \frac{1}{20}\|\nabla^{2} d \bar{x}^{\frac{a}{2}}\|_{L^{2}}^{2}
+ (\|\nabla d\|_{L^{2}}^{2}+\|\nabla^{2} d\|_{L^{2}}^{2})\|\nabla d\bar{x}^{\frac{a}{2}}\|_{L^{2}}(\|\nabla^{2}d \bar{x}^{\frac{a}{2}}\|_{L^{2}}+\|\nabla d \nabla \bar{x}^{\frac{a}{2}}\|_{L^{2}})\\
\leq& \frac{1}{10}\|\nabla^{2} d \bar{x}^{\frac{a}{2}}\|_{L^{2}}^{2}
+ C\Phi(t) \|\nabla d\bar{x}^{\frac{a}{2}}\|_{L^{2}}^{2};\\
I_{5}\leq&\int |\nabla d|^{3}\bar{x}^{a}\text{d}x\leq \|\nabla d\|_{L^{2}}\|\nabla d \bar{x}^{\frac{a}{2}}\|_{L^{4}}^{2}
\leq \frac{1}{10}\|\nabla^{2} d \bar{x}^{\frac{a}{2}}\|_{L^{2}}^{2}
+ C\Phi(t) \|\nabla d\bar{x}^{\frac{a}{2}}\|_{L^{2}}^{2},
\end{align*}
due to the energy inequality \eqref{eq3.3}, H\"{o}lder's and Gagliardro-Nirenberg inequalities. Hence, inserting the estimates of $I_{i} (i=1,2,\cdots,5)$ into \eqref{eq3.9}, we
obtain \eqref{eq3.5} after by using Gronwall's inequality. This completes the proof of Lemma \ref{lem3.2}.
\end{proof}

\begin{lemma}\label{lem3.3}
Under the assumptions of Proposition \ref{prop3.1}, let $(\varrho, u, P, d)$ be a smooth solution to the initial boundary value problem \eqref{eq1.1} and \eqref{eq2.2}. Then there exists a $T_{1}=T({N_{0},E_{0}})>0$ such that for all $t\in (0,T_{1}]$,
\begin{align}\label{eq3.10}
&\sup_{0\leq \tau\leq t} \left( \|\varrho \bar{x}^{a}\|_{L^{1}}+\|\nabla u\|_{L^{2}}^{2}+\|\nabla^{2} d\|_{L^{2}}^{2} \right)
\nonumber\\
&\quad \quad +\int_{0}^{t}\left(\|\varrho^{\frac{1}{2}} u_{\tau}\|_{L^{2}}^{2}+\|\nabla d_{\tau} \|_{L^{2}}^{2}
+\|\Delta\nabla d\|_{L^{2}}^{2}\right)\text{d}\tau\leq C\exp\left\{C\int_{0}^{t}\Phi(\tau)\text{d}\tau\right\},\\
  \label{eq3.11}
&\int_{0}^{t}\|\nabla^{2} u\|_{L^{2}}^{2}\text{d}\tau\leq C\exp\left\{C\int_{0}^{t}\Phi(\tau)\text{d}\tau\right\}.
\end{align}
\end{lemma}

\begin{proof}\label{proof of lem3.3}
First, for $N>1$, let $\varphi_{N}\in C_{0}^{\infty}(B_{N})$ satisfy
\begin{align}\label{eq3.12}
0\leq \varphi_{N}\leq 1, \quad \varphi_{N}(x)=
\begin{cases}
1,\quad \text{ if }|x|\leq \frac{N}{2},\\
0,\quad \text{ if }|x|\geq N,
\end{cases}
 \quad |\nabla^{k}\varphi_{N}|\leq CN^{-k}\quad \text{ for } k\in \mathbb{N}.
\end{align}
It follows from $\eqref{eq1.1}_{1}$ and the energy inequality \eqref{eq3.3} that
\begin{align*}
\frac{d}{dt}\int \varrho \varphi_{2N_{0}}\text{d}x=\int \varrho  u\cdot\nabla \varphi_{2N_{0}}\text{d}x
\geq -CN_{0}^{-1}\left(\int \varrho\text{d}x\right)^{\frac{1}{2}}\left(\int \varrho |u|^{2}\text{d}x\right)^{\frac{1}{2}}
\geq - \widetilde{C}(N_{0},E_{0}).
\end{align*}
Integrating the above inequality and  using \eqref{eq3.1}, it follows that
\begin{align}\label{eq3.13}
\inf_{0\leq t\leq T_{1}} \int_{B_{2N_{0}}} \varrho \text{d}x\geq \inf_{0\leq t\leq T_{1}} \int \varrho \varphi_{2N_{0}} \text{d}x
\geq \int \varrho_{0}\varphi_{2N_{0}}\text{d}x-\widetilde{C} T_{1}\geq \frac{1}{4},
\end{align}
where $T_{1}\triangleq \min \{1, \frac{1}{4\widetilde{C}}\}$. From now on, we will always assume that $t\leq T_{1}$ in this subsection. The combination of \eqref{eq3.13}, \eqref{eq3.3} and Lemma \ref{lem2.3} ensures that for $\varepsilon>0$, $\eta>0$, and every $v\in \widetilde{D}^{1,2}(B_{R})$ satisfies
\begin{align}\label{eq3.14}
\|v \bar{x}^{-\eta}\|_{L^{\frac{2+\varepsilon}{\widetilde{\eta}}}}\leq C(\varepsilon,\eta) \|\varrho^{\frac{1}{2}} v\|_{L^{2}} +
C(\varepsilon,\eta) \|\nabla v\|_{L^{2}},
\end{align}
where $\widetilde{\eta}\triangleq \min\{1,\eta\}$. Using the above estimate \eqref{eq3.14}, after multiplying $\eqref{eq1.1}_{1}$ by $\bar{x}^{a}$ and integrating by parts over $B_{R}$, we have
\begin{align*}
\frac{d}{dt} \|\varrho \bar{x}^{a}\|_{L^{1}}\leq& C\int \varrho |u| \bar{x}^{a-1} \ln^{2}(e+|x|^{2})\text{d}x
\leq C\|\varrho\bar{x}^{a-1+\frac{4}{4+a}}\|_{L^{\frac{a+4}{a+3}}} \|u\bar{x}^{-\frac{4}{4+a}}\|_{L^{a+4}}\nonumber\\
\leq & C\|\varrho\|_{L^{\infty}}^{\frac{1}{4+a}} \|\varrho\bar{x}^{a}\|_{L^{1}}^{\frac{3+a}{4+a}}
(\|\varrho^{\frac{1}{2}} u\|_{L^{2}} +\|\nabla u\|_{L^{2}})\nonumber\\
\leq& C(1+\|\varrho\bar{x}^{a}\|_{L^{1}}) (1+\|\nabla u\|_{L^{2}}^{2})
\end{align*}
due to \eqref{eq3.3} and \eqref{eq3.6}. This combined with Gronwall's inequality and \eqref{eq3.3} leads to
\begin{align}\label{eq3.15}
\sup_{0\leq \tau\leq t} \|\varrho\bar{x}^{a}\|_{L^{1}} \leq C\exp\left\{C\int_{0}^{t}(1+\|\nabla u\|_{L^{2}}^{2})\text{d}\tau\right\}\leq C.
\end{align}

 Now, multiplying $\eqref{eq1.1}_{2}$ by $u_{t}$ and integrating by parts over $B_{R}$, one has
 \begin{align}\label{eq3.16}
 \frac{1}{2}\frac{d}{dt} \|\nabla u\|_{L^{2}}^{2}+\|\varrho^{\frac{1}{2}} u_{t}\|_{L^{2}}^{2}
 \leq C\int\varrho |u||\nabla u||u_{t}|\text{d}x+\int (\nabla d\odot \nabla d) \cdot \nabla u_{t}\text{d}x.
 \end{align}
 By using  H\"{o}lder's and Gagliardo-Nirenberg inequalities, and \eqref{eq2.4} ensure that
 \begin{align}\label{eq3.17}
 \int\varrho |u||\nabla u||u_{t}|\text{d}x\leq& \frac{1}{2}\|\varrho^{\frac{1}{2}} u_{t}\|_{L^{2}}^{2}
 +C\int\varrho |u|^{2} |\nabla u|^{2}\text{d}x
 \leq \frac{1}{2}\|\varrho^{\frac{1}{2}} u_{t}\|_{L^{2}}^{2}
 +C\|\varrho^{\frac{1}{2}} u\|_{L^{8}}^{2} \|\nabla u\|_{L^{\frac{8}{3}}}^{2}\nonumber\\
 \leq &  \frac{1}{2}\|\varrho^{\frac{1}{2}} u_{t}\|_{L^{2}}^{2}
 +C\|\varrho^{\frac{1}{2}} u\|_{L^{8}}^{2} \|\nabla u\|_{L^{2}}^{\frac{3}{2}}\|\nabla u\|_{H^{1}}^{\frac{1}{2}}\nonumber\\
 \leq &  \frac{1}{2}\|\varrho^{\frac{1}{2}} u_{t}\|_{L^{2}}^{2}+\varepsilon \|\nabla^{2} u\|_{L^{2}}^{2}
 +C(\|\varrho^{\frac{1}{2}} u\|_{L^{2}}^{2}+\|\nabla u\|_{L{2}}^{2})^{\frac{4}{3}} \|\nabla u\|_{L^{2}}^{2}\nonumber\\
 \leq &  \frac{1}{2}\|\varrho^{\frac{1}{2}} u_{t}\|_{L^{2}}^{2}+\varepsilon \|\nabla^{2} u\|_{L^{2}}^{2}
 +C\Phi^{\beta}(t),
 \end{align}
where (and in what follows) we have used $\beta>1$ to denote a generic constant, which may be different from line to line.
For the second term on the right-hand  side of \eqref{eq3.16}, we have
\begin{align}\label{eq3.18}
\int (\nabla d\odot \nabla d) \cdot \nabla u_{t}\text{d}x=&\frac{d}{dt}\int (\nabla d\odot \nabla d) \cdot \nabla u\text{d}x
-\int \nabla d_{t}\odot \nabla d\cdot \nabla u\text{d}x- \int \nabla d\odot \nabla d_{t}\cdot \nabla u\text{d}x\nonumber\\
\leq&\frac{d}{dt}\int (\nabla d\odot \nabla d) \cdot \nabla u\text{d}x
+\frac{1}{4} \|\nabla d_{t}\|_{L^{2}}+C\| \nabla d\|_{L^{4}}^{2}\| \nabla u\|_{L^{4}}^{2}\nonumber\\
\leq&\frac{d}{dt}\int (\nabla d\odot \nabla d) \cdot \nabla u\text{d}x
+\frac{1}{4} \|\nabla d_{t}\|_{L^{2}}^{2}+C\| \nabla d\|_{L^{2}}\|\nabla d\|_{H^{1}}\| \nabla u\|_{L^{2}}\|\nabla u\|_{H^{1}}\nonumber\\
\leq&\frac{d}{dt}\int (\nabla d\odot \nabla d) \cdot \nabla u\text{d}x
+\frac{1}{4} \|\nabla d_{t}\|_{L^{2}}^{2}+\varepsilon \|\nabla^{2} u\|_{L^{2}}^{2}+C\Phi^{\beta}(t).
\end{align}
Inserting \eqref{eq3.17} and \eqref{eq3.18} into \eqref{eq3.16}, it follows that
\begin{align}\label{eq3.19}
\frac{d}{dt} B(t) +\|\varrho^{\frac{1}{2}} u_{t}\|_{L^{2}}^{2}\leq \varepsilon \|\nabla^{2} u\|_{L^{2}}^{2}+\frac{1}{4} \|\nabla d_{t}\|_{L^{2}}^{2}
+C\Phi^{\beta}(t),
\end{align}
where
\begin{align*}
B(t)\triangleq \|\nabla u\|_{L^{2}}^{2}-\int (\nabla d\odot \nabla d) \cdot \nabla u\text{d}x
\end{align*}
satisfies
\begin{align*}
\frac{1}{2}\|\nabla u\|_{L^{2}}^{2}-C_{1}\|\nabla d\|_{H^{1}}^{2}\leq B(t)\leq \frac{3}{2}\|\nabla u\|_{L^{2}}^{2}+C\|\nabla d\|_{H^{1}}^{2}.
\end{align*}

Next, from \eqref{eq3.8}, it is easy to obtain that
\begin{align}\label{eq3.20}
\frac{d}{dt}(\|\nabla^{2}\! d\|_{L^{2}}^{2} &\!+\!\frac{1}{R}\!\|\nabla d\|_{L^{\!2}(\partial B_{R})}^{2})\!+\!\|\nabla d_{t}\|_{\!L^{\!2}}^{2}\!+\!\|\Delta \nabla d\|_{\!L^{\!2}}^{2}\!=\!\!\int\! |\nabla d_{t}\!-\!\Delta\nabla d|^{2}\text{d}x\!\leq \!C\!\!\int\! |\nabla(u\cdot \nabla d)|^{2}\!+\!|\nabla(|\nabla d|^{2}d)|^{2}\text{d}x\nonumber\\
\leq& C\|\nabla u\|_{L^{2}}^{2} \|\nabla d\|_{L^{\infty}}^{2} +C\||u||\nabla^{2} d|\|_{L^{2}}^{2} +C \||\nabla d||\nabla^{2}d|\|_{L^{2}}^{2}
+C\|\nabla d\|_{L^{6}}^{6}\nonumber\\
\leq & C\!\left(\|\nabla u\|_{\!L^{\!2}}^{2} \|\nabla d\|_{\!L^{\!2}}(\|\nabla d\|_{\!L^{\!2}}\!+\!\|\nabla^{3} d\|_{\!L^{\!2}})\!+\!\|u\bar{x}^{-\frac{a}{4}}\|_{\!L^{\!8}}^{2} \|\nabla^{2}d\bar{x}^{\frac{a}{2}}\|_{\!L^{\!2}}\|\nabla^{2} d\|_{\!L^{\!4}}
\!+\!\|\nabla d\|_{\!L^{\!4}}^{2} \|\nabla^{2} d\|_{\!L^{\!4}}^{2}\! +\!\|\nabla d\|_{\!L^{\!6}}^{6}\right)\nonumber\\
\leq &C(\|\nabla u\|_{L^{2}}^{2} \|\nabla d\|_{L^{2}}(\|\nabla d\|_{L^{2}}+\|\nabla^{3}  d\|_{L^{2}})+\|u\bar{x}^{-\frac{a}{4}}\|_{L^{8}}^{2} \|\nabla^{2}d\bar{x}^{\frac{a}{2}}\|_{L^{2}}(\|\nabla^{2} d\|_{L^{2}}+\|\nabla^{3}  d\|_{L^{2}})\nonumber\\
&
+\|\nabla d\|_{L^{2}} \|\nabla^{2} d\|_{L^{2}}^{2}(\|\nabla^{2} d\|_{L^{2}}+\|\nabla^{3}  d\|_{L^{2}}) +\|\nabla d\|_{L^{2}}^{2}\|\nabla^{2}d\|_{L^{2}}^{4})\nonumber\\
\leq &\frac{1}{4C_{0}} \|\nabla^{3}  d\|_{L^{2}}^{2} +C \|\nabla^{2} d\bar{x}^{\frac{a}{2}}\|_{L^{2}}^{2}+C\Phi^{\beta}(t),
\end{align}
where we have used  H\"{o}lder's and  Gagliardo-Nirenberg inequalities, \eqref{eq3.14},  and the following equalities
\begin{align*}
-2\int \nabla d_{t}\Delta\nabla d\text{d}x=&-2\sum_{i,j=1}^{2} \int_{\partial B_{R}}\partial_{i}d_{t}\cdot \partial_{i}\partial_{j}d\nu_{j}\text{d}S
+\frac{d}{dt} \|\nabla^{2}d\|_{L^{2}}^{2}\nonumber\\
=&2\sum_{i,j=1}^{2} \int_{\partial B_{R}}\partial_{i}d_{t}\cdot\partial_{j}d \partial_{i}\nu_{j}\text{d}S
+\frac{d}{dt} \|\nabla^{2}d\|_{L^{2}}^{2}\nonumber\\
=&\frac{2}{R}\sum_{i,j=1}^{2} \int_{\partial B_{R}}\partial_{i}d_{t}\cdot\partial_{j}d \delta_{ij}\text{d}S
+\frac{d}{dt} \|\nabla^{2}d\|_{L^{2}}^{2}\nonumber\\
=&\frac{d}{dt}( \|\nabla^{2}d\|_{L^{2}}^{2}+\frac{1}{R}\|\nabla d\|_{L^{2}(\partial B_{R})}^{2}),
\end{align*}
due to $\frac{\partial}{\partial\nu}d=0$.
Here, $C_{0}$ is the constant defined in Lemma \ref{lem2.5}. On the other hand, by letting $v=\nabla d$, then it is easy to see that the condition \eqref{eq2.5} holds. Hence, Lemma \ref{lem2.5} ensures that
\begin{align*}
\|\nabla^{3} d\|_{L^{2}}\leq C_{0} \|\Delta\nabla d\|_{L^{2}}.
\end{align*}
Inserting the above inequality into \eqref{eq3.20}, and then adding the resulting inequality to \eqref{eq3.7}, it follows that
\begin{align*}
\frac{d}{dt}(\|\nabla d\|_{H^{1}}^{2} +\frac{1}{R}\!\|\nabla d\|_{L^{\!2}(\partial B_{R})}^{2})
+\!\|\nabla d_{t}\|_{\!L^{\!2}}^{2}\!+\!\frac{3}{4C_{0}} \|\nabla^{3}  d\|_{L^{2}}^{2}
\leq & C \|\nabla^{2} d\bar{x}^{\frac{a}{2}}\|_{L^{2}}^{2}+C\Phi^{\beta}(t).
\end{align*}
Multiplying the above inequality by $C_{1}+1$, and then adding the resulting inequality to \eqref{eq3.19}, it holds that
\begin{align}\label{eq3.21}
\frac{d}{dt}(B(t)\!+&(C_{1}\!+\!1)(\|\nabla d \|_{H^{1}}^{2}+\frac{1}{R}\!\|\nabla d\|_{L^{\!2}(\partial B_{R})}^{2}))+\!\|\varrho^{\frac{1}{2}} u_{t}\|_{L^{2}}^{2}\!+\!\|\nabla d_{t}\|_{L^{2}}^{2}\!
+\frac{3}{4C_{0}}\|\nabla^{3} d\|_{L^{2}}^{2}\nonumber\\
\leq& \varepsilon \|\nabla^{2} u\|_{L^{2}}^{2}\!+C \|\nabla^{2} d\bar{x}^{\frac{a}{2}}\|_{L^{2}}^{2} \!+C\Phi^{\beta}(t).
\end{align}

On the other hand, notice that $(\varrho, u, P, d)$ satisfies the following Stokes system
\begin{align*}
\begin{cases}
-\Delta u +\nabla P=-\varrho u_{t}-\varrho u\cdot\nabla u +\operatorname{div}(\nabla d\odot\nabla d), &x\in B_{R},\\
\quad\operatorname{div} u=0,& x\in B_{R},\\
\quad u(x)=0, & x\in\partial B_{R}.
\end{cases}
\end{align*}
Applying the standard $L^{p}$-estimate to the above system (see, e.g., \cite{Temam}) yields that for any $p\in (1,\infty)$,
\begin{align}\label{eq3.22}
\|\nabla^{2} u\|_{L^{p}}+\|\nabla P\|_{L^{p}}\leq C\|\varrho u_{t}\|_{L^{p}}+C\|\varrho u\cdot\nabla u\|_{L^{p}} +C\|\operatorname{div}(\nabla d\odot\nabla d)\|_{L^{p}},
\end{align}
from which, after using \eqref{eq2.4}, \eqref{eq3.3}, \eqref{eq3.6}, and Gagliardo-Nirenberg inequality, we have
\begin{align}\label{eq3.23}
\|\nabla^{2} u\|_{L^{2}}^{2}+\|\nabla P\|_{L^{2}}^{2}\leq &C\|\varrho u_{t}\|_{L^{2}}^{2}+C\|\varrho u\cdot\nabla u\|_{L^{2}}^{2} +C\|\operatorname{div}(\nabla d\odot\nabla d)\|_{L^{2}}^{2}\nonumber\\
\leq& C\|\varrho\|_{L^{\infty}}\|\varrho^{\frac{1}{2}} u_{t}\|_{L^{2}}^{2} +C\|\varrho u\|_{L^{4}}^{2} \|\nabla u\|_{L^{4}}^{2}
+C\|\nabla d\|_{L^{4}}^{2} \|\nabla^{2}d\|_{L^{4}}^{2}\nonumber\\
\leq& C\|\varrho^{\frac{1}{2}} u_{t}\|_{L^{2}}^{2}\!+\!C( \|\varrho^{\frac{1}{2}}u\|_{L^{2}}^{2}\!+\!\|\nabla u\|_{L^{2}}^{2} )\|\nabla u\|_{L^{2}} \|\nabla u\|_{H^{1}}
\!+\!C\|\nabla d\|_{L^{2}} \|\nabla^{2}d\|_{L^{2}}^{2}\|\nabla^{2} d\|_{H^{1}}\nonumber\\
\leq& \frac{1}{4C_{0}} \|\nabla^{3} d\|_{L^{2}}^{2} +\frac{1}{2}\|\nabla^{2} u\|_{L^{2}}^{2}+C\|\varrho^{\frac{1}{2}}u_{t}\|_{L^{2}}^{2}+C(1+\|\nabla u\|_{L^{2}}^{6} +\|\nabla^{2}d\|_{L^{2}}^{4})\nonumber\\
\leq&\frac{1}{4C_{0}} \|\nabla^{3} d\|_{L^{2}}^{2} +\frac{1}{2}\|\nabla^{2} u\|_{L^{2}}^{2}+C\|\varrho^{\frac{1}{2}}u_{t}\|_{L^{2}}^{2} +C\Phi^{\beta}(t).
\end{align}
Substituting \eqref{eq3.23} into \eqref{eq3.21}, and choosing $\varepsilon$ suitably small, we get
\begin{align*}
\frac{d}{dt}(B(t)+(C_{1}+&1)(\|\nabla d \|_{H^{1}}^{2}+\frac{1}{R}\!\|\nabla d\|_{L^{\!2}(\partial B_{R})}^{2}))+\frac{1}{2}\|\varrho^{\frac{1}{2}} u_{t}\|_{L^{2}}^{2}\!+\!\|\nabla d_{t}\|_{L^{2}}^{2}\!
+\frac{1}{2C_{0}}\|\nabla^{3} d\|_{L^{2}}^{2}\nonumber\\
\leq & C \|\nabla^{2} d\bar{x}^{\frac{a}{2}}\|_{L^{2}}^{2} \!+C\Phi^{\beta}(t).
\end{align*}
Integrating the above inequality with respect to time variable over $(0,t)$, it follows from the definition of $B(t)$, \eqref{eq3.5}  and Lemma \ref{lem2.5} that \eqref{eq3.10} holds.

By using \eqref{eq3.23} and \eqref{eq3.10}, it is easy to see that \eqref{eq3.11} holds. This completes the proof of Lemma \ref{lem3.3}.
\end{proof}

\begin{lemma}\label{lem3.5}
Let $(\varrho, u, P, d)$ and $T_{1}$ be as in Lemma \ref{lem3.3}. Then there exists a positive constant $\beta>1$ such that for all $t\in (0,T_{1}]$,
\begin{align}\label{eq3.24}
\sup_{0\leq \tau \leq t}\! \left(\tau \|\varrho^{\frac{1}{2}} u_{\tau}\|_{L^{2}}^{2}\!+\!\tau\| d_{t}\|_{H^{1}}^{2}\right)\!
+\!\int_{0}^{t}\!\! (\tau\|\nabla u_{\tau}\|_{L^{2}}^{2}\!+
\!\tau \|\nabla d_{\tau}\|_{H^{1}}^{2})\text{d}\tau\leq \!C\exp\left\{\!C\exp\left\{C\!\!\int_{0}^{t} \Phi^{\beta}(\tau)\text{d}\tau\right\}\right\}.
\end{align}
\end{lemma}

\begin{proof}\label{proof of lem3.5}
Differentiating $\eqref{eq1.1}_{2}$ with respect to time variable $t$ gives
\begin{align}\label{eq3.25}
\varrho u_{tt} +\varrho u\cdot\nabla u_{t}-\Delta u_{t}= -\varrho_{t}(u_{t}+u\cdot\nabla u)-\varrho u_{t}\cdot\nabla u -\nabla P_{t} -\operatorname{div}(\nabla d\odot \nabla d)_{t}.
\end{align}
Multiplying \eqref{eq3.25} by $u_{t}$ and integrating the resulting equality by parts over $B_{R}$, we obtain after applying $\eqref{eq1.1}_{1}$ and the divergence free condition $\eqref{eq1.1}_{4}$ that
\begin{align}\label{eq3.26}
\frac{1}{2}\frac{d}{dt} \|\varrho^{\frac{1}{2}} u_{t}\|_{L^{2}}^{2}&+\|\nabla u_{t}\|_{L^{2}}^{2}\leq C\int\! \varrho |u| |u_{t}|\left( |\nabla u_{t}|\!+ \!|\nabla u |^{2}\! +\!|u||\nabla^{2} u|\right)\text{d}x\nonumber\\
& +C\int \varrho |u|^{2}|\nabla u| |\nabla u_{t}|\text{d}x+ C\int \varrho |u_{t}|^{2}|\nabla u|\text{d}x +C\int|\nabla d||\nabla d_{t}|
|\nabla u_{t}|\text{d}x\nonumber\\
\triangleq& J_{1}+J_{2}+J_{3}+J_{4}.
\end{align}
By using \eqref{eq2.4}, \eqref{eq3.14}, H\"{o}lder's and Gagliardo-Nirenberg inequalities, it follows that
\begin{align*}
J_{1}\leq &C\|\varrho^{\frac{1}{2}} u\|_{L^{6}} \|\varrho^{\frac{1}{2}} u_{t}\|_{L^{2}}^{\frac{1}{2}}
\|\varrho^{\frac{1}{2}} u_{t}\|_{L^{6}}^{\frac{1}{2}} (\|\nabla u_{t}\|_{L^{2}}+\|\nabla u\|_{L^{4}}^{2})
+C\|\varrho^{\frac{1}{4}} u\|_{L^{12}}^{2} \|\varrho^{\frac{1}{2}} u_{t}\|_{L^{2}}^{\frac{1}{2}}
\|\varrho^{\frac{1}{2}} u_{t}\|_{L^{6}}^{\frac{1}{2}} \|\nabla^{2} u\|_{L^{2}}\nonumber\\
\leq& C(1\!+\!\|\nabla u\|_{L^{2}}^{2}) \|\varrho^{\frac{1}{2}} u_{t}\|_{L^{2}}^{\frac{1}{2}}
(\|\varrho^{\frac{1}{2}} u_{t}\|_{L^{2}}\! +\!\|\nabla u_{t}\|_{L^{2}})^{\frac{1}{2}} (\|\nabla u_{t}\|_{L^{2}}
\!+\!\|\nabla u\|_{L^{2}}^{2}\!+\!\|\nabla u\|_{L^{2}} \|\nabla^{2} u\|_{L^{2}}\!+\!\|\nabla^{2}u\|_{L^{2}})\nonumber\\
\leq& \frac{1}{6} \|\nabla u_{t}\|_{L^{2}}^{2}+ C\Phi^{\beta}(t)(1+\|\varrho^{\frac{1}{2}} u_{t}\|_{L^{2}}^{2}) +C
(1\!+\!\|\nabla u\|_{L^{2}}^{2})\|\nabla^{2} u\|_{L^{2}}^{2};\\
J_{2}+&J_{3}\leq C \|\varrho^{\frac{1}{2}} u\|_{L^{8}}^{2} \|\nabla u\|_{L^{4}}\|\nabla u_{t}\|_{L^{4}}+C\|\varrho^{\frac{1}{2}} u_{t}\|_{L^{6}}^{\frac{3}{2}} \|\varrho^{\frac{1}{2}} u_{t}\|_{L^{2}}^{\frac{1}{2}}\|\nabla u\|_{L^{2}}\nonumber\\
\leq& \frac{1}{6} \|\nabla u_{t}\|_{L^{2}}^{2}+ C\Phi^{\beta}(t)(1+\|\varrho^{\frac{1}{2}} u_{t}\|_{L^{2}}^{2}) +C
(1\!+\!\|\nabla u\|_{L^{2}}^{2})\|\nabla^{2} u\|_{L^{2}}^{2};\\
J_{4}\leq & C\|\nabla d\|_{L^{4}} \|\nabla d_{t}\|_{L^{4}} \|\nabla u_{t}\|_{L^{2}}
\leq  \frac{1}{6} \|\nabla u_{t}\|_{L^{2}}^{2} +C\|\nabla d\|_{L^{2}} \|\nabla d\|_{H^{1}} \|\nabla d_{t}\|_{L^{2}}\|\nabla d_{t}\|_{H^{1}}
\nonumber\\
\leq &\frac{1}{6} \|\nabla u_{t}\|_{L^{2}}^{2} +\frac{1}{4C_{2}+1}\|\nabla^{2} d_{t}\|_{L^{2}}^{2}+C\Phi^{\beta}(t)\|\nabla d_{t}\|_{L^{2}}^{2},
\end{align*}
where the positive constant $C_{2}$ is defined in the following \eqref{eq3.28} and \eqref{eq3.31}. Substituting the estimates of $J_{i} (i=1,2,\cdots, 4)$ into \eqref{eq3.26}, and then using \eqref{eq3.23}, we obtain
\begin{align}\label{eq3.27}
&\frac{d}{dt} \|\varrho^{\frac{1}{2}} u_{t}\|_{L^{2}}^{2}+\|\nabla u_{t}\|_{L^{2}}^{2}\nonumber\\
\leq& C\Phi^{\beta}(t)(1+\|\varrho^{\frac{1}{2}} u_{t}\|_{L^{2}}^{2} +\|\nabla d_{t}\|_{L^{2}}^{2})+ \frac{1}{4C_{2}+1}\|\nabla^{2} d_{t}\|_{L^{2}}^{2}+C
(1\!+\!\|\nabla u\|_{L^{2}}^{2})\| \nabla^{3} d\|_{L^{2}}^{2}.
\end{align}

Now, differentiating $\eqref{eq1.1}_{3}$ with respect to time variable $t$, and then multiplying the resulting equality by $d_{t}$, we deduce that
\begin{align*}
\frac{1}{2}\frac{d}{dt} \|d_{t}\|_{L^{2}}^{2}+\|\nabla d_{t}\|_{L^{2}}^{2}\leq& C\int |u_{t}||\nabla d||d_{t}|\text{d}x +C\int |\nabla d_{t}| |\nabla d||d_{t}|\text{d}x+C\int |\nabla d|^{2} |d_{t}|^{2}\text{d}x\nonumber\\
\triangleq & J_{5}+J_{6}+J_{7},
\end{align*}
where we have used the condition $\frac{\partial}{\partial\nu} d_{t} =0$ on $\partial B_{R}$.
Using  H\"{o}lder's and Gagliardo-Nirenberg inequalities, and \eqref{eq3.14}, we have
\begin{align*}
J_{5}\leq &\|u_{t}\bar{x}^{-\frac{a}{2}}\|_{L^{4}} \|\nabla d \bar{x}^{\frac{a}{2}}\|_{L^{2}} \|d_{t}\|_{L^{4}}\leq (\|\varrho^{\frac{1}{2}} u_{t}\|_{L^{2}}+\|\nabla u_{t}\|_{L^{2}})\|\nabla d \bar{x}^{\frac{a}{2}}\|_{L^{2}} \|d_{t}\|_{L^{2}}^{\frac{1}{2}}\| d_{t}\|_{H^{1}}^{\frac{1}{2}}\nonumber\\
\leq& \frac{1}{4}\|\nabla u_{t}\|_{L^{2}}^{2}+C\Phi^{\beta}(t)(\|\varrho^{\frac{1}{2}} u_{t}\|_{L^{2}}^{2}+ \|d_{t}\|_{L^{2}}^{2}+\|\nabla d_{t}\|_{L^{2}}^{2}),\nonumber\\
J_{6}+&J_{7}\leq \frac{1}{2}\|\nabla d_{t}\|_{L^{2}}+C\|\nabla d\|_{L^{4}}^{2}\|d_{t}\|_{L^{4}}^{2}
\leq \frac{1}{2}\|\nabla d_{t}\|_{L^{2}}+C\|\nabla d\|_{L^{2}}\|\nabla^{2} d\|_{L^{2}}\|d_{t}\|_{L^{2}}\| d_{t}\|_{H^{1}}\nonumber\\
\leq& \frac{1}{2}\|\nabla d_{t}\|_{L^{2}}+C\Phi^{\beta}(t)\|d_{t}\|_{L^{2}}^{2}.
\end{align*}
Hence
\begin{align}\label{eq3.28}
\frac{d}{dt} \|d_{t}\|_{L^{2}}^{2}+\|\nabla d_{t}\|_{L^{2}}^{2}\leq& C_{2}\|\nabla u_{t}\|_{L^{2}}^{2}+C\Phi^{\beta}(t)(\|\varrho^{\frac{1}{2}} u_{t}\|_{L^{2}}^{2}+ \|d_{t}\|_{L^{2}}^{2}+\|\nabla d_{t}\|_{L^{2}}^{2}).
\end{align}

Differentiating \eqref{eq3.8} with respect to time variable $t$ yields
\begin{align}\label{eq3.29}
\nabla d_{tt}-\Delta\nabla d_{t} =-\nabla (u\cdot \nabla d)_{t}+\nabla (|\nabla d|^{2}d)_{t}.
\end{align}
Multiplying \eqref{eq3.29} by $\nabla d_{t}$, and integrating the resulting equality over $B_{R}$, it follows that
\begin{align}\label{eq3.30}
\frac{1}{2}\frac{d}{dt}\|\nabla d_{t}\|_{L^{2}}^{2}&+\|\nabla^{2} d_{t}\|_{L^{2}}^{2}+\frac{1}{R} \|\nabla d_{t}\|_{L^{2}(\partial B_{R})}
\leq  C\int |\nabla u_{t}||\nabla d||\nabla d_{t}|\text{d}x+C\int |\nabla u||\nabla d_{t}|^{2}\text{d}x
\nonumber\\
&+
C\int |u_{t}||\nabla^{2} d| |\nabla d_{t}|\text{d}x
+C\int |\nabla d|^{2}|d_{t}||\nabla^{2} d_{t}|\text{d}x+C\int |\nabla d||\nabla d_{t}| |\nabla^{2} d_{t}|\text{d}x\nonumber\\
&\triangleq J_{8}+J_{9}+J_{10}+J_{11}+J_{12}.
\end{align}
By virtue of H\"{o}lder's and Gagliardo-Nirenberg inequalities, and \eqref{eq3.14}, we have
\begin{align*}
J_{8}\leq& \|\nabla u_{t}\|_{L^{2}}\|\nabla d_{t}\|_{L^{4}}\|\nabla d\|_{L^{4}}\leq \frac{1}{2} \|\nabla u_{t}\|_{L^{2}}^{2}+C
\|\nabla d\|_{L^{2}}\|\nabla d\|_{H^{1}} \|\nabla d_{t}\|_{L^{2}}\|\nabla d_{t}\|_{H^{1}}\nonumber\\
\leq& \frac{1}{2} \|\nabla u_{t}\|_{L^{2}}^{2}+ \frac{1}{4} \|\nabla^{2} d_{t} \|_{L^{2}}^{2}
+C\Phi^{\beta}(t) \|\nabla d_{t}\|_{L^{2}}^{2},\nonumber\\
J_{9}\leq &C\|\nabla u\|_{L^{2}}\|\nabla d_{t}\|_{L^{4}}^{2}
\leq C\|\nabla u\|_{L^{2}}\|\nabla d_{t}\|_{L^{2}} \|\nabla d_{t}\|_{H^{1}}
\leq \frac{1}{8} \|\nabla^{2} d_{t}\|_{L^{2}} +C\Phi^{\beta}(t)\|\nabla d_{t}\|_{L^{2}},\nonumber\\
 J_{10}\leq& C\int |u_{t}\bar{x}^{-\frac{a}{4}}||\nabla^{2}d|^{\frac{1}{2}}\bar{x}^{\frac{a}{4}}|\nabla^{2}d|^{\frac{1}{2}}|\nabla d_{t}|\text{d}x
\leq C\|u_{t}\bar{x}^{-\frac{a}{4}}\|_{L^{4}}\|\nabla^{2}d\bar{x}^{\frac{a}{2}}\|_{L^{2}}^{\frac{1}{2}} \|\nabla^{2}d\|_{L^{2}}^{\frac{1}{2}}
\|\nabla d_{t}\|_{L^{4}}\nonumber\\
\leq&  C\|u_{t}\bar{x}^{-\frac{a}{4}}\|_{L^{4}}\|\nabla^{2}d\bar{x}^{\frac{a}{2}}\|_{L^{2}}^{\frac{1}{2}} \|\nabla^{2}d\|_{L^{2}}^{\frac{1}{2}}
\|\nabla d_{t}\|_{L^{2}}^{\frac{1}{2}}\|\nabla d_{t}\|_{H^{1}}^{\frac{1}{2}}\nonumber\\
\leq& \frac{1}{8} \|\nabla^{2}d_{t}\|_{L^{2}}^{2}+C\|u_{t}\bar{x}^{-\frac{a}{4}}\|_{L^{4}}^{2}+C\|\nabla^{2}d\bar{x}^{\frac{a}{2}}\|_{L^{2}}^{2} \|\nabla^{2}d\|_{L^{2}}^{2}
\|\nabla d_{t}\|_{L^{2}}^{2}\nonumber\\
\leq& \frac{1}{8} \|\nabla^{2}d_{t}\|_{L^{2}}^{2}+C(\|\varrho^{\frac{1}{2}} u_{t}\|_{L^{2}}+\|\nabla u_{t}\|_{L^{2}}^{2})+C\|\nabla^{2}d\bar{x}^{\frac{a}{2}}\|_{L^{2}}^{2} \|\nabla^{2}d\|_{L^{2}}^{2}
\|\nabla d_{t}\|_{L^{2}}^{2},\nonumber\\
J_{11}\leq & \frac{1}{8}\|\nabla^{2} d_{t}\|_{L^{2}}^{2}+C\|\nabla d\|_{L^{8}}^{2}\| d_{t}\|_{L^{4}}^{2}\leq  \frac{1}{8}\|\nabla^{2} d_{t}\|_{L^{2}}^{2}+C\|\nabla d\|_{L^{2}}^{\frac{1}{2}}\|\nabla d\|_{H^{1}}^{\frac{3}{2}}\| d_{t}\|_{L^{2}}\| d_{t}\|_{H^{1}}\nonumber\\
\leq&  \frac{1}{8}\|\nabla^{2} d_{t}\|_{L^{2}}^{2}+C\Phi^{\beta}(t)(\| d_{t}\|_{L^{2}}^{2}+\|\nabla d_{t}\|_{L^{2}}^{2}),\nonumber\\
J_{12}\leq & \frac{1}{16}\|\nabla^{2} d_{t}\|_{L^{2}}^{2}+C\|\nabla d\|_{L^{4}}^{2}\| \nabla d_{t}\|_{L^{4}}^{2}\leq \frac{1}{8}\|\nabla^{2} d_{t}\|_{L^{2}}^{2}+C\Phi^{\beta}(t)\|\nabla d_{t}\|_{L^{2}}^{2}.
\end{align*}
Inserting the estimates of $J_{i} (i=8,9,\cdots, 12)$ into \eqref{eq3.30}, it yields that
\begin{align}\label{eq3.31}
&\frac{d}{dt}\|\nabla d_{t}\|_{L^{2}}^{2}+\|\nabla^{2} d_{t}\|_{L^{2}}^{2}\nonumber\\
\leq& C_{2} (\|\varrho^{\frac{1}{2}} u_{t}\|_{L^{2}}+\|\nabla u_{t}\|_{L^{2}}^{2})+C(\Phi^{\beta}(t)+\|\nabla^{2}d\bar{x}^{\frac{a}{2}}\|_{L^{2}}^{2} \|\nabla^{2}d\|_{L^{2}}^{2})
 (\| d_{t}\|_{L^{2}}^{2}+\|\nabla d_{t}\|_{L^{2}}^{2}).
\end{align}

Finally, multiplying \eqref{eq3.27} by  ${2C_{2}+1}$ and adding the resulting inequality with \eqref{eq3.28} and  \eqref{eq3.31}, we have
\begin{align*}
&\frac{d}{dt}({(2C_{2}+1)}\|\varrho^{\frac{1}{2}} u_{t}\|_{L^{2}}^{2}+\|d_{t}\|_{H^{1}}^{2})+\|\nabla u_{t}\|_{L^{2}}
+\frac{1}{2} \|\nabla d_{t}\|_{H^{1}}^{2} \nonumber\\
\leq& C(\Phi^{\beta}(t)+\|\nabla^{2}d\bar{x}^{\frac{a}{2}}\|_{L^{2}}^{2} \|\nabla^{2}d\|_{L^{2}}^{2})(1+\|\varrho^{\frac{1}{2}} u_{t}\|_{L^{2}}^{2}+\|\nabla d_{t}\|_{H^{1}}^{2})
+C(1+\|\nabla u\|_{L^{2}}^{2})\|\nabla^{3} d\|_{L^{2}}^{2}.
\end{align*}
Multiplying the above inequality by $t$, we obtain \eqref{eq3.24} after using Gronwall's inequality and \eqref{eq3.10}. This completes the proof of Lemma \ref{lem3.5}.
\end{proof}

\begin{lemma}\label{lem3.6}
Let  $(\varrho, u, P, d)$ and $T_{1}$ be as in Lemma \ref{lem3.3}. Then there exists a positive constant $\beta>1$ such that for all $t\in (0,T_{1}]$,
\begin{align}\label{eq3.32}
\sup_{0\leq \tau \leq t} \left(\tau \|\nabla^{2} u\|_{L^{2}}^{2}\!+\tau\|\nabla^{3} d\|_{L^{2}}^{2}\!+\tau \|\nabla P\|_{L^{2}}^{2}\right)
\leq C\exp\left\{\! C\exp\left\{\!C\!\!\int_{0}^{t} \!\Phi^{\beta}(\tau)\text{d}\tau\right\}\!\right\}.
\end{align}
\end{lemma}

\begin{proof}\label{proof of lem3.6}
Applying $H^{k} (k=2,3)$ estimates of elliptic equations (see \cite{HWW12}), it is easy to deduce from $\eqref{eq1.1}_{3}$ that
\begin{align*}
\|\nabla^{3}d \|_{L^{2}}^{2}\leq& C\|\nabla d_{t}\|_{L^{2}}^{2}+
C\||u||\nabla^{2}d|\|_{L^{2}}^{2} +C\||\nabla u||\nabla d|\|_{L^{2}}^{2} +C\|\nabla d\|_{L^{6}}^{6}+C\||\nabla d||\nabla^{2}d|\|_{L^{2}}^{2}.
\end{align*}
By virtue of \eqref{eq3.14}, H\"{o}lder's and Gagliardo-Nirenberg inequalities, it follows that
\begin{align*}
\||u||\nabla^{2}d|\|_{L^{2}}^{2}\leq & C\|u\bar{x}^{-\frac{a}{4}}\|_{\!L^{\!8}}^{2}\|\nabla^{2}d\bar{x}^{\frac{a}{2}}\|_{\!L^{\!2}}\|\nabla^{2}d\|_{\!L^{\!4}}
\leq C\|\nabla^{2}d\bar{x}^{\frac{a}{2}}\|_{\!L^{\!2}}^{2}+C\|u\bar{x}^{-\frac{a}{4}}\|_{\!L^{\!8}}^{4}\|\nabla^{2}d\|_{\!L^{\!4}}^{2}\nonumber\\
\leq& C\|\nabla^{2}d\bar{x}^{\frac{a}{2}}\|_{\!L^{\!2}}^{2}+C(\|\varrho^{\frac{1}{2}} u\|_{L^{2}}^{2}+\|\nabla u\|_{L^{2}}^{2} )^{2}\|\nabla^{2}d\|_{L^{2}}\|\nabla^{2}d\|_{H^{1}}\nonumber\\
\leq&\frac{1}{4}\|\nabla^{3}d\|_{L^{2}}^{2}+C\|\nabla^{2}d\bar{x}^{\frac{a}{2}}\|_{\!L^{\!2}}^{2}+C(1+\|\nabla u\|_{L^{2}}^{8} )\|\nabla^{2}d\|_{L^{2}}^{2},\nonumber\\
\||\nabla u||\nabla d|\|_{L^{2}}^{2}\leq &C\|\nabla u\|_{L^{4}}^{2}\|\nabla^{2 }d \|_{L^{4}}^{2}\leq \|\nabla u\|_{L^{2}}\|\nabla u\|_{H^{1}}
\|\nabla d\|_{L^{2}}\|\nabla d\|_{H^{1}}\nonumber\\
\leq& \frac{1}{4} \|\nabla^{2} u\|_{L^{2}}^{2}+C\|\nabla u\|_{L^{2}}^{2}(\|\nabla d\|_{L^{2}}^{4}+\|\nabla^{2 } d\|_{L^{2}}^{4}),\nonumber\\
\|\nabla d\|_{L^{6}}^{6}+\||\nabla d||\nabla^{2}d|\|_{L^{2}}^{2}&\leq C\|\nabla d\|_{L^{2}}^{2}\|\nabla d\|_{H^{1}}^{4}
+C\|\nabla d\|_{L^{4}}^{2}\|\nabla^{2} d\|_{L^{4}}^{2}\nonumber\\
\leq& C\|\nabla d\|_{L^{2}}^{2}\|\nabla d\|_{H^{1}}^{4}
+C\|\nabla d\|_{L^{2}}\|\nabla d\|_{H^{1}}\|\nabla^{2} d\|_{L^{2}}\|\nabla^{2} d\|_{H^{1}}\nonumber\\
\leq&\frac{1}{4} \|\nabla^{3}d\|_{L^{2}}^{2}+C(1+\|\nabla d\|_{L^{2}}^{8}+\|\nabla^{2}d\|_{L^{2}}^{8}).
\end{align*}
Hence
\begin{align*}
\|\nabla^{3}d \|_{L^{2}}^{2}\leq& \frac{1}{2}\|\nabla^{3}d\|_{L^{2}}^{2} + \frac{1}{4} \|\nabla^{2} u\|_{L^{2}}^{2}+C\|\nabla d_{t}\|_{L^{2}}^{2}+
C\|\nabla^{2}d\bar{x}^{\frac{a}{2}}\|_{\!L^{\!2}}^{2} +C(1+\|\nabla u\|_{L^{2}}^{12}+\|\nabla d\|_{H^{1}}^{12}),
\end{align*}
which together with \eqref{eq3.23} shows that
\begin{align*}
&\|\nabla^{2} u\|_{L^{2}}^{2}+\|\nabla P\|_{L^{2}}^{2} +\|\nabla^{3} d\|_{L^{2}}^{2}\nonumber\\
\leq& C(\|\varrho^{\frac{1}{2}}u_{t}\|_{L^{2}}^{2}+\|\nabla d_{t}\|_{L^{2}}^{2}+
\|\nabla^{2}d\bar{x}^{\frac{a}{2}}\|_{\!L^{\!2}}^{2})+C(1+\|\nabla u\|_{L^{2}}^{12}+\|\nabla d\|_{H^{1}}^{12}).
\end{align*}
Then, multiplying the above inequality by $t$, one obtains from \eqref{eq3.4}, \eqref{eq3.10}, \eqref{eq3.22} and \eqref{eq3.24} that
\begin{align*}
&\sup_{0\leq \tau \leq t} \left(\tau\|\nabla^{2} u\|_{L^{2}}^{2}+\tau\|\nabla P\|_{L^{2}}^{2} +\tau\|\nabla^{3} d\|_{L^{2}}^{2}\right)\nonumber\\
\leq & C\exp\left\{C\exp\left\{C\int_{0}^{t} \Phi^{\beta}(\tau)\text{d}\tau\right\}\right\}+
C\left(1+\exp\left\{C\int_{0}^{t}\Phi^{\beta}(\tau)\text{d}\tau\right\}\right)^{6}\nonumber\\
\leq&C\exp\left\{C\exp\left\{C\int_{0}^{t} \Phi^{\beta}(\tau)\text{d}\tau\right\}\right\},
\end{align*}
which completes the proof of \eqref{eq3.32}.
\end{proof}

\begin{lemma}\label{lem3.7}
Let  $(\varrho, u, P, d)$ and $T_{1}$ be as in Lemma \ref{lem3.3}. Then there exists a positive constant $\beta>1$ such that for all $t\in [0,T_{1}]$,
\begin{align}\label{eq3.33}
\sup_{0\leq \tau \leq t} \|\varrho \bar{x}^{a}\|_{L^{1}\cap H^{1}\cap W^{1,q}}
\leq \exp\left\{C\exp\left\{ C\exp\left\{C\int_{0}^{t} \Phi^{\beta}(\tau)\text{d}\tau\right\}\right\}\right\}.
\end{align}
\end{lemma}

\begin{proof}\label{proof of lem3.7}
Inspired by \cite{LJLZ}, we should first derive the following estimates
\begin{align}\label{eq3.34}
&\int_{0}^{t} \left(\|\nabla^{2}u\|_{L^{q}}^{\frac{q+1}{q}}+\|\nabla P\|_{L^{q}}^{\frac{q+1}{q}}+\tau \|\nabla^{2} u\|_{L^{q}}^{2}+\tau\|\nabla P\|_{L^{q}}^{2}\right)\text{d}\tau\leq  C\exp\left\{C\exp\left\{C\int_{0}^{t} \Phi^{\beta}(\tau)\text{d}\tau\right\}\right\}.
\end{align}
In fact, from \eqref{eq3.22}, H\"{o}lder's and Gagliadro-Nirenberg inequalities, we have for $p=q$,
\begin{align*}
\|\nabla^{2} u\|_{L^{q}}+\|\nabla P\|_{L^{q}}\leq& C\|\varrho u_{t}\|_{L^{q}}+C\|\varrho u\cdot\nabla u\|_{L^{q}} +C\|\operatorname{div}(\nabla d\odot\nabla d)\|_{L^{q}}\nonumber\\
\leq& C \|\varrho u_{t}\|_{L^{q}}+C\|\varrho u\|_{L^{2q}}\|\nabla u\|_{L^{2q}} +
C\|\nabla d\|_{L^{2q}} \|\nabla^{2}d\|_{L^{2q}}\nonumber\\
\leq&\|\varrho u_{t}\|_{L^{2}}^{\frac{2(\lambda-q)}{q(\lambda-2)}} \|\varrho u_{t}\|_{L^{\lambda}}^{\frac{\lambda(q-2)}{q(\lambda-2)}}
+C\Phi^{\beta}(t) (1+ \|\nabla^{2} u\|_{L^{2}}^{1-\frac{1}{q}} +\|\nabla^{3} d\|_{L^{2}}^{1-\frac{1}{q}} )\nonumber\\
\leq& C(\|\varrho^{\frac{1}{2}} u_{t}\|_{L^{2}}^{\frac{2(\lambda-q)}{q(\lambda-2)}}\|\nabla u_{t}\|_{L^{2}}^{\frac{\lambda(q-2)}{q(\lambda-2)}}
\!\!+\|\varrho^{\frac{1}{2}} u_{t}\|_{L^{2}})
+\!C\Phi^{\beta}(t) (1+\! \|\nabla^{2} u\|_{L^{2}}^{1-\frac{1}{q}} \!+\!\|\nabla^{3} d\|_{L^{2}}^{1-\frac{1}{q}} ),
\end{align*}
where $\lambda>q$ will be chosen later. The above inequality together with \eqref{eq3.10}, \eqref{eq3.11} and \eqref{eq3.24} ensures that
\begin{align*}
&\int_{0}^{t} \left(\|\nabla^{2}u\|_{L^{q}}^{\frac{q+1}{q}}+\|\nabla P\|_{L^{q}}^{\frac{q+1}{q}}\right)\text{d}\tau\nonumber\\
\leq& C\int_{0}^{t}\tau^{-\frac{q+1}{2q}}\left(\tau \|\varrho^{\frac{1}{2}} u_{t}\|_{L^{2}}^{2}\right)^{\frac{(\lambda-q)(q+1)}{q^{2}(\lambda-2)}}
\left(\tau\|\nabla u_{\tau}\|_{L^{2}}^{2}\right)^{\frac{\lambda(q-2)(q+1)}{2q^{2}(\lambda-2)}}\text{d}\tau
+C\int_{0}^{t} \|\varrho^{\frac{1}{2}} u_{\tau}\|_{L^{2}}^{\frac{q+1}{q}}\text{d}\tau\nonumber\\
&+C\int_{0}^{t} \Phi^{\beta}(t)(1+\|\nabla^{2} u\|_{L^{2}}^{\frac{q^{2}-1}{q^{2}}} +\|\nabla^{3} d\|_{L^{q}}^{\frac{q^{2}-1}{q^{2}}})\text{d}\tau\nonumber\\
\leq&C \sup_{0\leq \tau\leq t}\left\{\tau \|\varrho^{\frac{1}{2}} u_{\tau}\|_{L^{2}}^{2}\right\}^{\frac{(\lambda-q)(q+1)}{q^{2}(\lambda-2)}}
\int_{0}^{t}\tau^{-\frac{q+1}{2q}}\left(\tau\|\nabla u_{\tau}\|_{L^{2}}^{2}\right)^{\frac{\lambda(q-2)(q+1)}{2q^{2}(\lambda-2)}}\text{d}\tau\nonumber\\
&+C\int_{0}^{t}(\Phi^{\beta}(t)+\|\varrho^{\frac{1}{2}} u_{\tau}\|_{L^{2}}^{2}+\|\nabla^{2} u\|_{L^{2}}^{2} +\|\nabla^{3} d\|_{L^{q}}^{2})\text{d}\tau\nonumber\\
\leq&C\exp\left\{C\exp\left\{C\int_{0}^{t} \Phi^{\beta}(\tau)\text{d}\tau\right\}\right\}\left(1+\int_{0}^{t}\left( \tau^{-\frac{q(q+1)(\lambda-2)}{2q(\lambda-2)-\lambda(q-2)(q+1)}}+\tau\|\nabla u_{\tau}\|_{L^{2}}^{2}\right)\text{d}\tau\right).
\end{align*}
By selecting $\lambda=2q^{2}$,  we have $\frac{q(q+1)(\lambda-2)}{2q(\lambda-2)-\lambda(q-2)(q+1)}=\frac{q^{3}+q^{2}-q-1}{q^{3}+q^{2}}$, and  then
\begin{align*}
\|\nabla^{2} u\|_{L^{q}}+\|\nabla P\|_{L^{q}}\leq& C\exp\left\{C\exp\left\{C\int_{0}^{t} \Phi^{\beta}(\tau)\text{d}\tau\right\}\right\}\left(1+\int_{0}^{t} \left(\tau^{-\frac{q^{3}+q^{2}-q-1}{q^{3}+q^{2}}}+\tau\|\nabla u_{\tau}\|_{L^{2}}^{2}\right)\text{d}\tau\right)\nonumber\\
\leq& C\exp\left\{C\exp\left\{C\int_{0}^{t} \Phi^{\beta}(\tau)\text{d}\tau\right\}\right\}.
\end{align*}
Similarly,
\begin{align*}
\int_{0}^{t}(\tau \|\nabla^{2} u\|_{L^{q}}^{2}+&\tau\|\nabla P\|_{L^{q}}^{2})\text{d}x
\leq C\int_{0}^{t} \left(\tau \|\varrho^{\frac{1}{2}} u_{\tau}\|_{L^{2}}^{2}\right)^{\frac{2q-1}{q^{2}-1}}
\left(\tau\|\nabla u_{\tau}\|_{L^{2}}^{2}\right)^{\frac{q(q-2)}{q^{2}-1}}\text{d}\tau\qquad\quad(\lambda =2q^{2})\nonumber\\
&+C\int_{0}^{t}\tau \|\varrho^{\frac{1}{2}} u_{t}\|_{L^{2}}^{2}\text{d}\tau +C\int_{0}^{t}(\Phi^{\beta}(\tau)+\|\nabla^{2} u\|_{L^{2}}^{2}
+\|\nabla^{3}d\|_{L^{2}}^{2})\text{d}\tau\nonumber\\
\leq & C\int_{0}^{t}\left(\tau \|\varrho^{\frac{1}{2}} u_{t}\|_{L^{2}}^{2}+\tau\|\nabla^{2} u_{\tau}\|_{L^{2}}^{2}\right)\text{d}\tau +C\int_{0}^{t}(\Phi^{\beta}(\tau)+\|\nabla^{2} u\|_{L^{2}}^{2}
+\|\nabla^{3}d\|_{L^{2}}^{2})\text{d}\tau\nonumber\\
\leq& C\exp\left\{C\exp\left\{C\int_{0}^{t} \Phi^{\beta}(\tau)\text{d}\tau\right\}\right\}.
\end{align*}
Thus we obtain \eqref{eq3.34}.

Next, it is easy to derive from $\eqref{eq1.1}_{1}$ that $\varrho \bar{x}^{a}$ satisfies
\begin{align*}
\partial_{t}(\varrho\bar{x}^{a})+ u\cdot\nabla (\varrho\bar{x}^{a}) -a\varrho\bar{x}^{a} u\cdot\nabla \ln\bar{x}=0,
\end{align*}
which implies for all $r\in [2,q]$,
\begin{align}\label{eq3.35}
\frac{d}{dt} \|\nabla(\varrho\bar{x}^{a})\|_{L^{r}}\leq& C(1+\|\nabla u\|_{L^{\infty}}+\|u\cdot\nabla\ln\bar{x}\|_{L^{\infty}}) \|\nabla(\varrho\bar{x}^{a})\|_{L^{r}}\nonumber\\
&+C\|\varrho\bar{x}^{a}\|_{L^{\infty}}(\||\nabla u||\nabla \ln\bar{x}|\|_{L^{r}}+\||u||\nabla^{2}\ln\bar{x}|\|_{L^{r}})\nonumber\\
\leq& C(\Phi^{\beta}(t)+\|\nabla^{2} u\|_{L^{2}\cap L^{q}}) \|\nabla(\varrho\bar{x}^{a})\|_{L^{r}}
+C\|\varrho\bar{x}^{a}\|_{L^{\infty}}(\|\nabla u\|_{L^{r}}+\|u\bar{x}^{-\frac{2}{5}}\|_{L^{4r}}\|\bar{x}^{-\frac{3}{2}}\|_{L^{\frac{4r}{3}}})
\nonumber\\
\leq &C(\Phi^{\beta}(t)+\|\nabla^{2} u\|_{L^{2}\cap L^{q}})(1+\|\nabla(\varrho\bar{x}^{a})\|_{L^{r}}+\|\nabla(\varrho\bar{x}^{a})\|_{L^{q}}),
\end{align}
where we have used \eqref{eq3.14}, \eqref{eq3.15} and the following calculation\footnote{We notice that the inequality \eqref{eq3.36} still holds if we replace the integral domain $B_{R}$ by $\mathbb{R}^{2}$.}
\begin{align}\label{eq3.36}
\|u\bar{x}^{-\delta}\|_{L^{\infty}} \leq& C(\delta) (\|u\bar{x}^{-\delta}\|_{L^{\frac{4}{\delta}}}) +\|\nabla(u\bar{x}^{-\delta})\|_{L^{3}})\nonumber\\
\leq& C(\delta) (\|u\bar{x}^{-\delta}\|_{L^{\frac{4}{\delta}}} +\|\nabla u\|_{L^{3}}+ \|u\bar{x}^{-\delta}\|_{L^{\frac{4}{\delta}}}\|\bar{x}^{-1}\nabla \bar{x}\|_{L^{\frac{12}{4-3\delta}}})\nonumber\\
\leq&C(\delta) (\Phi^{\beta}(t)+\|\nabla^{2} u\|_{L^{2}}), \quad \quad \text{ for } 0<\delta<1.
\end{align}
Utilizing \eqref{eq3.34}, \eqref{eq3.11} and Gronwall's inequality, it yields from \eqref{eq3.35} that
\begin{align*}
\sup_{0\leq \tau\leq t} \|\nabla (\varrho\bar{x}^{a})\|_{L^{2}\cap L^{p}} \leq \exp\left\{ C\exp\left\{C\exp\left\{C\int_{0}^{t} \Phi^{\beta}(\tau)\text{d}\tau\right\}\right\}\right\}.
\end{align*}
This, along with \eqref{eq3.15} gives \eqref{eq3.33}. This completes the proof of Lemma \ref{lem3.7}.
\end{proof}
\medskip

Now, we can deal with the proof of Proposition \ref{prop3.1}.
\medskip

\textbf{Proof of Proposition \ref{prop3.1}.} \ It follows from\eqref{eq3.3}, \eqref{eq3.4}, \eqref{eq3.5}, \eqref{eq3.10}, and \eqref{eq3.33} that
\begin{align*}
\Phi(t)\leq \exp\left\{C\exp\left\{C\exp\left\{C\int_{0}^{t}\Phi^{\beta}(\tau)\text{d}\tau\right\}\right\}\right\},
\quad\quad \tau\in (0,T_{1}].
\end{align*}
By using standard arguments yields that for $M\triangleq e^{Ce^{Ce}}$ and $T_{0}\triangleq\min\{T_{1}, (CM^{\beta})^{-1}\}$,
\begin{align*}
\sup_{0\leq t\leq T_{0}}\Phi (t)\leq M,
\end{align*}
which together with \eqref{eq3.4}, \eqref{eq3.5}, \eqref{eq3.10}, \eqref{eq3.11}, \eqref{eq3.24} and \eqref{eq3.34} ensures \eqref{eq3.2}. This completes the proof of Proposition \ref{prop3.1}.\hfill$\Box$

\subsection{The proof of Theorem \ref{thm3.1}}

With the a priori estimates obtained in Subsection 3.1, we shall give the proof of Theorem \ref{thm3.1}.  Let  $(\varrho_{0}, u_{0}, d_{0})$ be as
in Theorem \ref{thm3.1}.Then it follows from \eqref{eq1.5} that there exists a positive constant $N_{0}$ such that
\begin{align*}
\int_{B_{N_{0}}} \varrho_{0}(x)\text{d}x\geq \frac{3}{4}\int_{\mathbb{R}^{2}}\varrho(x)\text{d}x=\frac{3}{4}.
\end{align*}

Now, we construct $\varrho_{0}^{R}=\widehat{\varrho}_{0}^{T}+R^{-1} e^{-|x|^{2}}$, where $0\leq \varrho_{0}\in C^{\infty}_{0}(\mathbb{R}^{2})$ satisfies
\begin{align*}
\begin{cases}
\int_{B_{N_{0}}}\widehat{\varrho}_{0}^{R}(x)\text{d}x\geq \frac{1}{2}, \\
\widehat{\varrho}_{0}^{R}\bar{x}^{a}\rightarrow \varrho_{0}\bar{x}^{a}\quad \text{in } L^{1}(\mathbb{R}^{2})\cap H^{1}(\mathbb{R}^{2})\cap W^{1,q}(\mathbb{R}^{2}), \quad \text{as } R\rightarrow \infty.
\end{cases}
\end{align*}
Noticing that $d_{0}\in L^{2}(\mathbb{R}^{2};\mathbb{S}^{2})$, $\nabla d_{0}\bar{x}^{\frac{a}{2}}\in L^{2}(\mathbb{R}^{2})$ and  $\nabla^{2}d_{0}\in L^{2}(\mathbb{R}^{2})$, standard arguments (cf. \cite{LW08} and also \cite{HWW12}) yield that there exists $d_{0}^{R}\in H^{3}(\mathbb{R}^{2}; \mathbb{S}^{2})$ such that
$d_{0}^{R}\equiv n_{0}$ outside $B_{\frac{R}{2}}$ for some constant vector $n_{0}\in \mathbb{S}^{2}$ and
\begin{align*}
\lim_{R\rightarrow\infty}\|\nabla d_{0}^{R}- \nabla d_{0}\|_{H^{2}(\mathbb{R}^{2})} =0,\quad \lim_{R\rightarrow\infty}\|\nabla d_{0}^{R} \bar{x}^{\frac{a}{2}} - \nabla d_{0}\bar{x}^{\frac{a}{2}}\|_{L^{2}(\mathbb{R}^{2})}=0.
\end{align*}
Since $\nabla u_{0}\in L^{2}(\mathbb{R}^{2})$, choosing $v_{i}^{R}\in C_{0}^{\infty}(B_{R}) (i=1,2)$ such that for $i=1,2$,
\begin{align*}
\lim_{R\rightarrow\infty}\|v_{i}^{R}-\partial_{i}u_{0}\|_{L^{2}(\mathbb{R}^{2})}=0.
\end{align*}
Let $u_{0}^{R}$ be the unique smooth solution of the following elliptic problem :
\begin{align*}
\begin{cases}
-\Delta u_{0}^{R} +\varrho_{0}^{R} u_{0}^{R} +\nabla P_{0}^{R}=\sqrt{\varrho_{0}^{R}} h^{R}-\partial_{i}v_{i}^{R}, \quad &\text{ in }B_{R},\\
\operatorname{div}u_{0}^{R} =0,\quad&\text{ in }B_{R},\\
u_{0}^{R}=0,\quad &\text{ in }\partial B_{R},
\end{cases}
\end{align*}
where $h^{R}=(\varrho_{0}^{\frac{1}{2}} u_{0})\ast j_{\frac{1}{R}}$ with $j_{\delta}$ being the standard mollifying kernel of width $\delta$. Extending $u_{0}^{R}$ to $\mathbb{R}^{2}$  by defining $0$ outside $B_{R}$ and denoting it by $\widetilde{u}_{0}^{R}$, then by the same arguments as those for the proof of Theorem 1.1 in \cite{LJLZ} (see also \cite{LXZ}), we obtain that
\begin{align*}
\lim_{R\rightarrow \infty} \left(\|\nabla (\widetilde{u}_{0}^{R}-u_{0})\|_{L^{2}(\mathbb{R}^{2})}
+\|(\varrho_{0}^{R})^{\frac{1}{2}} \widetilde{u}_{0}^{R}-\varrho_{0}^{\frac{1}{2}} u_{0}\|_{L^{2}(\mathbb{R}^{2})}\right) =0.
\end{align*}

Then, by virtue of Lemma \ref{lem2.1}, the initial-boundary value problem \eqref{eq1.1} and \eqref{eq2.2} with the initial datum $(\varrho_{0}^{R}, u_{0}^{R}, d_{0}^{R})$ has a classical solution $(\varrho^{R}, u^{R}, P^{R},d^{R})$ on $B_{R}\times [0,T_{R}]$. Furthermore, Proposition \ref{prop3.1} shows that there exists a $T_{0}$ independent of $R$ such that \eqref{eq3.2} holds for $(\varrho^{R}, u^{R}, P^{R},d^{R})$. Extending $(\varrho^{R}, u^{R}, P^{R},d^{R})$ by zero on $\mathbb{R}^{2}\backslash B_{R}$ and denoting it by
\begin{align*}
(\widetilde{\varrho}^{R}\triangleq \varphi_{R}\varrho^{R}, \widetilde{u}^{R}, \widetilde{P}^{R}, \widetilde{d}^{R}),
\end{align*}
with $\varphi_{R}$ as in \eqref{eq3.12}, it deduces from \eqref{eq3.2} that
\begin{align}\label{eq3.37}
&\sup_{0\leq t\leq T_{0}} \left(
\|(\widetilde{\varrho}^{R})^{\frac{1}{2}} \widetilde{u}^{R}\|_{L^{2}(\mathbb{R}^{2})} +\|\nabla \widetilde{u}^{R}\|_{L^{2}(\mathbb{R}^{2})}
+\|\nabla^{2} \widetilde{d}\|_{L^{2}(\mathbb{R}^{2})}+\|\nabla \widetilde{d} \bar{x}^{\frac{a}{2}}\|_{L^{2}(\mathbb{R}^{2})}\right)\nonumber\\
\leq&\sup_{0\leq t\leq T_{0}} \left(
\|(\varrho^{R})^{\frac{1}{2}} u^{R}\|_{L^{2}(B_{R})} +\|\nabla u^{R}\|_{L^{2}(B_{R})}
+\|\nabla^{2} d\|_{L^{2}(B_{R})}+\|\nabla d \bar{x}^{\frac{a}{2}}\|_{L^{2}(B_{R})}\right)
\leq C,
\end{align}
and
\begin{align}\label{eq4.2}
\sup_{0\leq t\leq T_{0}} \|\widetilde{\varrho}^{R}\bar{x}^{a}\|_{L^{1}(\mathbb{R}^{2})\cap L^{\infty}(\mathbb{R}^{2})}\leq C \sup_{0\leq t\leq T_{0}} \|\varrho^{R}\bar{x}^{a}\|_{L^{1}(B_{R})\cap L^{\infty}(B_{R})}\leq C.
\end{align}
Similarly, it follows from \eqref{eq3.2} that for $q> 2$,
\begin{align}\label{eq4.3}
&\sup_{0\leq t\leq T_{0}}t^{\frac{1}{2}} \left(
\|(\widetilde{\varrho}^{R})^{\frac{1}{2}} \widetilde{u}_{t}^{R}\|_{L^{2}(\mathbb{R}^{2})}
+\|\widetilde{d}_{t}^{R}\|_{H^{1}(\mathbb{R}^{2})} +\|\nabla^{2} \widetilde{u}^{R}\|_{L^{2}(\mathbb{R}^{2})}
+\|\nabla^{3} \widetilde{d}^{R}\|_{L^{2}(\mathbb{R}^{2})}
\right) \nonumber\\
+&\int_{0}^{T_{0}} \left( \|(\widetilde{\varrho}^{R})^{\frac{1}{2}} \widetilde{u}_{t}^{R}\|_{L^{2}(\mathbb{R}^{2})}^{2}
+\|\nabla^{2} \widetilde{u}^{R}\|_{L^{2}(\mathbb{R}^{2})}^{2} +\|\widetilde{d}_{t}^{R}\|_{H^{1}(\mathbb{R}^{2})}^{2}
+\|\nabla^{3}\widetilde{d}^{R}\|_{L^{2}(\mathbb{R}^{2})}^{2} +\|\nabla^{2}\widetilde{d}^{R} \bar{x}^{\frac{a}{2}}\|_{L^{2}(\mathbb{R}^{2})}^{2}
\right)\text{d}t\nonumber\\
+&\int_{0}^{T_{0}}\left(\|\nabla^{2}\widetilde{u}^{R}\|_{L^{q}(\mathbb{R}^{2})}^{\frac{q+1}{q}} + t \|\nabla^{2} \widetilde{u}^{R}\|_{L^{q}(\mathbb{R}^{2})}^{2} +t\|\nabla \widetilde{u}_{t}^{R}\|_{L^{2}(\mathbb{R}^{2})}^{2}
+t\|\nabla \widetilde{d}_{t}^{R}\|_{H^{1}(\mathbb{R}^{2})}^{2}\right)
\leq C.
\end{align}
Next, for $p\in [2,q]$, it follows from \eqref{eq3.2} and \eqref{eq3.33} that
\begin{align}\label{eq4.4}
\sup_{0\leq t\leq T_{0}} \|\nabla(\widetilde{\varrho}^{R}\bar{x}^{a})\|_{L^{p}(\mathbb{R}^{2})}\leq & C
\sup_{0\leq t\leq T_{0}} \left(  \|\nabla(\varrho^{R}\bar{x}^{a})\|_{L^{p}(B_{R})} +R^{-1} \|\varrho^{R}\bar{x}^{a}\|_{L^{p}(B_{R})} \right)
\nonumber\\
\leq& C\sup_{0\leq t\leq T_{0}} \|\varrho^{R}\bar{x}^{a}\|_{H^{1}(B_{R})\cap W^{1,p}(B_{R})} \leq C,
\end{align}
which together with \eqref{eq3.36} and \eqref{eq3.2} ensures
\begin{align}\label{eq4.5}
&\int_{0}^{T_{0}} \|\widetilde{\varrho}^{R}_{t}\bar{x}^{a}\|_{L^{p}(\mathbb{R}^{2})}^{2}\text{d}t
\leq C\int_{0}^{T_{0}} \||u^{R}||\nabla \varrho^{R}| \bar{x}^{a}\|_{L^{p}(B_{R})}^{2}\text{d}t\nonumber\\
\leq & C\int_{0}^{T_{0}} \|u^{R}\bar{x}^{1-a}\|_{L^{\infty}(B_{R})}^{2}  \|\nabla\varrho \bar{x}^{a}\|_{L^{p}(B_{R})}^{2}\text{d}t\leq C.
\end{align}
By virtue of the same arguments as those of \eqref{eq3.23} and
\eqref{eq3.34}, it follows that for $q>2$
\begin{align}\label{eq3.42}
\sup_{0\leq t\leq T_{0}} t^{\frac{1}{2}} \|\nabla\widetilde{P}^{R}\|_{L^{2}(\mathbb{R}^{2})} +\int_{0}^{T_{0}}
\left(\|\nabla\widetilde{P}^{R}\|_{L^{2}(\mathbb{R}^{2})}^{2} + \|\nabla\widetilde{P}^{R}\|_{L^{q}(\mathbb{R}^{2})}^{\frac{q+1}{q}} \right)\leq C.
\end{align}

With all these estimates \eqref{eq3.37}--\eqref{eq3.42} at hand, we find that the sequence $(\widetilde{\varrho}^{R}, \widetilde{u}^{R}, \widetilde{P}^{R}, \widetilde{d}^{R})$ converges, up to the extraction of subsequences, to some limit $(\varrho, u, P, d)$ in the obvious weak sense, i.e., as $R\rightarrow \infty$, we have
\begin{align}\label{eq3.43}
\begin{cases}
&\widetilde{\varrho}^{R}\bar{x}\rightarrow \varrho \bar{x}\quad \text{ in } C(B_{N}\times [0,T_{0}]) \text{ for all } N>0,\\
 &\widetilde{\varrho}^{R}\bar{x}^{a}\rightharpoonup \varrho \bar{x}^{a}\quad \text{ weakly * in  } L^{\infty}(0,T_{0}; H^{1}(\mathbb{R}^{2})\cap W^{1,q}(\mathbb{R}^{2})),\\
 &\nabla \widetilde{d}^{R}\bar{x}^{\frac{a}{2}}\rightharpoonup \nabla d \bar{x}^{\frac{a}{2}}\quad \text{ weakly * in  } L^{\infty}(0,T_{0}; L^{2}(\mathbb{R}^{2})),\\
 &(\widetilde{\varrho}^{R})^{\frac{1}{2}}\widetilde{u}^{R} \rightharpoonup \varrho^{\frac{1}{2}} u,
  \nabla \widetilde{u}^{R} \rightharpoonup \nabla u, \nabla^{2}\widetilde{d}^{R} \rightharpoonup \nabla^{2} d \quad \text{ weakly * in }
  L^{\infty}(0,T_{0}; L^{2}(\mathbb{R}^{2})),\\
  &\nabla^{2} \widetilde{u}^{R}\rightharpoonup \nabla^{2} u, \nabla\widetilde{P}^{R}\rightharpoonup \nabla P \quad
  \text{ weakly in  } L^{\frac{q+1}{q}}(0,T_{0}; L^{q}(\mathbb{R}^{2}))\cap L^{2} (\mathbb{R}^{2}\times (0,T_{0})),\\
  &\widetilde{d}_{t}^{R}\rightharpoonup d_{t} \quad \text{ weakly in  } L^{2}(0,T_{0}; H^{1}(\mathbb{R}^{2})),\\
  &\nabla^{2} \widetilde{d}^{R}\bar{x}^{\frac{a}{2}} \rightharpoonup \nabla^{2}d\bar{x}^{\frac{a}{2}}, \nabla^{3}\widetilde{d}^{R}\rightharpoonup
  \nabla^{3} d\quad \text{ weakly in } L^{2}(\mathbb{R}^{2}\times (0,T_{0})),\\
  & t^{\frac{1}{2}}\nabla^{2}\widetilde{u}^{R}\rightharpoonup t^{\frac{1}{2}} \nabla^{2} u\quad \text{ weakly in } L^{2}(0,T_{0}; L^{q}(\mathbb{R}^{2})), \text{ weakly  * in } L^{\infty}(0,T_{0};L^{2}(\mathbb{R}^{2})),\\
  & t^{\frac{1}{2}} (\widetilde{\varrho}^{R})^{\frac{1}{2}} \widetilde{u}_{t}^{R} \rightharpoonup t^{\frac{1}{2}}\varrho^{\frac{1}{2}} u_{t},
  t^{\frac{1}{2}}\nabla \widetilde{P}^{R}\rightharpoonup t^{\frac{1}{2}}\nabla P \quad \text{ weakly * in } L^{\infty}(0,T_{0}; L^{2}(\mathbb{R}^{2})),\\
  &t^{\frac{1}{2}}\widetilde{d}_{t}^{R} \rightharpoonup  t^{\frac{1}{2}}d_{t}\quad \text{ weakly * in } L^{\infty}(0,T_{0}; H^{1}(\mathbb{R}^{2})),\\
  &t^{\frac{1}{2}}\nabla^{3}\widetilde{d}^{R}\rightharpoonup t^{\frac{1}{2}}\nabla^{3} d\quad \text{ weakly in } L^{\infty} (0,T_{0}; L^{2}(\mathbb{R}^{2})),\\
  &t^{\frac{1}{2}}\nabla \widetilde{u}_{t}^{R}\rightharpoonup t^{\frac{1}{2}}\nabla u_{t}   \quad \text{ weakly in } L^{2}(\mathbb{R}^{2}\times (0,T_{0})),\\
&t^{\frac{1}{2}}\nabla \widetilde{d}_{t}^{R}\rightharpoonup t^{\frac{1}{2}}\nabla d_{t}\quad \text{ weakly in }  L^{2}(0,T_{0}; H^{1}(\mathbb{R}^{2})),
\end{cases}
\end{align}
with
\begin{align}\label{eq3.44}
\varrho\bar{x}^{a}\in L^{\infty}(0,T_{0}; L^{1}(\mathbb{R}^{2})),\quad \inf_{0\leq t\leq T_{0}} \int_{B_{2N_{0}}} \varrho(x,t)\text{d}x\geq \frac{1}{4}.
\end{align}
Then, letting $R\rightarrow \infty$, standard arguments together with  \eqref{eq3.43} and \eqref{eq3.44} yield that $(\varrho, u, P, d)$ is a strong solution of system \eqref{eq1.1}--\eqref{eq1.2} on $\mathbb{R}^{2}\times (0,T_{0}]$ satisfying \eqref{eq1.8} and \eqref{eq1.9}. Indeed, the existence of a pressure $P$ follows immediately from the equations $\eqref{eq1.1}_{2}$ and $\eqref{eq1.1}_{4}$ by a classical consideration. Thus we complete the proof of the existence part of Theorem \ref{thm3.1}.

In what follows, we shall prove the uniqueness of the strong solutions. Let $(\varrho_{1}, u_{1}, P_{1},d_{1})$ and  $(\varrho_{2}, u_{2}, P_{2},d_{2})$ be two strong solutions satisfying \eqref{eq1.8} and \eqref{eq1.9} with the same initial data. Let $\widetilde{\varrho}=\varrho_{1}-\varrho_{2}$, $\widetilde{u}=u_{1}-u_{2}$, $\widetilde{P}=P_{1}-P_{2}$ and $\widetilde{d}=d_{1}-d_{2}$, it follows that
\begin{align}\label{eq3.45}
\begin{cases}
\widetilde{\varrho}_{t}+ u_{2}\cdot\nabla\widetilde{\varrho}+\widetilde{u}\cdot\nabla\varrho_{1}=0,\\
\varrho_{1}\widetilde{u}_{t}\!+\!\varrho_{1}u_{1}\!\cdot\!\nabla\widetilde{u}-\!\Delta\widetilde{u}=\varrho_{1} \widetilde{u}\cdot\nabla u_{2}
\!-\widetilde{\varrho} (u_{2t}\!+u_{2}\cdot\nabla u_{2})\!-\nabla \widetilde{P}\!-\operatorname{div}(\nabla\widetilde{d}\odot\nabla d_{1})\!-\operatorname{div}(\nabla d_{2}\odot\nabla \widetilde{d}),\\
\widetilde{d}_{t}-\Delta \widetilde{d}=-u_{2}\cdot\nabla \widetilde{d} -\widetilde{u}\cdot\nabla d_{1} +|\nabla d_{2}|^{2}\widetilde{d}
+\nabla(d_{1}+d_{2})\nabla \widetilde{d} d_{1},\\
\operatorname{div} \widetilde{u} =0,
\end{cases}
\end{align}
for $(x,t)\in \mathbb{R}^{2}\times (0,T_{0}]$ with
\begin{align*}
\widetilde{\varrho}(x,0) =\widetilde{u}(x,0)=\widetilde{d}(x,0)=0.
\end{align*}
Multiplying $\eqref{eq3.45}_{1}$ by $2\widetilde{\varrho}\bar{x}^{2r}$ for $r\in (1,\bar{a})$ with $\bar{a}=\min\{2,a\}$, and then integrating by parts over $\mathbb{R}^{2}$ yield
\begin{align*}
\frac{d}{dt} \!\int_{\!\mathbb{R}^{2}}\! |\widetilde{\varrho}\bar{x}^{r}|^{2}\text{d}x\!
\leq& C\|u_{2}\bar{x}^{\frac{1}{2}}\|_{\!L^{\!\infty}(\mathbb{R}^{2})}
\|\widetilde{\varrho}\bar{x}^{r}\|_{\!L^{2}(\mathbb{R}^{2})}^{2}\!
+\! C \|\widetilde{\varrho}\bar{x}^{r}\|_{L^{2}(\mathbb{R}^{2})} \|\widetilde{u}\bar{x}^{-(\bar{a}-r)}\|_{\!L^{\!\frac{2q}{(q-2)(\bar{a}-r)}}(\mathbb{R}^{2})} \|\varrho_{1} \bar{x}^{\bar{a}}\|_{\!L^{\!\frac{2q}{q-(q-2)(\bar{a}-r)}}(\mathbb{R}^{2})}\nonumber\\
\leq & C(1+\|u_{2}\|_{W^{1,q}(\mathbb{R}^{2})}) \|\widetilde{\varrho}\bar{x}^{r}\|_{L^{2}(\mathbb{R}^{2})}^{2}
+C\|\widetilde{\varrho}\bar{x}^{r}\|_{L^{2}(\mathbb{R}^{2})}(\|\nabla \widetilde{u}\|_{L^{2}(\mathbb{R}^{2})} +\|\varrho_{1}^{\frac{1}{2}} \widetilde{u}\|_{L^{2}(\mathbb{R}^{2})}),
\end{align*}
due to Sobolev's inequality, \eqref{eq1.8}, Lemma \ref{lem2.3} and \eqref{eq3.36}. This combined with Gronwall's inequality shows that for all $0\leq t\leq T_{0}$,
\begin{align}\label{eq3.46}
\|\widetilde{\varrho}\bar{x}^{r}\|_{L^{2}(\mathbb{R}^{2})}\leq C\int_{0}^{t} (\|\nabla \widetilde{u}\|_{L^{2}(\mathbb{R}^{2})} +\|\varrho_{1}^{\frac{1}{2}} \widetilde{u}\|_{L^{2}(\mathbb{R}^{2})})\text{d}\tau.
\end{align}

Multiplying $\eqref{eq3.45}_{3}$ with $\widetilde{d}$, and then integrating by parts over $\mathbb{R}^{2}$, it follows from \eqref{eq1.8} and Lemma \ref{lem2.3} that
\begin{align}\label{eq3.47}
&\frac{1}{2}\frac{d}{dt} \int_{\mathbb{R}^{2}} |\widetilde{d}|^{2}\text{d}x+\int_{\mathbb{R}^{2}} |\nabla \widetilde{d}|^{2}\text{d}x\nonumber\\
\leq& \int_{\mathbb{R}^{2}} |\widetilde{u}||\nabla d_{1}||\widetilde{d}|\text{d}x+\int_{\mathbb{R}^{2}} |\nabla d_{2}|^{2}|\widetilde{d}|^{2}\text{d}x+\int_{\mathbb{R}^{2}}|\nabla (d_{1}+d_{2})||\nabla\widetilde{d}||\widetilde{d}|\text{d}x\nonumber\\
\leq & C\|\widetilde{u}\bar{x}^{-\frac{a}{4}}\|_{L^{4}(\mathbb{R}^{2})}\|\nabla d_{1}\bar{x}^{\frac{a}{2}}\|_{L^{2}(\mathbb{R}^{2})}^{\frac{1}{2}}  \|\nabla d_{1}\|_{L^{2}(\mathbb{R}^{2})}^{\frac{1}{2}}\|\widetilde{d}\|_{L^{4}(\mathbb{R}^{2})} +C\|\nabla d_{2}\|_{L^{4}(\mathbb{R}^{2})}^{2} \|\widetilde{d}\|_{L^{4}(\mathbb{R}^{2})}^{2}\nonumber\\
 &
+C(\|\nabla d_{1}\|_{L^{4}(\mathbb{R}^{2})}+\|\nabla d_{2}\|_{L^{4}(\mathbb{R}^{2})})\|\widetilde{d}\|_{L^{4}(\mathbb{R}^{2})} \|\nabla \widetilde{d}\|_{L^{2}(\mathbb{R}^{2})}\nonumber\\
\leq& \varepsilon\|\nabla \widetilde{d}\|_{L^{2}(\mathbb{R}^{2})}^{2} +\varepsilon\|\widetilde{u}\bar{x}^{-\frac{a}{4}}\|_{L^{4}(\mathbb{R}^{2})}^{2}
+C\|\nabla d_{1}\bar{x}^{\frac{a}{2}}\|_{L^{2}(\mathbb{R}^{2})}^{2}  \|\nabla d_{1}\|_{L^{2}(\mathbb{R}^{2})}^{2} \|\widetilde{d}\|_{L^{2}(\mathbb{R}^{2})}^{2}\nonumber\\
&+C(\varepsilon)(\|\nabla d_{1}\|_{L^{2}(\mathbb{R}^{2})}^{2} \|\nabla^{2} d_{1}\|_{L^{2}(\mathbb{R}^{2})}^{2}+\|\nabla d_{2}\|_{L^{2}(\mathbb{R}^{2})}^{2} \|\nabla^{2} d_{2}\|_{L^{2}(\mathbb{R}^{2})}^{2})\|\widetilde{d}\|_{L^{2}(\mathbb{R}^{2})}^{2}\nonumber\\
\leq& \varepsilon\|\nabla \widetilde{d}\|_{L^{2}(\mathbb{R}^{2})}^{2} +\varepsilon(\|\varrho_{1} \widetilde{u}\|_{L^{2}(\mathbb{R}^{2})}^{2}+\|\nabla \widetilde{u}\|_{L^{2}(\mathbb{R}^{2})}^{2})+C(\varepsilon)\|\widetilde{d}\|_{L^{2}(\mathbb{R}^{2})}^{2},
\end{align}
where we have used $|d_{1}|=1$, and  the fact that the divergence free condition $\operatorname{div}u_{2}=0$ ensures the identity $\int_{\mathbb{R}^{2}} u_{2}\cdot\nabla\widetilde{d}\cdot\widetilde{d}\text{d}x=0$ in the first inequality.

Next, multiplying $\eqref{eq3.45}_{2}$ and $\eqref{eq3.45}_{3}$ by $\widetilde{u}$ and $-\Delta\widetilde{d}$ respectively, and adding the resulting equations together, we obtain after integration by parts over $\mathbb{R}^{2}$ that
\begin{align}\label{eq3.48}
&\frac{1}{2}\frac{d}{dt}\int_{\mathbb{R}^{2}} (\varrho_{1} |\widetilde{u}|^{2}+|\nabla \widetilde{d}|^{2})\text{d}x+\int_{\mathbb{R}^{2}} (|\nabla \widetilde{u}|^{2}+|\nabla^{2}\widetilde{d}|^{2})\text{d}x\nonumber\\
\leq &C\|\nabla u_{2}\|_{L^{\infty}(\mathbb{R}^{2})}\!\int_{\mathbb{R}^{2}}\!(\varrho_{1} |\widetilde{u}|^{2}\!+|\nabla \widetilde{d}|^{2})\text{d}x
+\!\int_{\mathbb{R}^{2}}\! |\widetilde{\varrho}||\widetilde{u}|(|u_{2t}|+|u_{2}||\nabla u_{2}|)\text{d}x +\!\int_{\mathbb{R}^{2}} \!(|\nabla d_{1}|+|\nabla d_{2}|)|\nabla \widetilde{d}||\nabla \widetilde{u}|\text{d}x\nonumber\\
&+\int_{\mathbb{R}^{2}}|\widetilde{u}||\nabla d_{1}| |\Delta \widetilde{d}|\text{d}x +\int_{\mathbb{R}^{2}}|\nabla d_{2}|^{2}|\widetilde{d}| |\Delta\widetilde{d}|\text{d}x+\int_{\mathbb{R}^{2}}(|\nabla d_{1}|+|\nabla d_{2}|)|\nabla\widetilde{d}||\Delta\widetilde{d}|\text{d}x\nonumber\\
\triangleq&C\|\nabla u_{2}\|_{L^{\infty}(\mathbb{R}^{2})}\int_{\mathbb{R}^{2}}(\varrho_{1} |\widetilde{u}|^{2}+|\nabla \widetilde{d}|^{2})\text{d}x
+II_{1}+II_{2}+\cdots+II_{5}.
\end{align}
By using H\"{o}lder's inequality, Lemma \ref{lem2.3} and \eqref{eq3.46}, we have for $r\in (1,\bar{a})$,
\begin{align*}
II_{1}\leq &C\|\widetilde{\varrho} \bar{x}^{r}\|_{L^{2}(\mathbb{R}^{2})}\|\widetilde{u}\bar{x}^{-\frac{r}{2}}\|_{L^{4}(\mathbb{R}^{2})}
\left( \| u_{2t}\bar{x}^{-\frac{r}{2}}\|_{L^{4}(\mathbb{R}^{2})} +\|\nabla u_{2}\|_{L^{\infty}(\mathbb{R}^{2})} \|u_{2}\bar{x}^{-\frac{r}{2}}\|_{L^{4}(\mathbb{R}^{2})}\right)\nonumber\\
\leq &\varepsilon \|\widetilde{u}\bar{x}^{-\frac{r}{2}}\|_{L^{4}(\mathbb{R}^{2})}^{2}+C(\varepsilon) \left(
\|\varrho_{2}^{\frac{1}{2}} u_{2t}\|_{L^{2}(\mathbb{R}^{2})}^{2} +\|\nabla u_{2t}\|_{L^{2}(\mathbb{R}^{2})}^{2} +\|\nabla u_{2}\|_{L^{\infty}(\mathbb{R}^{2})}^{2}\right)\|\widetilde{\varrho} \bar{x}^{r}\|_{L^{2}(\mathbb{R}^{2})}^{2}\nonumber\\
\leq&\varepsilon(\|\varrho_{1}^{\frac{1}{2}} \widetilde{u}\|_{\!L^{2}(\mathbb{R}^{2})}^{2}\!+\!\|\nabla\widetilde{u}\|_{\!L^{2}(\mathbb{R}^{2})}^{2}\!)
\!+\!C(\varepsilon) (1\!+\!t\|\nabla u_{2t}\|_{\!L^{2}(\mathbb{R}^{2})}^{2}\!+t\|\nabla^{2}u_{2}\|_{\!L^{q}(\mathbb{R}^{2})}^{2})\!\!\int_{0}^{t} \! (\|\varrho_{1}^{\frac{1}{2}} \widetilde{u}\|_{\!L^{2}(\mathbb{R}^{2})}^{2}\!+\|\nabla \widetilde{u}\|_{\!L^{2}(\mathbb{R}^{2})}^{2})\text{d}\tau.
\end{align*}
For $II_{2}$ and $II_{5}$, we derive from H\"{o}lder's and Gagliardo-Nirenberg inequalities, and \eqref{eq1.8} that
\begin{align*}
II_{2}+II_{5}\leq& C(\|\nabla d_{1}\|_{L^{4}(\mathbb{R}^{2})}+\|\nabla d_{2}\|_{L^{4}(\mathbb{R}^{2})}) \|\nabla \widetilde{d}\|_{L^{4}(\mathbb{R}^{2})}( \|\nabla \widetilde{u}\|_{L^{2}(\mathbb{R}^{2})}+\|\nabla^{2} \widetilde{d}\|_{L^{2}(\mathbb{R}^{2})})\nonumber\\
\leq& \varepsilon(\|\nabla \widetilde{u}\|_{L^{2}(\mathbb{R}^{2})}^{2}+\|\nabla^{2}\widetilde{d}\|_{L^{2}(\mathbb{R}^{2})}^{2})
+C(\varepsilon) \|\nabla \widetilde{d}\|_{L^{2}(\mathbb{R}^{2})}^{2}.
\end{align*}
By using  H\"{o}lder's and Gagliardo-Nirenberg inequalities, and Lemma \ref{lem2.3} again, we can estimate $II_{3}$ as
\begin{align*}
II_{3}\leq& C\|\widetilde{u}\bar{x}^{-\frac{a}{2}}\|_{L^{4}(\mathbb{R}^{2})} \|\nabla d_{1}\bar{x}^{\frac{a}{2}}\|_{L^{4}(\mathbb{R}^{2})}
\|\nabla^{2}\widetilde{d}\|_{L^{2}(\mathbb{R}^{2})}\nonumber\\
\leq&\varepsilon\|\nabla^{2}\widetilde{d}\|_{L^{2}(\mathbb{R}^{2})}^{2}
+C(\varepsilon)(1+\|\nabla^{2}d_{1}\bar{x}^{\frac{a}{2}}\|_{L^{2}(\mathbb{R}^{2})}^{2})
(\|\varrho_{1}^\frac{1}{2}\widetilde{u}\|_{L^{2}(\mathbb{R}^{2})}^{2}+\|\nabla \widetilde{u}\|_{L^{2}(\mathbb{R}^{2})}^{2}).
\end{align*}
The term $II_{4}$ can be estimated as follows
\begin{align*}
II_{4}\leq& C \|\nabla d_{2}\|_{L^{8}(\mathbb{R}^{2})}^{2}\|\widetilde{d}\|_{L^{4}(\mathbb{R}^{2})}
\|\nabla^{2}\widetilde{d}\|_{L^{2}(\mathbb{R}^{2})}
\leq \varepsilon \|\nabla^{2}\widetilde{d}\|_{L^{2}(\mathbb{R}^{2})}^{2}+ C(\varepsilon) \|\widetilde{d}\|_{H^{1}(\mathbb{R}^{2})}^{2},
\end{align*}
owing to H\"{o}lder's and  Gagliardo-Nirenberg inequalities, and \eqref{eq1.8}. Inserting these estimates $II_{i} (i=1,2,\cdots,5)$ into \eqref{eq3.48},  adding the resulting inequality with  \eqref{eq3.47}, and then choosing $\varepsilon$ suitably small lead to
\begin{align}\label{eq3.49}
G'(t)\leq C(1+\|\nabla u_{2}\|_{L^{\infty}(\mathbb{R}^{2})}+\|\nabla^{2}d_{1}\bar{x}^{\frac{a}{2}}\|_{L^{2}(\mathbb{R}^{2})}^{2}+t\|\nabla u_{2t}\|_{L^{2}(\mathbb{R}^{2})}^{2}\!+t\|\nabla^{2}u_{2}\|_{L^{q}(\mathbb{R}^{2})}^{2})G(t)
\end{align}
where $G(t)$ is defined as
\begin{align*}
G(t)\triangleq \|\varrho_{1}^{\frac{1}{2}} \widetilde{u}\|_{L^{2}(\mathbb{R}^{2})}^{2}+\|\nabla \widetilde{d}\|_{H^{1}(\mathbb{R}^{2})}^{2}+\int_{0}^{t}(\|\nabla \widetilde{u}\|_{L^{2}(\mathbb{R}^{2})}^{2}+\|\nabla^{2}\widetilde{d}\|_{L^{2}(\mathbb{R}^{2})}^{2}+\|\varrho_{1}^{\frac{1}{2}}\widetilde{u}\|_{L^{2}}^{2})
\text{d}\tau.
\end{align*}
The inequality \eqref{eq3.49} together with Gronwall's inequality and \eqref{eq1.8} implies that $G(t)=0$ for all $0< t\leq T_{0}$. Hence, we have $\widetilde{u}=0$ and $\widetilde{d}=0$ for almost everywhere $(x,t)\in \mathbb{R}^{2}\times(0,T_{0}]$. Finally, one can deduce from \eqref{eq3.46} that $\widetilde{\varrho}=0$ for almost everywhere $(x,t)\in \mathbb{R}^{2}\times(0,T_{0}]$. This completes the proof of Theorem \ref{thm3.1}.
\hfill$\Box$

\section{Proof of Theorem \ref{thm1.2}}

Before going to do it, let us  recall the following rigidity theorem, which was recently established in \cite{LLZ}.

\begin{theorem}\label{thm5.1}
Let $\varepsilon_{0}>0$ as \eqref{eq1.3}, and $M_{0}>0$. There exists a positive constant $\bar{\omega}=\bar{\omega}(\varepsilon_{0},M_{0})\in (0,1)$ such that the following holds:\\
If $d:\mathbb{R}^{2}\rightarrow \mathbb{S}^{1}$,  $\nabla d \in H^{1}(\mathbb{R}^{2})$ with
$\|\nabla d\|_{L^{2}}\leq M_{0}$ and $d_{2}\geq \varepsilon_{0}$, then
\begin{align*}
\|\nabla d\|_{L^{4}(\mathbb{R}^{2})}^{4}\leq (1-\bar{\omega})\|\Delta d\|_{L^{2}(\mathbb{R}^{2})}^{2}.
\end{align*}
Consequently for such maps the associated harmonic energy is coercive, i.e.,
\begin{align}\label{eq5.1}
\|\Delta d + |\nabla d|^{2} d\|_{L^{2}(\mathbb{R}^{2})}^{2}\geq \frac{\bar{\omega}}{2} \left(\|\Delta d\|_{L^{2}(\mathbb{R}^{2})}^{2}+\|\nabla d\|_{L^{4}(\mathbb{R}^{2})}^{4}\right).
\end{align}
\end{theorem}

In what follows, for $p\in[1,\infty]$ and $k\geq 0$, we denote
\begin{align*}
\int f\text{d}x=\int_{\mathbb{R}^{2}} f\text{d}x, \quad L^{p}=L^{p}(\mathbb{R}^{2}),\quad W^{1,p}=W^{1,p}(\mathbb{R}^{2}), \quad H^{k}=W^{k,2},
\end{align*}
and denote by $\|\cdot\|_{X}$
 the norm of the $X(\mathbb{R}^{2})$-functions, for simplicity.

\subsection{Lower order estimates}

In this subsection, we shall establish some necessary a priori lower order estimates for strong solutions $(\varrho, u, P, d)$ to the Cauchy problem of system \eqref{eq1.1}--\eqref{eq1.2}. Let $T>0$ be a fixed time and $(\varrho, u, P, d)$ be the strong solution to system \eqref{eq1.1}--\eqref{eq1.2} on $\mathbb{R}^{2}\times (0,T]$ with initial data $(\varrho_{0}, u_{0},d_{0})$ satisfying \eqref{eq1.3} and \eqref{eq1.5}--\eqref{eq1.7}.

\begin{lemma}\label{lem4.2}
There exists a positive constant $C$  depending only on $\varepsilon_{0}$, $\|\varrho_{0}\|_{L^{1}\cap L^{\infty}}$, $\|\varrho^{\frac{1}{2}}_{0} u_{0}\|_{L^{2}}$,
$\|\nabla u_{0}\|_{L^{2}}$, $\|\nabla d_{0}\|_{L^{2}}$ and $\|\nabla^{2} d_{0}\|_{L^{2}}$ such that
\begin{align}\label{eq5.2}
&\sup_{0\leq t\leq T}\left(\|\varrho\|_{L^{1}\cap L^{\infty}}+\|\nabla u\|_{L^{2}}^{2} +\|\nabla^{2} d\|_{L^{2}}^{2}\right)
+\int_{0}^{T} \left( \|\varrho^{\frac{1}{2}} \dot{u}\|_{L^{2}}^{2} +\|\nabla^{3} d\|_{L^{2}}^{2}\right)\text{d}t\leq C,
\end{align}
where $\dot{u}\triangleq u_{t}+u\cdot\nabla u$. Furthermore,  we have
\begin{align}\label{eq5.3}
&\sup_{0\leq t\leq T} t\left( \|\nabla u\|_{L^{2}}^{2} +\|\nabla^{2} d\|_{L^{2}}^{2}\right)
+\int_{0}^{T} t\left( \|\varrho^{\frac{1}{2}} \dot{u}\|_{L^{2}}^{2} +\|\nabla^{3} d\|_{L^{2}}^{2}\right)\text{d}t\leq C.
\end{align}
\end{lemma}

\begin{proof}
We first notice that, the  basic energy law \eqref{eq3.3} and the identity \eqref{eq3.6} still holds in $\mathbb{R}^{2}$. Under the assumption \eqref{eq1.3}, then it follows from the standard maximum principle method that $d_{3}\geq \varepsilon_{0}$ (see \cite{LW1}).
Hence, \eqref{eq1.3} and \eqref{eq3.3} together with \eqref{eq5.1} in Theorem \ref{thm5.1},  it follows that for all $t\geq 0$
\begin{align}\label{eq5.4}
&\sup_{0\leq \tau\leq t} \left(
\|\varrho^{\frac{1}{2}} u\|_{L^{2}}^{2}+\|\nabla d\|_{L^{2}}^{2}\right)
+\frac{\bar{\omega}}{2}\int_{0}^{t}\left( \|\nabla u\|_{L^{2}}^{2}+\|\Delta d\|_{L^{2}}^{2} +\|\nabla d\|_{L^{4}}^{4}\right)\text{d}\tau
\leq \|\varrho_{0}^{\frac{1}{2}} u_{0}\|_{L^{2}}^{2}+\|\nabla d_{0}\|_{L^{2}}^{2}.
\end{align}

Now,   multiplying  $\eqref{eq1.1}_{2}$ by $\dot{u}$,   and integrating
over $\mathbb{R}^{2}$, it follows that
\begin{align}\label{eq4.5}
\int \varrho |\dot{u}|^{2}\text{d}x=&\int \Delta u\cdot\dot{u}\text{d}x -\int \nabla P\cdot\dot{u}\text{d}x-\int \operatorname{div}(\nabla d\odot\nabla d) \cdot\dot{u}\text{d}x\nonumber\\
\triangleq& \overline{I}_{1}+\overline{I}_{2}+\overline{I}_{3}.
\end{align}
By using the definition of $\dot{u}$, and the Gagliardo-Nirenberg inequality, we have
\begin{align*}
\overline{I}_{1}=&\int \Delta u\cdot(u_{t}+u\cdot\nabla u)\text{d}x=-\int \nabla u\cdot(\nabla u_{t}+\nabla (u\cdot\nabla u))\text{d}x\nonumber\\
\leq& -\!\frac{1}{2}\frac{d}{dt} \|\nabla u\|_{L^{2}}^{2}-\sum_{i,j,k=1}^{2}\int \partial_{i} u_{j}\partial_{i}u_{k}\partial_{k}u_{j} \text{d}x\nonumber\\
\leq&-\!\frac{1}{2}\frac{d}{dt} \|\nabla u\|_{L^{2}}^{2} +\|\nabla u\|_{L^{3}}^{3}\leq -\!\frac{1}{2}\frac{d}{dt} \|\nabla u\|_{L^{2}}^{2} +C\|\nabla u\|_{L^{2}}^{2} \|\nabla^{2} u\|_{L^{2}},
\end{align*}
where we have used the divergence free condition $\eqref{eq1.1}_{4}$ in the first inequality. By integrating by parts, and then using the divergence free condition $\eqref{eq1.1}_{4}$ and the duality of $\mathcal{H}^{1}(\mathbb{R}^{2})$ and $BMO(\mathbb{R}^{2})$ (see \cite{Stein}, Charpter IV), it follows that
\begin{align*}
\overline{I}_{2}=&-\int \nabla P\cdot(u_{t}+u\cdot\nabla u)\text{d}x =\int P \nabla(u\cdot\nabla u)\text{d}x\nonumber\\
=&\sum_{i,j=1}^{2}\int P \partial_{i} u_{j}\partial_{j} u_{i}\text{d}x\leq C\sum_{i=1}^{2}\|P\|_{BMO} \|\partial_{i}u\cdot\nabla u_{i}\|_{\mathcal{H}^{1}}.
\end{align*}
Notice that $\operatorname{div}(\partial_{i} u)=\partial_{i} \operatorname{div}u =0$ and $\nabla^{\perp}\cdot(\nabla u_{i})=0$ for $i=1,2$. Hence, the above inequality together with Lemma \ref{lem2.6} ensures that
\begin{align*}
\overline{I}_{2}\leq C \sum_{i=1}^{2}\|P\|_{BMO} \|\partial_{i}u\cdot\nabla u_{i}\|_{\mathcal{H}^{1}}\leq C \|\nabla P\|_{L^{2}}\|\nabla u\|_{L^{2}}^{2}.
\end{align*}
To bound the term $\overline{I}_{3}$, integrating by parts together with \eqref{eq3.8}, H\"{o}lder's  and Gagliardo-Nirenberg inequalities gives
\begin{align*}
\overline{I}_{3}=&\int (\nabla d\odot\nabla d)\cdot\nabla u_{t}\text{d}x -\int \operatorname{div}(\nabla d\odot\nabla d)\cdot(u\cdot\nabla u)\text{d}x\nonumber\\
=& \frac{d}{dt}\!\int \!(\nabla d\odot\nabla d)\cdot\nabla u\text{d}x -\!\!\int\! (\nabla d_{t}\odot\nabla d)\cdot\!\nabla u\text{d}x -\!\!\int\! (\nabla d\odot\nabla d_{t})\cdot\!\nabla u\text{d}x
-\!\!\int\!\operatorname{div}(\nabla d\odot\nabla d)\cdot\!(u\cdot\!\nabla u)\text{d}x\nonumber\\
=& \frac{d}{dt}\!\int \!(\nabla d\odot\nabla d)\cdot\!\nabla u\text{d}x -\!\!\int\! [(\Delta\nabla d-\nabla(u\cdot \nabla d)+\nabla(|\nabla d|^{2}d))\odot\nabla d]\cdot\!\nabla u\text{d}x\nonumber\\
&-\int\! [\nabla d\odot(\Delta\nabla d-\nabla(u\cdot \nabla d)+\nabla(|\nabla d|^{2}d))]\cdot\nabla u\text{d}x
-\!\!\int\!\operatorname{div}(\nabla d\odot\nabla d)\cdot(u\cdot\nabla u)\text{d}x\nonumber\\
=& \frac{d}{dt}\!\int \!(\nabla d\odot\!\nabla d)\cdot\!\nabla u\text{d}x -\!\!\int\! [(\Delta\nabla d-\!\nabla u\!\cdot\! \nabla d+\!\nabla(|\nabla d|^{2}d))\odot\!\nabla d]\cdot\!\nabla u\text{d}x\!+\!\!\sum_{i,j,k=1}^{2}\!\!\int\! u_{k}\partial_{k}\partial_{i}d\cdot \partial_{j}d \partial_{j}u_{i}\text{d}x\nonumber\\
&-\int\! [\nabla d\odot(\Delta\nabla d-\nabla u\cdot \nabla d+\nabla(|\nabla d|^{2}d))]\cdot\nabla u\text{d}x
+\!\sum_{i,j,k=1}^{2}\int u_{k}\partial_{i}d\cdot \partial_{k} \partial_{j}d \partial_{j} u_{i}\text{d}x\nonumber\\
&+\!\sum_{i,j,k=1}^{2}\!\int\!\partial_{i} d \cdot\partial_{j} d \partial_{j}(u_{k}\partial_{k} u_{i})\text{d}x\nonumber\\
=& \frac{d}{dt}\!\int \!(\nabla d\odot\nabla d)\cdot\nabla u\text{d}x - \!\!\int\! [(\Delta\nabla d-\!\nabla u\cdot \nabla d+\!\nabla(|\nabla d|^{2}d))\odot\!\nabla d]\cdot\nabla u\text{d}x\nonumber\\
&-\int\! [\nabla d\odot(\Delta\nabla d-\nabla u\cdot \nabla d+\nabla(|\nabla d|^{2}d))]\cdot\nabla u\text{d}x
+\!\sum_{i,j,k=1}^{2}\int \partial_{i}d\cdot \partial_{j}d \partial_{j} u_{k}\partial_{k} u_{i}\text{d}x\nonumber\\
\leq& \frac{d}{dt}\!\int \!(\nabla d\odot\nabla d)\cdot\nabla u\text{d}x +C\|\nabla^{3} d\|_{L^{2}} \|\nabla d\|_{L^{6}} \|\nabla u\|_{L^{3}}
+C\|\nabla d\|_{L^{6}}^{2}\|\nabla u\|_{L^{3}}^{3}\nonumber\\
&+C\|\nabla d\|_{L^{6}}^{4} \|\nabla u\|_{L^{3}} +C\|\nabla^{2}d\|_{L^{3}}\|\nabla d\|_{L^{6}}^{2} \|\nabla u\|_{L^{3}}\nonumber\\
\leq& \frac{d}{dt}\!\int \!(\nabla d\odot\nabla d)\cdot\nabla u\text{d}x +\frac{\varepsilon}{4}\|\nabla^{3} d\|_{L^{2}}^{2}+C\|\nabla u\|_{L^{3}}^{3}+C\|\nabla d\|_{L^{6}}^{6}+C\|\nabla^{2} d\|_{L^{3}}^{\frac{3}{2}} \|\nabla d\|_{L^{6}}^{3}\quad (\text{with } \varepsilon>0)\nonumber\\
\leq & \frac{d}{dt}\!\int \!(\nabla d\odot\nabla d)\cdot\nabla u\text{d}x +\frac{\varepsilon}{2}\|\nabla^{3} d\|_{L^{2}}^{2}+C\|\nabla u\|_{L^{3}}^{3}+C\|\nabla d\|_{L^{6}}^{6}\nonumber\\
\leq & \frac{d}{dt}\!\int \!(\nabla d\odot\nabla d)\cdot\nabla u\text{d}x +\frac{\varepsilon}{2}\|\nabla^{3} d\|_{L^{2}}^{2}+C\|\nabla u\|_{L^{2}}^{2}
\|\nabla^{2} u\|_{L^{2}}+C\|\nabla d\|_{L^{2}}^{2}\|\nabla^{2} d\|_{L^{2}}^{4}.
\end{align*}
Inserting the estimates of $\overline{I}_{i} (i=1,2,3)$ into \eqref{eq4.5}, and then using \eqref{eq5.4},  it follows that
\begin{align}\label{eq4.6}
&\frac{d}{dt}\left(\frac{1}{2}\|\nabla u\|_{L^{2}}^{2}-\int(\nabla d\odot\nabla d)\cdot\nabla u\text{d}x\right)+\|\varrho^{\frac{1}{2}} \dot{u}\|_{L^{2}}^{2}\nonumber\\
\leq& \frac{\varepsilon}{2}\|\nabla^{3} d\|_{L^{2}}^{2}+C \|\nabla^{2} d\|_{L^{2}}^{4}
+C(\|\nabla^{2} u\|_{L^{2}}+\|\nabla P\|_{L^{2}})\|\nabla u\|_{L^{2}}^{2}.
\end{align}

Notice that $(\varrho, u, P, d)$ satisfies the following Stokes system
\begin{align*}
\begin{cases}
-\Delta u+\nabla P =-\varrho\dot{u} -\operatorname{div}(\nabla d\odot\nabla d),\quad& x\in\mathbb{R}^{2},\\
\operatorname{div}u=0,\quad& x\in\mathbb{R}^{2}\\
u(x)\rightarrow 0,\quad& |x|\rightarrow \infty.
\end{cases}
\end{align*}
Applying the standard $L^{p}$-estimate to the above system (see \cite{Temam}) ensures that for any $p\in(1,\infty)$
\begin{align}\label{eq4.7}
\|\nabla^{2} u\|_{L^{p}}+\|\nabla P\|_{L^{p}}\leq C\|\varrho \dot{u}\|_{L^{p}}+C\||\nabla d||\Delta d|\|_{L^{p}}
\leq C\|\varrho^{\frac{1}{2}} \dot{u}\|_{L^{p}}+C\||\nabla d||\nabla^{2} d|\|_{L^{p}},
\end{align}
where we have used the identity $\operatorname{div}(\nabla d\cdot\nabla d)=\nabla d\cdot\Delta d$ and \eqref{eq3.6}. Combining \eqref{eq4.6} and \eqref{eq4.7} together, it follows that
\begin{align}\label{eq4.08}
\frac{d}{dt} B(t)\! +\!\|\varrho^{\frac{1}{2}}\dot{u}\|_{L^{2}}^{2}
\leq&\frac{\varepsilon}{2}\|\nabla^{3} d\|_{L^{2}}^{2}+C\|\nabla^{2} d\|_{L^{2}}^{4} +
C\|\nabla u\|_{L^{2}}^{4} +\frac{\varepsilon}{2} (\|\varrho^{\frac{1}{2}} \dot{u}\|_{L^{2}}^{2}+\||\nabla d||\nabla^{2}d|\|_{L^{2}}^{2})\nonumber\\
\leq&\frac{\varepsilon}{4}\|\nabla^{3} d\|_{L^{2}}^{2}+C\|\nabla^{2} d\|_{L^{2}}^{4} +
C\|\nabla u\|_{L^{2}}^{4} +\frac{\varepsilon}{2} (\|\varrho^{\frac{1}{2}} \dot{u}\|_{L^{2}}^{2}+\|\nabla d\|_{L^{2}} \|\nabla^{2}d\|_{L^{2}}^{2}\|\nabla^{3}d|\|_{L^{2}})\nonumber\\
\leq&  \varepsilon \|\nabla^{3} d\|_{L^{2}}^{2}+\varepsilon \|\varrho^{\frac{1}{2}} \dot{u}\|_{L^{2}}^{2}
+C\|\nabla^{2} d\|_{L^{2}}^{4} +
C\|\nabla u\|_{L^{2}}^{4},
\end{align}
where
\begin{align*}
B(t)=\frac{1}{2} \|\nabla u\|_{L^{2}}^{2}-\int (\nabla d\cdot\nabla d)\cdot\nabla u\text{d}x,
\end{align*}
satisfies
\begin{align*}
\frac{1}{4} \|\nabla u\|_{L^{2}}^{2} -C_{1} \|\nabla^{2} d\|_{L^{2}}^{2}\leq B(t)\leq C\|\nabla u\|_{L^{2}}^{2}+C\|\nabla^{2} d\|_{L^{2}}^{2}
\end{align*}
owing to Gagliardo-Nirenberg inequality, \eqref{eq5.4} and the following estimate
\begin{align*}
\int (\nabla d\cdot\nabla d)\cdot\nabla u\text{d}x\leq &\frac{1}{4} \|\nabla u\|_{L^{2}}^{2}+C\|\nabla d\|_{L^{4}}^{4} \leq \frac{1}{4} \|\nabla u\|_{L^{2}}^{2}+C\|\nabla d\|_{L^{2}}^{2} \|\nabla^{2}d\|_{L^{2}}^{2}\nonumber\\
\leq &\frac{1}{4} \|\nabla u\|_{L^{2}}^{2}+C \|\nabla^{2}d\|_{L^{2}}^{2}. \quad (\text{by } \eqref{eq5.4})
\end{align*}

Now, multiplying $\eqref{eq1.1}_{3}$ by $-\nabla\Delta d$ and then integrating by parts over $\mathbb{R}^{2}$, it follows from H\"{o}lder's and Gagliardo-Nirenberg inequalities, \eqref{eq5.4} and \eqref{eq4.7} that
\begin{align}\label{eq4.09}
\frac{1}{2}\frac{d}{dt}\|\nabla^{2} d\|_{L^{2}}^{2}&+\|\nabla^{3}d\|_{L^{2}}^{2}=\int \nabla(u\cdot\nabla d)\cdot\nabla\Delta d\text{d}x
-\int \nabla (|\nabla d|^{2} d)\nabla\Delta d\text{d}x\nonumber\\
=&\int (\nabla u\cdot\nabla d)\cdot \nabla\Delta d\text{d}x+\!\sum_{i,j,k=1}^{2}\!\int u_{i}\partial_{i}\partial_{j} d\cdot\partial_{i}\partial_{k}^{2}d\text{d}x- \int \nabla (|\nabla d|^{2} d)\nabla\Delta d\text{d}x\nonumber\\
=&\int (\nabla u\cdot\nabla d)\cdot \nabla\Delta d\text{d}x-\!\sum_{i,j,k=1}^{2}\!\int \partial_{k} u_{i} \partial_{i}\partial_{j} d\cdot
\partial_{i} \partial_{k} d \text{d}x-\int \nabla(|\nabla d|^{2}d)\cdot\nabla\Delta d\text{d}x\nonumber\\
\leq& C\!\left(\|\nabla^{3} d\|_{\!L^{2}} \|\nabla u\|_{\!L^{3}} \|\nabla d\|_{\!L^{6}}\! +\|\nabla u\|_{\!L^{3}}\|\nabla^{2} d\|_{\!L^{3}}^{2}\!
+
\|\nabla^{3} d\|_{\!L^{2}} \|\nabla d\|_{\!L^{6}}^{3}\! +\|\nabla^{3} d\|_{\!L^{2}}\|\nabla^{2}d\|_{\!L^{3}} \|\nabla d\|_{\!L^{6}}\right)\nonumber\\
\leq&\frac{\varepsilon}{4} \|\nabla^{3} d\|_{L^{2}}^{2} +C\|\nabla u\|_{L^{3}}^{3}+C\|\nabla d\|_{L^{6}}^{6}+C\|\nabla^{2}d\|_{L^{3}}^{3}+C\|\nabla^{2} d\|_{L^{3}}^{2}\|\nabla d\|_{L^{6}}^{2}\nonumber\\
\leq&\frac{\varepsilon}{2} \|\nabla^{3} d\|_{L^{2}}^{2} +C\|\nabla u\|_{L^{3}}^{3}+C\|\nabla d\|_{L^{6}}^{6}+C\|\nabla^{2} d\|_{L^{2}}^{4}\nonumber\\
\leq&\frac{\varepsilon}{2} \|\nabla^{3} d\|_{L^{2}}^{2} +C\|\nabla u\|_{L^{2}}^{2}\|\nabla^{2} u\|_{L^{2}}
+C \|\nabla d\|_{L^{2}}^{2}\|\nabla^{2}d \|_{L^{2}}^{4}+C\|\nabla^{2} d\|_{L^{2}}^{4}\nonumber\\
\leq& \frac{\varepsilon}{2} \|\nabla^{3} d\|_{L^{2}}^{2}+\frac{\varepsilon}{2} (\|\varrho^{\frac{1}{2}} \dot{u}\|_{L^{2}}^{2}+\||\nabla d||\nabla^{2}d|\|_{L^{2}}^{2})+C\|\nabla^{2} d\|_{L^{2}}^{4} +
C\|\nabla u\|_{L^{2}}^{4}\nonumber\\
\leq &  \varepsilon\|\nabla^{3} d\|_{L^{2}}^{2}+\varepsilon \|\varrho^{\frac{1}{2}} \dot{u}\|_{L^{2}}^{2}
+C\|\nabla^{2} d\|_{L^{2}}^{4} +
C\|\nabla u\|_{L^{2}}^{4}.
\end{align}
Multiplying \eqref{eq4.09} by $(C_{1}+1)$, then adding the resulting inequality with \eqref{eq4.08} and choosing $\varepsilon$ suitably small, we obtain
\begin{align}\label{eq4.10}
&\frac{d}{dt}\left(B(t)+(C_{1}+1)\|\nabla^{2} d\|_{L^{2}}^{2}\right)+\|\varrho^{\frac{1}{2}} \dot{u}\|_{L^{2}}^{2}+\|\nabla^{3}d\|_{L^{2}}^{2}\nonumber\\
\leq& C(\|\nabla^{2} d\|_{L^{2}}^{4} +
\|\nabla u\|_{L^{2}}^{4})\leq C(\|\nabla^{2} d\|_{L^{2}}^{2} +
\|\nabla u\|_{L^{2}}^{2})(B(t)+(C_{1}+1)\|\nabla^{2} d\|_{L^{2}}^{2}).
\end{align}
The above estimate \eqref{eq4.10} combined with energy inequality \eqref{eq5.4}, the definition of $B(t)$, and Gronwall's inequality ensures
\begin{align*}
&\sup_{0\leq t\leq T}\left(\|\nabla u\|_{L^{2}}^{2} +\|\nabla^{2} d\|_{L^{2}}^{2}\right)
+\int_{0}^{T} \left( \|\varrho^{\frac{1}{2}} \dot{u}\|_{L^{2}}^{2} +\|\nabla^{3} d\|_{L^{2}}^{2}\right)\text{d}t\leq C.
\end{align*}
This together with \eqref{eq3.6} ensures \eqref{eq5.2}.

Finally,  multiplying \eqref{eq4.10} by $t$, and then applying Gronwall's inequality to the resulting inequality, it follows from \eqref{eq5.4} and the definition of $B(t)$ yields \eqref{eq5.3}. This completes the proof of Lemma \ref{lem4.2}.
\end{proof}

\begin{lemma}\label{lem4.3}
There exists a positive constant $C$  depending only on  $\varepsilon_{0}$, $\|\varrho_{0}\|_{L^{1}\cap L^{\infty}}$, $\|\varrho^{\frac{1}{2}}_{0} u_{0}\|_{L^{2}}$,
$\|\nabla u_{0}\|_{L^{2}}$, $\|\nabla d_{0}\|_{L^{2}}$ and $\|\nabla^{2} d_{0}\|_{L^{2}}$ such that for $i=1,2$
\begin{align}\label{eq4.11}
&\sup_{0\leq t\leq T}t^{i}\left(\|\varrho^{\frac{1}{2}} \dot{u}\|_{L^{2}}^{2} +\||\nabla d||\nabla^{2} d|\|_{L^{2}}^{2}\right)
+\int_{0}^{T} \left( \|\nabla \dot{u}\|_{L^{2}}^{2} +\||\nabla d||\Delta\nabla d| \|_{L^{2}}^{2}\right)\text{d}t\leq C,
\end{align}
and
\begin{align}\label{eq4.12}
&\sup_{0\leq t\leq T} t^{i}\left( \|\nabla^{2} u\|_{L^{2}}^{2} +\|\nabla P\|_{L^{2}}^{2}\right)
\leq C.
\end{align}
\end{lemma}

\begin{proof}
Operating $\partial_{t}+u\cdot\nabla $ to $\eqref{eq1.1}_{2j}$ ($j=1,2$), one obtains by some simple calculations that
\begin{align*}
&\partial_{t}(\varrho \dot{u}_{j})+\operatorname{div}(\varrho u \dot{u}_{j}) -\Delta \dot{u}_{j}\nonumber\\
=&\sum_{i=1}^{2}\Big(-\partial_{i} (\partial_{i} u\cdot\nabla u_{j})-\operatorname{div} (\partial_{i} u\partial_{i}u_{j})  -\partial_{t}\partial_{i}(\partial_{i}d \cdot\partial_{j} d) -
u\cdot\nabla \partial_{i}(\partial_{i}d \cdot \partial_{j} d)\Big)-\partial_{t}\partial_{j} P -u\cdot\nabla\partial_{j} P.
\end{align*}
Multiplying the above equality by $\dot{u}_{j}$, and then integrating by parts over $\mathbb{R}^{2}$, it follows that
\begin{align}\label{eq4.13}
\frac{1}{2}\frac{d}{dt} \|\varrho^{\frac{1}{2}}\dot{u}\|_{L^{2}}^{2}&+ \|\nabla \dot{u}\|_{L^{2}}^{2}
=-\sum_{i,j=1}^{2}\int \left[\partial_{i}(\partial_{i} u\cdot\nabla u_{j})+\operatorname{div} (\partial_{i} u\partial_{i} u_{j})\right]\dot{u}_{j}\text{d}x\nonumber\\
&-\!\sum_{j=1}^{2}\!\int\! (\partial_{t}\partial_{j} P\!+u\cdot\nabla\partial_{j} P)\dot{u}_{j}\text{d}x-\!\sum_{i,j=1}^{2}\!\!
\left(\int \!\partial_{t}\partial_{i}(\partial_{i}d\cdot\partial_{j}d)\dot{u}_{j}\text{d}x
\!+\!\int\! u\cdot\nabla\partial_{i}(\partial_{i}d\cdot\partial_{j}d) \dot{u}_{j}\text{d}x\right)\nonumber\\
=& \overline{J}_{1}+\overline{J}_{2}+\overline{J}_{3}.
\end{align}
In what follows, we shall estimate each term on the right-hand side of \eqref{eq4.13} term by term. We first notice that exactly the same arguments as Lemma 3.3 of paper \cite{LSZ}, one has
\begin{align*}
\overline{J}_{1}+\overline{J}_{2}\leq \frac{d}{dt} \int \sum_{i,j=1}^{2} P\partial_{j} u_{i}\partial_{i}u_{j}\text{d}x +C(\|P\|_{L^{4}}^{4}+\|\nabla u\|_{L^{4}}^{4}) +\frac{1}{4} \|\nabla \dot{u}\|_{L^{2}}^{2}.
\end{align*}
For the term $\overline{J}_{3}$, by using $\eqref{eq1.1}_{3}$, $\eqref{eq1.1}_{4}$ and Gagliardo-Nirenberg inequality, it follows that
\begin{align*}
\overline{J}_{3}=& \sum_{i,j=1}^{2}\int \partial_{i}\dot{u}_{j} \partial_{i} d_{t}\cdot\partial_{j}d\text{d}x +\sum_{i,j=1}^{2}\int \partial_{i}\dot{u}_{j}\partial_{i}d\cdot\partial_{j}d_{t}\text{d}x-\sum_{i,j,k=1}^{2}\int u_{k}\partial_{k}\partial_{i}(\partial_{i}d\cdot\partial_{j}d)\dot{u}_{j}\text{d}x\nonumber\\
\!=&\!\sum_{i,j=1}^{2}\!\int\! \partial_{i}\dot{u}_{j}  \partial_{j}d \cdot\!\partial_{i} (\Delta d\!-\!u\cdot\! \nabla d\!+\!|\nabla d|^{2}d)\text{d}x\! +\!\sum_{i,j=1}^{2}\!\int\! \partial_{i}\dot{u}_{j}\partial_{i}d\cdot\!\partial_{j}(\Delta d\!-\!u\cdot\! \nabla d\!+\!|\nabla d|^{2}d)\text{d}x\nonumber\\
&-\sum_{i,j,k,\ell=1}^{2}\int u_{k}\partial_{k}\partial_{i}(\partial_{i}d\cdot\partial_{j}d)\dot{u}_{j}\text{d}x\nonumber\\
=&\sum_{i,j=1}^{2}\!\int\! \partial_{i}\dot{u}_{j}  \partial_{j}d\cdot  (\Delta \partial_{i} d-\partial_{i}u\cdot \nabla d+\partial_{i}(|\nabla d|^{2}d))\text{d}x-\sum_{i,j,k,\ell=1}^{2}\int \partial_{i}\dot{u}_{j} \partial_{j}d\cdot \partial_{k}\partial_{i}d u_{k}\text{d}x\nonumber\\
&+\!\sum_{i,j=1}^{2}\!\int\! \partial_{i}\dot{u}_{j}\partial_{i}d\cdot(\Delta\partial_{j} d-\partial_{j} u\cdot \nabla d+\partial_{j}(|\nabla d|^{2}d))\text{d}x-\sum_{i,j,k=1}^{2}\int \partial_{i}\dot{u}_{j} \partial_{i}d\cdot \partial_{k}\partial_{j}d u_{k}\text{d}x\nonumber\\
&+\sum_{i,j,k=1}^{2}\int \partial_{i} u_{k}\partial_{k}(\partial_{i}d\cdot\partial_{j}d)\dot{u}_{j}\text{d}x
+\sum_{i,j,k=1}^{2}\int u_{k}\partial_{k}(\partial_{i}d\cdot\partial_{j}d)\partial_{i} \dot{u}_{j}\text{d}x\nonumber\\
=&\sum_{i,j=1}^{2}\!\int\! \partial_{i}\dot{u}_{j}  \partial_{j}d \cdot (\Delta \partial_{i} d-\partial_{i}u\cdot \nabla d+\partial_{i}(|\nabla d|^{2}d))\text{d}x\nonumber\\
&+\!\sum_{i,j=1}^{2}\!\int\! \partial_{i}\dot{u}_{j}\partial_{i}d\cdot(\Delta\partial_{j} d-\partial_{j} u\cdot \nabla d+\partial_{j}(|\nabla d|^{2}d))\text{d}x +\sum_{i,j,k,\ell=1}^{2}\int \partial_{i} u_{k}\partial_{i}d\cdot\partial_{j}d \partial_{k}\dot{u}_{j}\text{d}x\nonumber\\
\leq& C\|\nabla\dot{u}\|_{L^{2}} \||\nabla d||\Delta \nabla d|\|_{L^{2}} +C\|\nabla \dot{u}\|_{L^{2}} \|\nabla u\|_{L^{4}} \||\nabla d|^{2}\|_{L^{4}}\nonumber\\
\leq &\frac{1}{4} \|\nabla \dot{u}\|_{L^{2}}^{2}+C\||\nabla d||\Delta \nabla d|\|_{L^{2}}^{2} +C\|\nabla u\|_{L^{4}}^{4}+C \||\nabla d|^{2}\|_{L^{4}}^{4}.
\end{align*}
Inserting the estimates of $\overline{J}_{i} (i=1,2,3)$ into \eqref{eq4.13}, one obtains
\begin{align}\label{eq4.14}
&\frac{1}{2}\frac{d}{dt} \|\varrho^{\frac{1}{2}}\dot{u}\|_{L^{2}}^{2}+\frac{1}{2} \|\nabla \dot{u}\|_{L^{2}}^{2}\nonumber\\
\leq& \frac{d}{dt} \int \sum_{i,j=1}^{2} P\partial_{j} u_{i}\partial_{i}u_{j}\text{d}x
+C(\|P\|_{L^{4}}^{4}+\|\nabla u\|_{L^{4}}^{4}+\||\nabla d|^{2}\|_{L^{4}}^{4})+C\||\nabla d||\Delta \nabla d|\|_{L^{2}}^{2}
\end{align}

Now, inspired by the papers \cite{LH15,LXZ,LXZ1}, for $a_{1},a_{2}\in \{-1,0,1\}$, let us denote
\begin{align*}
(\widetilde{\nabla} d)(a_{1},a_{2})=a_{1}\partial_{1}d+a_{2}\partial_{2}d, \quad (\widetilde{\nabla} u)(a_{1},a_{2})=a_{1}\partial_{1} u+a_{2}\partial_{2} u,\quad \widetilde{u}(a_{1},a_{2})=a_{1}u_{1}+a_{2} u_{2},
\end{align*}
then it is easy to deduce from \eqref{eq3.8} that
\begin{align*}
\widetilde{\nabla }d_{t} -\Delta\widetilde{\nabla } d=-\widetilde{\nabla}u\cdot\nabla d -u\cdot\nabla\widetilde{\nabla}d +|\nabla d|^{2}\widetilde{\nabla}d+ 2(\nabla d\cdot d)\nabla\widetilde{\nabla}d.
\end{align*}
Multiplying the above equality by $4\widetilde{\nabla}d\Delta|\widetilde{\nabla} d|^{2}$, and then integrating by parts over $\mathbb{R}^{2}$, it follows that
\begin{align}\label{eq4.15}
&\frac{d}{dt} \|\nabla |\widetilde{\nabla }d|^{2}\|_{L^{2}}^{2}+2\|\Delta |\widetilde{\nabla}d|^{2}\|_{L^{2}}^{2}\nonumber\\
=& -4\int (\widetilde{\nabla} u\cdot\nabla d)\cdot\widetilde{\nabla} d\Delta|\widetilde{\nabla}d|^{2}\text{d}x
-4\int (u\cdot\nabla \widetilde{\nabla}d)\cdot\widetilde{\nabla}d\Delta|\widetilde{\nabla }d|^{2}\text{d}x\nonumber\\
&+4\int |\nabla d|^{2}|\widetilde{\nabla}d|^{2} \Delta|\widetilde{\nabla}d|^{2}\text{d}x+8\int ( \nabla d\cdot d)\nabla\widetilde{\nabla} d\cdot\widetilde{\nabla}d\Delta|\widetilde{\nabla}d|^{2}\text{d}x\nonumber\\
\leq& C\int |\widetilde{\nabla} u| |\nabla d||\widetilde{\nabla}d||\Delta|\widetilde{\nabla}d|^{2}|\text{d}x
-2\int (u\cdot\nabla |\widetilde{\nabla }d|^{2})\Delta|\widetilde{\nabla}d|^{2}\text{d}x\nonumber\\
&+C\int |\nabla d|^{2}|\widetilde{\nabla}d|^{2} |\Delta|\widetilde{\nabla}d|^{2}|\text{d}x+C\int |\nabla d| \nabla|\widetilde{\nabla} d|^{2}|\Delta|\widetilde{\nabla}d|^{2}|\text{d}x\nonumber\\
\leq &C\|\nabla u\|_{L^{4}}\||\nabla d|^{2}\|_{L^{4}} \|\Delta|\widetilde{\nabla }d|^{2}\|_{L^{2}}
+C\|\nabla u\|_{L^{4}} \|\nabla |\widetilde{\nabla}d|^{2}\|_{L^{\frac{8}{3}}}^{2}\nonumber\\
&+C\||\nabla d|^{2}\|_{L^{4}}^{2}
\|\Delta|\widetilde{\nabla} d|^{2}\|_{L^{2}} +C\|\Delta|\widetilde{\nabla }d|^{2}\|_{L^{2}} \|\nabla|\widetilde{\nabla}d|^{2}\|_{L^{\frac{8}{3}}} \|\nabla d\|_{L^{8}}\nonumber\\
\leq& \frac{1}{2}\|\Delta|\widetilde{\nabla}d|^{2}\|_{L^{2}}^{2}+C(\||\nabla d|^{2}\|_{L^{4}}^{4}+\|\nabla u\|_{L^{4}}^{4})+C\|\nabla |\widetilde{\nabla} d|^{2}\|_{L^{\frac{8}{3}}}^{\frac{8}{3}} +C\|\nabla d\|_{L^{8}}^{8}\nonumber\\
\leq& \frac{1}{2}\|\Delta|\widetilde{\nabla}d|^{2}\|_{L^{2}}^{2}+C(\||\nabla d|^{2}\|_{L^{4}}^{4}+\|\nabla u\|_{L^{4}}^{4})+C\||\widetilde{\nabla} d|^{2}\|_{L^{4}}^{\frac{4}{3}}\|\Delta|\widetilde{\nabla} d|^{2}\|_{L^{2}}^{\frac{4}{3}} +C\|\nabla d\|_{L^{2}}^{4}\|\nabla^{2} d\|_{L^{4}}^{4}\nonumber\\
\leq& \|\Delta|\widetilde{\nabla}d|^{2}\|_{L^{2}}^{2}+C(\||\nabla d|^{2}\|_{L^{4}}^{4}+\|\nabla u\|_{L^{4}}^{4})+C\|\nabla^{2} d\|_{L^{4}}^{4},
\end{align}
where we have used $\eqref{eq1.1}_{4}$ H\"{o}lder's and  Gagliardo-Nirenberg inequalities, and \eqref{eq5.4} in the above estimates. Noticing that
\begin{align*}
\||\Delta\nabla d||\nabla d|\|_{L^{2}}^{2}\!\leq \! C\|\nabla^{2}d\|_{L^{4}}^{4}\!+\! \|\Delta|\widetilde{\nabla }d(1,0)|^{2}\|_{L^{2}}^{2}\!+\!\|\Delta|\widetilde{\nabla }d(0,1)|^{2}\|_{L^{2}}^{2}\!+\!\|\Delta|\widetilde{\nabla }d(1,1)|^{2}\|_{L^{2}}^{2}\!+\!\|\Delta|\widetilde{\nabla }d(1,-1)|^{2}\|_{L^{2}}^{2},
\end{align*}
and
\begin{align}\label{eq4.16}
\||\nabla^{2} d||\nabla d|\|_{L^{2}}^{2}\leq G(t)\leq C\||\nabla^{2} d||\nabla d|\|_{L^{2}}^{2}
\end{align}
with
\begin{align*}
G(t)\triangleq \|\nabla|\widetilde{\nabla}d(1,0)|^{2}\|_{L^{2}}^{2}+\|\nabla|\widetilde{\nabla}d(0,1)|^{2}\|_{L^{2}}^{2}
+\|\nabla|\widetilde{\nabla}d(1,1)|^{2}\|_{L^{2}}^{2}+\|\nabla|\widetilde{\nabla}d(1,-1)|^{2}\|_{L^{2}}^{2}.
\end{align*}
Thus, it follows from \eqref{eq4.15} multiplied by $(C_{2}+1)$ that
\begin{align*}
\frac{d}{dt} \left((C_{2}+1)G(t)\right) +(C_{2}+1)\||\Delta\nabla d||\nabla d|\|_{L^{2}}^{2}
\leq  C\|\nabla u\|_{L^{4}}^{4} +C\|\nabla^{2} d\|_{L^{4}}^{4}+C\||\nabla d|^{2}\|_{L^{4}}^{4},
\end{align*}
which combined with \eqref{eq4.14} ensures that
\begin{align}\label{eq4.17}
\frac{d}{dt}F(t)+\frac{1}{2}\|\nabla \dot{u}\|_{L^{2}}^{2}+\||\Delta\nabla d||\nabla d|\|_{L^{2}}^{2}
\leq C\|P\|_{L^{4}}^{4}+C\|\nabla u\|_{L^{4}}^{4} +C\|\nabla^{2} d\|_{L^{4}}^{4}+C\||\nabla d|^{2}\|_{L^{4}}^{4},
\end{align}
where
\begin{align*}
F(t)\triangleq \frac{1}{2}\|\varrho^{\frac{1}{2}}\dot{u}\|_{L^{2}}^{2}+(C_{2}+1) G(t)-\int \sum_{i,j=1}^{2} P\partial_{j}u_{i}\partial_{i}u_{j}\text{d}x
\end{align*}
satisfies
\begin{align}\label{eq4.18}
\frac{1}{4}\|\varrho^{\frac{1}{2}}\dot{u}\|_{L^{2}}^{2}+\frac{C_{2}+1}{2} G(t)-C\|\nabla u\|_{L^{4}}^{4}
\leq F(t)
\leq \|\varrho^{\frac{1}{2}}\dot{u}\|_{L^{2}}^{2}+C G(t)+C\|\nabla u\|_{L^{4}}^{4}
\end{align}
owing to the following estimate
\begin{align*}
\left|\int \sum_{i,j=1}^{2} P\partial_{j}u_{i}\partial_{i}u_{j}\text{d}x\right|
\leq & \sum_{i=1}^{2} C\|P\|_{BMO} \|\partial_{i}u\cdot\nabla u_{i}\|_{\mathcal{H}^{1}}\leq C\|\nabla P\|_{L^{2}}\|\nabla u\|_{L^{2}}^{2}\quad (\text{by Lemma \ref{lem2.6}})\nonumber\\
\leq &C(\|\varrho^{\frac{1}{2}}\dot{u}\|_{L^{2}} +\||\nabla^{2}d||\nabla d|\|_{L^{2}})\|\nabla u\|_{L^{2}}^{2}
\quad (\text{by} \eqref{eq4.7})\nonumber\\
\leq& \frac{1}{2}\|\varrho^{\frac{1}{2}}\dot{u}\|_{L^{2}}^{2} +\frac{C_{2}+1}{2} G(t)+C\|\nabla u\|_{L^{4}}^{4}.
\end{align*}

Next, we shall estimate the terms on the right-hand side of \eqref{eq4.17}. To bound the terms $\|P\|_{L^{4}}$ and $\|\nabla u\|_{L^{4}}$, it follows from Sobolev embedding, \eqref{eq4.7}, H\"{o}lder's inequality, \eqref{eq5.4}, \eqref{eq4.7}  and \eqref{eq4.18}  that
\begin{align}\label{eq4.19}
\|P\|_{L^{4}}^{4} +\|\nabla u\|_{L^{4}}^{4}\leq &C(\|\nabla P\|_{L^{\frac{4}{3}}}^{4}+\|\nabla^{2} u\|_{L^{\frac{4}{3}}}^{4})\leq C(\|\varrho\dot{u}\|_{L^{\frac{4}{3}}}^{4}+\||\nabla d||\nabla^{2} d|\|_{L^{\frac{4}{3}}}^{4})\nonumber\\
\leq&C\|\varrho\|_{L^{2}}^{2}\|\varrho^{\frac{1}{2}} \dot{u}\|_{L^{2}}^{4}+C\|\nabla d\|_{L^{2}}^{4}\|\nabla^{2} d\|_{L^{4}}^{4}\nonumber\\
\leq& C \|\varrho^{\frac{1}{2}} \dot{u}\|_{L^{2}}^{2}(F(t)+\|\nabla u\|_{L^{2}}^{4})
+C\|\nabla^{2} d\|_{L^{2}}^{2}\|\nabla^{3} d\|_{L^{2}}^{2}.
\end{align}
For the rest two terms $\|\nabla^{2} d\|_{L^{4}}$ and $\||\nabla d|^{2}\|_{L^{4}}$, by using Gagliardo-Nirenberg inequality, \eqref{eq5.4} and \eqref{eq4.18}, one has
\begin{align}\label{eq4.20}
\|\nabla^{2} d\|_{L^{4}}^{4}+\||\nabla d|^{2}\|_{L^{4}}^{4}\leq& C\|\nabla^{2}d\|_{L^{2}}^{2}\|\nabla^{3} d\|_{L^{2}}^{2}+C\|\nabla d\|_{L^{2}}^{2} \|\nabla^{2} d\|_{L^{2}}^{2} \||\nabla d||\nabla^{2}d|\|_{L^{2}}^{2}\nonumber\\
\leq& C\|\nabla^{2} d\|_{L^{2}}^{2}(F(t)+\|\nabla u\|_{L^{4}}^{4})+C\|\nabla^{2}d\|_{L^{2}}^{2}\|\nabla^{3} d\|_{L^{2}}^{2}.
\end{align}
Hence, inserting \eqref{eq4.19} and \eqref{eq4.20} into \eqref{eq4.17}, one obtains
\begin{align}\label{eq4.21}
&\frac{d}{dt} F(t) +\frac{1}{2}\|\nabla \dot{u}\|_{L^{2}}^{2}+\||\Delta\nabla d||\nabla d|\|_{L^{2}}^{2}\nonumber\\
\leq & C(\|\varrho^{\frac{1}{2}} \dot{u}\|_{L^{2}}^{2}+\|\nabla^{2} d\|_{L^{2}}^{2})(F(t)+\|\nabla u\|_{L^{2}}^{4})
+C\|\nabla^{2} d\|_{L^{2}}^{2}\|\nabla^{3} d\|_{L^{2}}^{2}.
\end{align}
Multiplying \eqref{eq4.21} by $t^{i}\ (i=1,2)$, and then applying Gronwall's inequality, it follows from \eqref{eq4.16}, \eqref{eq4.18}, \eqref{eq5.2}, \eqref{eq5.3} and \eqref{eq5.4} that
\begin{align*}
&\sup_{0\leq t\leq T} (t^{i}F(t))+\int_{0}^{T} t^{i}(\|\nabla \dot{u}\|_{L^{2}}^{2}+\||\Delta\nabla d||\nabla d|\|_{L^{2}}^{2})\text{d}t\nonumber\\
\leq& C\int_{0}^{T} t^{i-1} F(t)\text{d}t+C\int_{0}^{T} t^{i} \|\nabla^{2} d\|_{L^{2}}^{2}\|\nabla^{3} d\|_{L^{2}}^{2}\text{d}t+C\int_{0}^{T} (\|\varrho^{\frac{1}{2}} \dot{u}\|_{L^{2}}^{2}+\|\nabla^{2} d\|_{L^{2}}^{2})t^{i} \|\nabla u\|_{L^{4}}^{4}\text{d}t\nonumber\\
\leq& C\int_{0}^{T} t^{i-1} (\|\varrho^{\frac{1}{2}} \dot{u}\|_{L^{2}}^{2}+\||\nabla^{2}d||\nabla d|\|_{L^{2}}^{2})\text{d}t+C\sup_{0\leq t\leq T}(t^{i-1}\|\nabla u\|_{L^{2}}^{2})\int_{0}^{T} \|\nabla u\|_{L^{2}}^{2}\text{d}t\nonumber\\
&+C\!\sup_{0\leq t\leq T} (t^{i-1}\|\nabla^{2} d\|_{L^{2}}^{2})\int_{0}^{T} t\|\nabla^{3}d\|_{L^{2}}^{2}\text{d}t
+C\!\sup_{0\leq t\leq T}(t^{i} \|\nabla u\|_{L^{4}}^{4})\int_{0}^{T} (\|\varrho^{\frac{1}{2}} \dot{u}\|_{L^{2}}^{2}+\|\nabla^{2} d\|_{L^{2}}^{2})\text{d}t\nonumber\\
\leq& C\int_{0}^{T} t^{i-1} (\|\varrho^{\frac{1}{2}} \dot{u}\|_{L^{2}}^{2}+ \|\nabla d\|_{L^{2}}\|\nabla^{2} d\|_{L^{2}}^{2} \|\nabla^{3}d\|_{L^{2}})\text{d}t+C\nonumber\\
\leq& C\int_{0}^{T} t^{i-1} (\|\varrho^{\frac{1}{2}} \dot{u}\|_{L^{2}}^{2}+ \|\nabla^{3}d\|_{L^{2}}^{2})\text{d}t+C\sup_{0\leq t\leq T} (t^{i-1}\|\nabla^{2}d \|_{L^{2}}^{2})\int_{0}^{T} \|\nabla^{2}d\|_{L^{2}}^{2}\text{d}t+C\nonumber\\
\leq& C.
\end{align*}
This together with \eqref{eq4.16}, \eqref{eq4.18}, \eqref{eq5.2}, \eqref{eq5.3} and \eqref{eq5.4} ensures \eqref{eq4.11}.

 Finally, it is easy to see that the estimate \eqref{eq4.11} combined with \eqref{eq4.7} implies \eqref{eq4.12}. This completes the proof of Lemma \ref{lem4.3}.
\end{proof}

\subsection{Higher order estimates}
In this subsection, we shall give some higher order norms estimates to the solution $(\varrho, u, P, d)$. To this end, we first list the following spatial weighted estimates on the density.

\begin{lemma}\label{lem4.4}
There exists a positive constant $C$ depending on $T$ such that
\begin{align}\label{eq4.22}
\sup_{0\leq t\leq T} \|\varrho\bar{x}^{a}\|_{L^{1}}\leq C(T).
\end{align}
\end{lemma}

\begin{proof}
Similarly to \eqref{eq3.15}, we can prove \eqref{eq4.22} with some suitable revisions, and we omit it. Moreover, we notice that \eqref{eq3.14} still holds in $\mathbb{R}^{2}$.
\end{proof}

\begin{lemma}\label{lem4.5}
There exists a positive constant $C$ depending on $T$ such that
\begin{align}\label{eq4.23}
\sup_{0\leq t\leq T}\! \|\varrho\|_{H^{1}\cap W^{1,q}}+&\!\int_{0}^{T} \!\!\left(\|\nabla^{2} u\|_{L^{2}}^{2}\!+\!\|\nabla P\|_{L^{2}}^{2}\!+\!\|\nabla^{2} u\|_{L^{q}}^{\frac{q+1}{q}}\!\!+\!\|\nabla P\|_{L^{q}}^{\frac{q+1}{q}}\!\!+t(\|\nabla^{2} u\|_{L^{2}\cap L^{q}}^{2}\!+\!\|\nabla P\|_{L^{2}\cap L^{q}}^{2})\!\right)\text{d}t\nonumber\\
\leq& C(T).
\end{align}
\end{lemma}

\begin{proof}
We first notice that for any $r\geq 2$, one follows from $\eqref{eq1.1}_{1}$ and $\eqref{eq1.1}_{4}$ that
\begin{align}\label{eq4.24}
\frac{d}{dt} \|\nabla \varrho\|_{L^{r}}\leq C\|\nabla u\|_{L^{\infty}} \|\nabla \varrho\|_{L^{r}}.
\end{align}
Using Gagliardo-Nirenberg inequality, \eqref{eq5.2}, and \eqref{eq4.7}, it follows that
\begin{align}\label{eq4.25}
\|\nabla u\|_{L^{\infty}}\leq &C\|\nabla u\|_{L^{2}}^{\frac{q-2}{2(q-1)}} \|\nabla^{2} u\|_{L^{q}}^{\frac{q}{2(q-1)}}
\leq C\left( \|\varrho\dot{u}\|_{L^{q}}^{\frac{q}{2(q-1)}}+\||\nabla d||\nabla^{2}d|\|_{L^{q}}^{\frac{q}{2(q-1)}}\right),
\end{align}
for all $q\in (2,\infty)$. On the one hand, by using Gagliardo-Nirenberg and Lemma \ref{lem2.4}, one has
\begin{align*}
\|\varrho\dot{u}\|_{L^{q}} \leq& C\|\varrho \dot{u}\|_{L^{2}}^{\frac{2q-1}{q^{2}-1}} \|\varrho\dot{u}\|_{L^{2q^{2}}}^{\frac{q(q-2)}{q^{2}-1}}
\leq C\|\varrho^{\frac{1}{2}} \dot{u}\|_{L^{2}}^{\frac{2q-1}{q^{2}-1}} (\|\varrho^{\frac{1}{2}} \dot{u}\|_{L^{2}}+\|\nabla \dot{u}\|_{L^{2}})^{\frac{q(q-2)}{q^{2}-1}}\nonumber\\
\leq& C \|\varrho^{\frac{1}{2}} \dot{u}\|_{L^{2}}+C \|\varrho^{\frac{1}{2}} \dot{u}\|_{L^{2}}^{\frac{2q-1}{q^{2}-1}}\|\nabla \dot{u}\|_{L^{2}}^{\frac{q(q-2)}{q^{2}-1}}.
\end{align*}
From the above inequality, it is easy to see that
\begin{align}\label{eq4.26}
\int_{0}^{T}\!\!\!\left(\|\varrho\dot{u}\|_{L^{q}}^{\frac{q+1}{q}}\!\!+t\|\varrho\dot{u}\|_{L^{q}}^{2}\!\right)\!\text{d}t
\leq &C\!\!\int_{0}^{T}\! t^{-\frac{q+1}{2q}} \! \left(t\|\varrho^{\frac{1}{2}}\dot{u}\|_{L^{2}}^{2}\right)^{\frac{2q-1}{2q(q-1)}}
\left(t\|\nabla \dot{u}\|_{L^{2}}^{2}\right)^{\frac{q-2}{2(q-1)}}\text{d}t +C\!\int_{0}^{T}\! \|\varrho^{\frac{1}{2}} \dot{u}\|_{L^{2}}^{\frac{q+1}{q}}\text{d}t\nonumber\\
&+C\!\int_{0}^{T} \!\left(t\|\varrho^{\frac{1}{2}} \dot{u}\|_{L^{2}}^{2}\right)^{\frac{2q-1}{q^{2}-1}}
\left( t \|\nabla \dot{u}\|_{L^{2}}^{2}\right)^{\frac{q(q-2)}{q^{2}-1}}\text{d}t+C\int_{0}^{T} t\|\varrho^{\frac{1}{2}} \dot{u}\|_{L^{2}}^{2}\text{d}t\nonumber\\
\leq& C\sup_{0\leq t\leq T}\left(t\|\varrho^{\frac{1}{2}}\dot{u}\|_{L^{2}}^{2}\right)^{\frac{(2q-1)}{2q(q-1)}} \int_{0}^{T}\left(t\|\nabla \dot{u}\|_{L^{2}}^{2}+t^{-\frac{q^{3}+q^{2}-q-1}{q^{3}+q^{2}}}\right)\text{d}t\nonumber\\
&+C\int_{0}^{T}\left(t\|\varrho^{\frac{1}{2}} \dot{u}\|_{L^{2}}^{2}+t\|\nabla \dot{u}\|_{L^{2}}^{2}\right)\text{d}t+C\int_{0}^{T}(1+\|\varrho^{\frac{1}{2}} \dot{u}\|_{L^{2}}^{2})\text{d}t\nonumber\\
\leq& C(T)
\end{align}
owing to \eqref{eq5.2} and \eqref{eq4.11}.  On the other hand, it follows from H\"{o}lder's and Gagliardo-Nirenberg inequalities,  and \eqref{eq5.2} that
\begin{align}\label{eq4.27}
&\int_{0}^{T} \left(\||\nabla d||\nabla^{2}d|\|_{L^{q}}^{\frac{q+1}{q}}+t\||\nabla d||\nabla^{2}d|\|_{L^{q}}^{2}\right)\text{d}t\nonumber\\
\leq &C\int_{0}^{T} \left(\left(\|\nabla d\|_{L^{2q}}\|\nabla^{2}d\|_{L^{2q}}\right)^{\frac{q+1}{q}}+t\left(\|\nabla d\|_{L^{2q}}\|\nabla^{2}d\|_{L^{2q}}\right)^{2}\right)\text{d}t\nonumber\\
\leq&C\int_{0}^{T}  \left(\left(\|\nabla d\|_{L^{2}}^{\frac{1}{q}}\|\nabla^{2} d\|_{L^{2}}\|\nabla^{3}d\|_{L^{2}}^{1-\frac{1}{q}}\right)^{\frac{q+1}{q}}+t\left(\|\nabla d\|_{L^{2}}^{\frac{1}{q}}\|\nabla^{2} d\|_{L^{2}}\|\nabla^{3}d\|_{L^{2}}^{1-\frac{1}{q}}\right)^{2}\right)\text{d}t\nonumber\\
\leq& C\!\int_{0}^{T} \!\left(\|\nabla^{3} d\|_{L^{2}}^{1-\frac{1}{q^{2}}}+t\|\nabla^{3}d\|_{L^{2}}^{2-\frac{2}{q}}\right)\text{d}t
\leq C\!\int_{0}^{T}\! \left(\|\nabla^{3} d\|_{L^{2}}^{2}+t^{q}+1\right)\text{d}t\leq C(T).
\end{align}
Hence, combining \eqref{eq4.25}, \eqref{eq4.26} and \eqref{eq4.27} together, it follows that
\begin{align}\label{eq4.28}
\int_{0}^{T} \|\nabla u\|_{L^{\infty}}\text{d}t\leq C(T).
\end{align}
Thus, applying Gronwall's inequality to \eqref{eq4.24} yields
\begin{align}\label{eq4.29}
\sup_{0\leq t\leq T} \|\nabla \varrho\|_{L^{2}\cap L^{q}}\leq C(T).
\end{align}

Finally, it is easy to deduce from \eqref{eq4.7}, \eqref{eq4.26}, \eqref{eq4.27}, \eqref{eq5.2} and \eqref{eq5.4} that
\begin{align*}
\int_{0}^{T} \!\!\left(\|\nabla^{2} u\|_{L^{2}}^{2}\!+\!\|\nabla P\|_{L^{2}}^{2}\!+\!\|\nabla^{2} u\|_{L^{q}}^{\frac{q+1}{q}}\!\!+\!\|\nabla P\|_{L^{q}}^{\frac{q+1}{q}}\!\!+t(\|\nabla^{2} u\|_{L^{2}\cap L^{q}}^{2}\!+\!\|\nabla P\|_{L^{2}\cap L^{q}}^{2})\!\right)\text{d}t
\leq C(T).
\end{align*}
This together with \eqref{eq3.6} and \eqref{eq4.29} ensures \eqref{eq4.23}. This completes the proof of Lemma \ref{lem4.5}.
\end{proof}
\medskip

In the following Lemma, we shall give some spatial estimate on $\nabla \varrho$, $\nabla d$ and $\nabla^{2}d$, which are crucial to derive the estimates on the gradients of both $u_{t}$, $d_{t}$ and $\nabla d_{t}$.

\begin{lemma}\label{lem4.6}
There exists a positive constant $C$ depending on $T$ such that
\begin{align}\label{eq4.30}
&\sup_{0\leq t\leq T} \|\varrho \bar{x}^{a}\|_{L^{1}\cap H^{1}\cap W^{1,q}}\leq C(T),\\
      \label{eq4.31}
&\sup_{0\leq t\leq T} \|\nabla d\bar{x}^{\frac{a}{2}}\|_{L^{2}}^{2} +\int_{0}^{T} \|\nabla^{2} d\bar{x}^{\frac{a}{2}}\|_{L^{2}}^{2}\text{d}t
\leq C(T),
\end{align}
and
\begin{align}\label{eq4.32}
\sup_{0\leq t\leq T} t\|\nabla^{2} d\bar{x}^{\frac{a}{2}}\|_{L^{2}}^{2} +\int_{0}^{T} t\|\nabla^{3} d\bar{x}^{\frac{a}{2}}\|_{L^{2}}^{2}\text{d}t
\leq C(T).
\end{align}
\end{lemma}

\begin{proof}
With \eqref{eq4.22} in hand, the proof of \eqref{eq4.30} is exactly the same as \cite[Lemma 3.6]{LSZ}. Similarly to \eqref{eq3.5}, we can prove \eqref{eq4.31} with some suitable revisions, and we omit it for simplicity.

 Now, multiplying \eqref{eq3.8}  by $\Delta \nabla d \bar{x}^{a}$ and integrating by parts, it follows that
              \begin{align}\label{eq4.33}
              \frac{1}{2}\frac{d}{dt} \|\nabla^{2}d\bar{x}^{\frac{a}{2}}\|_{L^{2}}^{2}+&\! \|\nabla^{3} d\bar{x}^{\frac{a}{2}}\|_{L^{2}}^{2}
              =-\!\int\! \nabla d_{t}\nabla^{2} d\nabla \bar{x}^{a}\text{d}x+\!\int \!\nabla(u\cdot\nabla d)\nabla \Delta d\bar{x}^{a}\text{d}x\nonumber\\
              &-\!\int\! \nabla(|\nabla d|^{2}d)\nabla \Delta d\bar{x}^{a}\text{d}x
              -2\int \nabla^{3}d\nabla^{2}d\nabla\bar{x}^{a}\text{d}x -\int |\nabla^{2} d|^{2} \nabla^{2}\bar{x}^{a}\text{d}x
              \nonumber\\
              \leq&\frac{1}{4} \!\|\nabla^{3} d\bar{x}^{\frac{a}{2}}\|_{L^{2}}^{2}\!+\!\int\!|\Delta\nabla d||\nabla^{2} d|\nabla\bar{x}^{a}\text{d}x
              \!+\!\int\! | u||\nabla d| |\nabla^{3} d|\nabla\bar{x}^{a}\text{d}x\!+\!\int\! | u||\nabla d| |\nabla^{2} d|\nabla^{2}\bar{x}^{a}\text{d}x\nonumber\\
              &+\!\int\! |u||\nabla^{2}d|^{2}\nabla\bar{x}^{a}\text{d}x+\int |\nabla d|^{3}|\nabla^{2} d| \nabla \bar{x}^{a}\text{d}x+\int |\nabla d||\nabla^{2} d|^{2}\nabla \bar{x}^{a}\text{d}x\nonumber\\
              &+\int |\nabla u|^{2}|\nabla d|^{2}\bar{x}^{a}\text{d}x+\int |\nabla d|^{6}\bar{x}^{a}\text{d}x+\int
              |\nabla d|^{2}|\nabla^{2} d|^{2}\bar{x}^{a}\text{d}x+\int |\nabla^{2} d|^{2} \nabla^{2}\bar{x}^{a}\text{d}x\nonumber\\
              \triangleq&\frac{1}{4} \|\nabla^{3} d\bar{x}^{\frac{a}{2}}\|_{L^{2}}^{2}+K_{1}+K_{2}+\cdots+K_{10}.
              \end{align}
              We estimate all the terms $K_{i} (i=1,2,\cdots, 10)$ on the right side of \eqref{eq4.33} term by term as follows,
              \begin{align*}
              K_{1}\leq & \int|\nabla\Delta d| |\nabla^{2} d| \bar{x}^{a}\text{d}x\leq \frac{1}{16} \|\nabla^{3}d\bar{x}^{\frac{a}{2}}\|_{L^{2}}^{2}
              +C\|\nabla^{2} d\bar{x}^{\frac{a}{2}}\|_{L^{2}}^{2},\nonumber\\
              K_{2}\leq & \int |u||\nabla d||\nabla^{3} d|\bar{x}^{a-\frac{3}{4}}\text{d}x\leq \frac{1}{32}\|\nabla^{3}d\bar{x}^{\frac{a}{2}}\|_{L^{2}}^{2}
              +C\|\nabla^{2} d\bar{x}^{\frac{a}{2}}\|_{L^{4}}^{2}\|u\bar{x}^{-\frac{3}{4}}\|_{L^{2}}^{2}\nonumber\\
              \leq&\frac{1}{32}\|\nabla^{3}d\bar{x}^{\frac{a}{2}}\|_{L^{2}}^{2}\!
              +\!C\|\nabla^{2} d\bar{x}^{\frac{a}{2}}\|_{\!L^{2}}(\|\nabla^{3}d\bar{x}^{\frac{a}{2}}\|_{\!L^{2}}\!+\!\|\nabla^{2}d\nabla \bar{x}^{\frac{a}{2}}\|_{\!L^{2}}\!)(\|\varrho^{\frac{1}{2}} u\|_{\!L^{2}}^{2}\!+\!\|\nabla u\|_{\!L^{2}}^{2})\nonumber\\
              \leq&\frac{1}{16}\|\nabla^{3}d\bar{x}^{\frac{a}{2}}\|_{L^{2}}^{2}
              +C\|\nabla^{2} d\bar{x}^{\frac{a}{2}}\|_{L^{2}}^{2},\nonumber\\
              K_{3}\leq &C\|\nabla^{2} d\bar{x}^{\frac{a}{2}}\|_{L^{2}}
              \|\nabla d\bar{x}^{\frac{a}{2}}\|_{L^{4}} \|u\bar{x}^{-\frac{3}{4}}\|_{L^{4}}\nonumber\\
              \leq & \|\nabla^{2} d\bar{x}^{\frac{a}{2}}\|_{L^{2}}
              \|\nabla d\bar{x}^{\frac{a}{2}}\|_{L^{2}}(\|\nabla^{2} d\bar{x}^{\frac{a}{2}}\|_{L^{2}}+\|\nabla d \nabla\bar{x}^{\frac{a}{2}}\|_{L^{2}})(\|\varrho^{\frac{1}{2}}u\|_{L^{2}}+\|\nabla u\|_{L^{2}})\nonumber\\
              \leq& C(1+\|\nabla^{2} d\bar{x}^{\frac{a}{2}}\|_{L^{2}}^{2})\\
              K_{4}\leq &C \|\nabla^{2} d\bar{x}^{\frac{a}{2}}\|_{L^{2}} \|\nabla^{2}d\bar{x}^{\frac{a}{2}}\|_{L^{4}} \|u\bar{x}^{-\frac{3}{4}}\|_{L^{4}}
              \nonumber\\
              \leq&\|\nabla^{2} d\bar{x}^{\frac{a}{2}}\|_{L^{2}}^{\frac{3}{2}}
               \|\nabla^{2}d\bar{x}^{\frac{a}{2}}\|_{L^{2}}^{\frac{1}{2}}(\|\nabla^{3} d\bar{x}^{\frac{a}{2}}\|_{L^{2}}+\|\nabla^{2}d\nabla \bar{x}^{\frac{a}{2}}\|_{L^{2}})^{\frac{1}{2}} (\|\varrho^{\frac{a}{2}} u\|_{L^{2}}
              +\|\nabla u\|_{L^{2}})\nonumber\\
              \leq&\frac{1}{16}\|\nabla^{3}d\bar{x}^{\frac{a}{2}}\|_{L^{2}}^{2}
              +C\|\nabla^{2} d\bar{x}^{\frac{a}{2}}\|_{L^{2}}^{2},\nonumber\\
              K_{5}\leq &C \int |\nabla d|^{3} |\nabla^{2}d|\bar{x}^{a}\text{d}x\leq C\|\nabla^{2}d\bar{x}^{\frac{a}{2}}\|_{L^{2}}
              \|\nabla d\bar{x}^{\frac{a}{2}}\|_{L^{4}} \|\nabla d\|_{L^{8}}^{2}\nonumber\\
              \leq&  C\|\nabla^{2}d\bar{x}^{\frac{a}{2}}\|_{L^{2}}
              \|\nabla d\bar{x}^{\frac{a}{2}}\|_{L^{2}}^{\frac{1}{2}}(\|\nabla^{2}d\bar{x}^{\frac{a}{2}}\|_{L^{2}}+\|\nabla d\nabla\bar{x}^{\frac{a}{2}}\|_{L^{2}})^{\frac{1}{2}} \|\nabla d\|_{L^{2}}^{\frac{1}{2}} \|\nabla^{2} d\|_{L^{2}}^{\frac{3}{2}}\nonumber\\
              \leq& C(1+ \|\nabla^{2} d\bar{x}^{\frac{a}{2}}\|_{L^{2}}^{2}),\nonumber\\
              K_{6}\leq\! C\!\! & \int\! |\nabla d||\nabla^{2} \!d|\nabla\bar{x}^{a}\text{d}x\!\leq\!C\! \|\nabla d\bar{x}^{\frac{a}{2}}\|_{\!L^{2}}\|\nabla^{2}\!d\bar{x}^{\frac{a}{2}}\|_{\!L^{4}}^{2}\!
              \leq\! C\|\nabla^{2}\!d\bar{x}^{\frac{a}{2}}\|_{\!L^{2}}(\|\nabla^{3}\!d\bar{x}^{\frac{a}{2}}\|_{\!L^{\!2}}
              \!+\!\|\nabla^{2}\!d\bar{x}^{\frac{a}{2}}\|_{\!L^{\!2}})\nonumber\\
              \leq& \frac{1}{16}\|\nabla^{3}d\bar{x}^{\frac{a}{2}}\|_{L^{2}}^{2}
              +C\|\nabla^{2} d\bar{x}^{\frac{a}{2}}\|_{L^{2}}^{2},\nonumber\\
              K_{7}\leq& C\|\nabla d\bar{x}^{\frac{a}{2}}\|_{L^{4}}^{4} \|\nabla u\|_{L^{4}}^{2}
              \leq C \|\nabla d\bar{x}^{\frac{a}{2}}\|_{L^{2}} (\|\nabla^{2} d\bar{x}^{\frac{a}{2}}\|_{L^{2}} +\|\nabla d\nabla \bar{x}^{\frac{a}{2}}\|_{L^{2}} )\|\nabla u\|_{L^{2}}\|\nabla^{2} u\|_{L^{2}}\nonumber\\
              \leq& C\|\nabla^{2} u\|_{L^{2}}^{2}+C(1+ \|\nabla^{2} d\bar{x}^{\frac{a}{2}}\|_{L^{2}}),\nonumber\\
              K_{8}+ K_{9}\leq& C\|\nabla d\|_{L^{8}}^{4} \|\nabla^{2}d\bar{x}^{\frac{a}{2}}\|_{L^{4}}^{2}+C\|\nabla d\|_{L^{4}}^{2}\|\nabla^{2}d\bar{x}^{\frac{a}{2}}\|_{L^{4}}^{2}\nonumber\\
              \leq& C(\|\nabla d\|_{L^{2}}\|\nabla^{2}d\|_{L^{2}}^{2}+\|\nabla d\|_{L^{2}}\|\nabla^{2} d\|_{L^{2}}) \|\nabla^{2} d\bar{x}^{\frac{a}{2}}\|_{L^{2}}
              (\|\nabla^{3}d\bar{x}^{\frac{a}{2}}\|_{L^{2}}\!+\!\|\nabla^{2}d\nabla\bar{x}^{\frac{a}{2}}\|_{L^{2}})\nonumber\\
              \leq& \frac{1}{16}\|\nabla^{3}d\bar{x}^{\frac{a}{2}}\|_{L^{2}}^{2}+C\|\nabla^{2}d\bar{x}^{\frac{a}{2}}\|_{L^{2}}^{2},
              \nonumber\\
              K_{10}\leq & \int |\nabla^{2} d|^{2} \bar{x}^{a}\text{d}x\leq C\|\nabla^{2} d\bar{x}^{\frac{a}{2}}\|_{L^{2}}^{2},
              \end{align*}
              where we have used H\"{o}lder's and Gagliardo-Nirenberg inequalities, \eqref{eq5.2}, \eqref{eq5.4}, \eqref{eq3.14} and \eqref{eq4.31}. Inserting these estimates
              $K_{i} (i=1,2,\cdots,10)$ above into \eqref{eq4.33}, after by using \eqref{eq4.21}, we have
              \begin{align*}
              \frac{d}{dt} \|\nabla^{2}d\bar{x}^{\frac{a}{2}}\|_{L^{2}}^{2}+&\! \|\nabla^{3} d\bar{x}^{\frac{a}{2}}\|_{L^{2}}^{2}
              \leq C\|\nabla^{2}d\bar{x}^{\frac{a}{2}}\|_{L^{2}}^{2}+C(1+\|\nabla^{2} u\|_{L^{2}}^{2}).
              \end{align*}
              Multiplying the above inequality with $t$, and then using  Gronwall's inequality and \eqref{eq4.23},  we prove \eqref{eq4.32}, this completes the proof of Lemma \ref{lem4.6}.
\end{proof}

\begin{lemma}\label{lem4.7}
There exists a positive constant $C$ depending on $T$ such that
\begin{align}\label{eq4.34}
&\sup_{0\leq t\leq T} t\left(\|\varrho^{\frac{1}{2}} u_{t}\|_{L^{2}}^{2}+\|d_{t}\|_{H^{1}}^{2}+\|\nabla^{3} d\|_{L^{2}}^{2}\right)+
\int_{0}^{T} t\left(\|\nabla u_{t}\|_{L^{2}}^{2}+\|\nabla d_{t}\|_{H^{1}}^{2}\right)\text{d}t
\leq C(T).
\end{align}
\end{lemma}

\begin{proof}
We shall first prove that
\begin{align}\label{eq4.35}
\int_{0}^{T} \left(\|\varrho^{\frac{1}{2}} u_{t}\|_{L^{2}}^{2}+\|\nabla d_{t}\|_{L^{2}}^{2}\right)\text{d}t\leq C(T).
\end{align}
Indeed, it is easy to see
\begin{align}\label{eq4.36}
\|\varrho^{\frac{1}{2}} u_{t}\|_{L^{2}}^{2}\leq & \|\varrho^{\frac{1}{2}}\dot{u}\|_{L^{2}}^{2}+\|\varrho^{\frac{1}{2}}|u||\nabla u|\|_{L^{2}}^{2}
\leq \|\varrho^{\frac{1}{2}}\dot{u}\|_{L^{2}}^{2}+C\|\varrho^{\frac{1}{2}} u\|_{L^{6}}^{2} \|\nabla u\|_{L^{3}}^{2}\nonumber\\
\leq&  \|\varrho^{\frac{1}{2}}\dot{u}\|_{L^{2}}^{2}+C(\|\varrho^{\frac{1}{2}} u\|_{L^{2}}^{2}+\|\nabla u\|_{L^{2}}^{2}) \|\nabla u\|_{L^{2}}^{\frac{4}{3}} \|\nabla^{2} u\|_{L^{2}}^{\frac{2}{3}}\nonumber\\
\leq & \|\varrho^{\frac{1}{2}}\dot{u}\|_{L^{2}}^{2}+C(1+ \|\nabla^{2} u\|_{L^{2}}^{2})
\end{align}
due to Lemma \ref{lem2.4}, \eqref{eq5.2} and \eqref{eq5.4}. By using \eqref{eq3.8} and Lemma \ref{lem2.3}, it follows that
\begin{align}\label{eq4.37}
\|\nabla d_{t}\|_{L^{2}}^{2}\leq& C(\|\nabla^{3} d\|_{L^{2}}^{2}+\||\nabla u| |\nabla d|\|_{L^{2}}^{2}+\||u||\nabla^{2}d|\|_{L^{2}}^{2}+\||\nabla d|^{3}\|_{L^{2}}^{2}+\||\nabla d||\nabla^{2}d|\|_{L^{2}}^{2})\nonumber\\
\leq& C(\|\nabla^{3} d\|_{L^{2}}^{2}+\|\nabla u\|_{L^{4}}^{2} \|\nabla d\|_{L^{4}}^{2}+\|u\bar{x}^{-\frac{a}{4}}\|_{L^{8}}^{2}
\|\nabla^{2}d\bar{x}^{\frac{a}{2}}\|_{L^{2}}\|\nabla^{2}d\|_{L^{4}}
+\|\nabla d\|_{L^{6}}^{6}+\||\nabla d||\nabla^{2}d|\|_{L^{2}}^{2})\nonumber\\
\leq& C(\|\nabla^{3} d\|_{L^{2}}^{2}+ \|\nabla^{2} u\|_{L^{2}}^{2}+\|\nabla^{2}d\bar{x}^{\frac{a}{2}}\|_{L^{2}}^{2}+\|u\bar{x}^{-\frac{a}{4}}\|_{L^{8}}^{4}
\|\nabla^{2}d\|_{L^{4}}^{2}+1)\nonumber\\
\leq& C(\|\nabla^{3} d\|_{L^{2}}^{2}+ \|\nabla^{2} u\|_{L^{2}}^{2}+\|\nabla^{2}d\bar{x}^{\frac{a}{2}}\|_{L^{2}}^{2}+(\|\varrho^{\frac{1}{2}} u\|_{L^{2}}^{2}+\|\nabla u\|_{L^{2}}^{2})^{2}
\|\nabla^{2}d\|_{L^{4}}^{2}+1)\nonumber\\
\leq&C(\|\nabla^{3} d\|_{L^{2}}^{2}+ \|\nabla^{2} u\|_{L^{2}}^{2}+1)
\end{align}
owing to \eqref{eq5.2}, \eqref{eq5.4} and \eqref{eq4.31}. Combining \eqref{eq4.36}, \eqref{eq4.37}, \eqref{eq5.2}  and \eqref{eq4.23} together, it is easy to see that  \eqref{eq4.35} holds.

Now,
multiplying \eqref{eq3.25} by $u_{t}$ and integrating the resulting equality by parts over $\mathbb{R}^{2}$, we obtain after using $\eqref{eq1.1}_{1}$ and the divergence free condition $\eqref{eq1.1}_{4}$ that
\begin{align}\label{eq4.38}
\frac{1}{2}\frac{d}{dt} \|\varrho^{\frac{1}{2}} u_{t}\|_{L^{2}}^{2}&+\|\nabla u_{t}\|_{L^{2}}^{2}\leq C\int\! \varrho |u| |u_{t}|\left( |\nabla u_{t}|\!+ \!|\nabla u |^{2}\! +\!|u||\nabla^{2} u|\right)\text{d}x\nonumber\\
& +C\int \varrho |u|^{2}|\nabla u| |\nabla u_{t}|\text{d}x+ C\int \varrho |u_{t}|^{2}|\nabla u|\text{d}x +C\int|\nabla d||\nabla d_{t}|
|\nabla u_{t}|\text{d}x\nonumber\\
\triangleq& M_{1}+M_{2}+M_{3}+M_{4}.
\end{align}
By using \eqref{eq5.2}, \eqref{eq5.4}, H\"{o}lder's and Gagliardo-Nirenberg inequalities, it follows that
\begin{align*}
M_{1}\leq &C\|\varrho^{\frac{1}{2}} u\|_{L^{6}} \|\varrho^{\frac{1}{2}} u_{t}\|_{L^{2}}^{\frac{1}{2}}
\|\varrho^{\frac{1}{2}} u_{t}\|_{L^{6}}^{\frac{1}{2}} (\|\nabla u_{t}\|_{L^{2}}+\|\nabla u\|_{L^{4}}^{2})
+C\|\varrho^{\frac{1}{4}} u\|_{L^{12}}^{2} \|\varrho^{\frac{1}{2}} u_{t}\|_{L^{2}}^{\frac{1}{2}}
\|\varrho^{\frac{1}{2}} u_{t}\|_{L^{6}}^{\frac{1}{2}} \|\nabla^{2} u\|_{L^{2}}\nonumber\\
\leq& C(1\!+\!\|\nabla u\|_{L^{2}}^{2}) \|\varrho^{\frac{1}{2}} u_{t}\|_{L^{2}}^{\frac{1}{2}}
(\|\varrho^{\frac{1}{2}} u_{t}\|_{L^{2}}\! +\!\|\nabla u_{t}\|_{L^{2}})^{\frac{1}{2}} (\|\nabla u_{t}\|_{L^{2}}
\!+\!\|\nabla u\|_{L^{2}}^{2}\!+\!\|\nabla u\|_{L^{2}} \|\nabla^{2} u\|_{L^{2}}\!+\!\|\nabla^{2}u\|_{L^{2}})\nonumber\\
\leq& \frac{1}{6} \|\nabla u_{t}\|_{L^{2}}^{2}+ C(1+\|\varrho^{\frac{1}{2}} u_{t}\|_{L^{2}}^{2}
+\|\nabla^{2} u\|_{L^{2}}^{2}),\\
M_{2}+&M_{3}\leq C \|\varrho^{\frac{1}{2}} u\|_{L^{8}}^{2} \|\nabla u\|_{L^{4}}\|\nabla u_{t}\|_{L^{2}}+C\|\varrho^{\frac{1}{2}} u_{t}\|_{L^{6}}^{\frac{3}{2}} \|\varrho^{\frac{1}{2}} u_{t}\|_{L^{2}}^{\frac{1}{2}}\|\nabla u\|_{L^{2}}\nonumber\\
\leq& \frac{1}{6} \|\nabla u_{t}\|_{L^{2}}^{2}+ C(1+\|\varrho^{\frac{1}{2}} u_{t}\|_{L^{2}}^{2}
+\|\nabla^{2} u\|_{L^{2}}^{2}),\\
M_{4}\leq & C\|\nabla d\|_{L^{4}} \|\nabla d_{t}\|_{L^{4}} \|\nabla u_{t}\|_{L^{2}}
\leq  \frac{1}{6} \|\nabla u_{t}\|_{L^{2}}^{2} +C\|\nabla d\|_{L^{2}} \|\nabla^{2} d\|_{L^{2}} \|\nabla d_{t}\|_{L^{2}}\|\nabla^{2} d_{t}\|_{L^{2}}
\nonumber\\
\leq &\frac{1}{6} \|\nabla u_{t}\|_{L^{2}}^{2} +\frac{1}{4C_{3}+1}\|\nabla^{2} d_{t}\|_{L^{2}}^{2}+C\|\nabla d_{t}\|_{L^{2}}^{2},
\end{align*}
where the positive constant $C_{3}$ is defined in the following \eqref{eq4.40} and \eqref{eq4.43}. Substituting the estimates of $M_{i} (i=1,2,\cdots, 4)$ into \eqref{eq4.38},  it follows that
\begin{align}\label{eq4.39}
&\frac{d}{dt} \|\varrho^{\frac{1}{2}} u_{t}\|_{L^{2}}^{2}+\|\nabla u_{t}\|_{L^{2}}^{2}
\leq C(\|\varrho^{\frac{1}{2}} u_{t}\|_{L^{2}}^{2} +\|\nabla d_{t}\|_{L^{2}}^{2})+ \frac{1}{4C_{3}+1}\|\nabla^{2} d_{t}\|_{L^{2}}^{2}+C (
\|\nabla^{2} u\|_{L^{2}}^{2}+1).
\end{align}

Differentiating $\eqref{eq1.1}_{3}$ with respect to time variable $t$, multiplying the resulting equality with $d_{t}$ and then integrating by parts over $\mathbb{R}^{2}$, we have
\begin{align*}
\frac{1}{2}\frac{d}{dt} \|d_{t}\|_{L^{2}}^{2}+\|\nabla d_{t}\|_{L^{2}}^{2}\leq& C\int |u_{t}||\nabla d||d_{t}|\text{d}x +C\int |\nabla d_{t}| |\nabla d||d_{t}|\text{d}x+C\int |\nabla d|^{2} |d_{t}|^{2}\text{d}x\nonumber\\
\triangleq & M_{5}+M_{6}+M_{7}.
\end{align*}
By using  H\"{o}lder's and Gagliardo-Nirenberg inequalities,  \eqref{eq5.2}, \eqref{eq5.4}, \eqref{eq3.14} and \eqref{eq4.31}, we have
\begin{align*}
M_{5}\leq &\|u_{t}\bar{x}^{-\frac{a}{2}}\|_{L^{4}} \|\nabla d \bar{x}^{\frac{a}{2}}\|_{L^{2}} \|d_{t}\|_{L^{4}}\leq (\|\varrho^{\frac{1}{2}} u_{t}\|_{L^{2}}+\|\nabla u_{t}\|_{L^{2}})\|\nabla d \bar{x}^{\frac{a}{2}}\|_{L^{2}} \|d_{t}\|_{L^{2}}^{\frac{1}{2}}\| \nabla d_{t}\|_{L^{2}}^{\frac{1}{2}}\nonumber\\
\leq& \frac{1}{4}\|\nabla u_{t}\|_{L^{2}}^{2}+C(\|\varrho^{\frac{1}{2}} u_{t}\|_{L^{2}}^{2}+ \|d_{t}\|_{L^{2}}^{2}+\|\nabla d_{t}\|_{L^{2}}^{2}),\nonumber\\
M_{6}+&M_{7}\leq \frac{1}{2}\|\nabla d_{t}\|_{L^{2}}+C\|\nabla d\|_{L^{4}}^{2}\|d_{t}\|_{L^{4}}^{2}
\leq \frac{1}{2}\|\nabla d_{t}\|_{L^{2}}+C\|\nabla d\|_{L^{2}}\|\nabla^{2} d\|_{L^{2}}\|d_{t}\|_{L^{2}}\| \nabla d_{t}\|_{L^{2}}\nonumber\\
\leq& \frac{1}{2}\|\nabla d_{t}\|_{L^{2}}+C\|d_{t}\|_{L^{2}}^{2}.
\end{align*}
Hence
\begin{align}\label{eq4.40}
\frac{d}{dt} \|d_{t}\|_{L^{2}}^{2}+\|\nabla d_{t}\|_{L^{2}}^{2}\leq& C_{3}\|\nabla u_{t}\|_{L^{2}}^{2}+C(\|\varrho^{\frac{1}{2}} u_{t}\|_{L^{2}}^{2}+ \|d_{t}\|_{L^{2}}^{2}+\|\nabla d_{t}\|_{L^{2}}^{2}).
\end{align}

Differentiating \eqref{eq3.8} with respect to time variable $t$ ensures
\begin{align}\label{eq4.41}
\nabla d_{tt}-\Delta\nabla d_{t} =-\nabla (u\cdot \nabla d)_{t}+\nabla (|\nabla d|^{2}d)_{t}.
\end{align}
Multiplying \eqref{eq4.41} by $\nabla d_{t}$, and integrating the resulting equality over $\mathbb{R}^{2}$, it follows that
\begin{align}\label{eq4.42}
\frac{1}{2}\frac{d}{dt}\|\nabla d_{t}\|_{L^{2}}^{2}&+\|\nabla^{2} d_{t}\|_{L^{2}}^{2}
\leq  C\int |\nabla u_{t}||\nabla d||\nabla d_{t}|\text{d}x+C\int |\nabla u||\nabla d_{t}|^{2}\text{d}x
\nonumber\\
&+
C\int |u_{t}||\nabla d| |\nabla^{2} d_{t}|\text{d}x
+C\int |\nabla d|^{2}|d_{t}||\nabla^{2} d_{t}|\text{d}x+C\int |\nabla d||\nabla d_{t}| |\nabla^{2} d_{t}|\text{d}x\nonumber\\
&\triangleq M_{8}+M_{9}+M_{10}+M_{11}+M_{12}.
\end{align}
By using  H\"{o}lder's and Gagliardo-Nirenberg inequalities,  \eqref{eq5.2}, \eqref{eq5.4}, \eqref{eq3.14} and \eqref{eq4.31} we have
\begin{align*}
M_{8}\leq& \|\nabla u_{t}\|_{L^{2}}\|\nabla d_{t}\|_{L^{4}}\|\nabla d\|_{L^{4}}\leq \frac{1}{2} \|\nabla u_{t}\|_{L^{2}}^{2}+C
\|\nabla d\|_{L^{2}}\|\nabla^{2} d\|_{L^{2}} \|\nabla d_{t}\|_{L^{2}}\|\nabla^{2} d_{t}\|_{L^{2}}\nonumber\\
\leq& \frac{1}{2} \|\nabla u_{t}\|_{L^{2}}^{2}+ \frac{1}{4} \|\nabla^{2} d_{t} \|_{L^{2}}^{2}
+C\|\nabla d_{t}\|_{L^{2}}^{2},\nonumber\\
M_{9}\leq &C\|\nabla u\|_{L^{2}}\|\nabla d_{t}\|_{L^{4}}^{2}
\leq C\|\nabla u\|_{L^{2}}\|\nabla d_{t}\|_{L^{2}} \|\nabla^{2} d_{t}\|_{L^{2}}
\leq \frac{1}{8} \|\nabla^{2} d_{t}\|_{L^{2}} +C\|\nabla d_{t}\|_{L^{2}},\nonumber\\
 M_{10}\leq \!C\!\!&\int\!\! |u_{t}\bar{x}^{-\frac{2a-1}{4}}| |\nabla d \bar{x}^{\frac{a}{2}}|^{\frac{2a-1}{2a}}
 |\nabla d|^{\frac{1}{2a}}|\nabla^{2} \!d_{t}|\text{d}x
\!\leq\! C\|u_{t}\bar{x}^{-\frac{2a-1}{4}}\|_{L^{8a}}\|\nabla d\bar{x}^{\frac{a}{2}}\|_{L^{2}}^{\frac{2a-1}{2a}} \|\nabla d\|_{L^{2}}^{\frac{1}{4a}} \|\nabla^{2} d\|_{L^{2}}^{\frac{1}{4a}}
\|\nabla^{2} d_{t}\|_{L^{2}}\nonumber\\
\leq& \frac{1}{8} \|\nabla^{2}d_{t}\|_{L^{2}}^{2}+C\|u_{t}\bar{x}^{-\frac{2a-1}{4}}\|_{L^{8a}}^{2}
\leq \frac{1}{8} \|\nabla^{2}d_{t}\|_{L^{2}}^{2}
+C(\|\varrho^{\frac{1}{2}} u_{t}\|_{L^{2}}+\|\nabla u_{t}\|_{L^{2}}^{2}),\nonumber\\
M_{11}\leq & \frac{1}{8}\|\nabla^{2} d_{t}\|_{L^{2}}^{2}+C\|\nabla d\|_{L^{8}}^{2}\| d_{t}\|_{L^{4}}^{2}\leq  \frac{1}{8}\|\nabla^{2} d_{t}\|_{L^{2}}^{2}+C\|\nabla d\|_{L^{2}}^{\frac{1}{2}}\|\nabla^{2} d\|_{L^{2}}^{\frac{3}{2}}\| d_{t}\|_{L^{2}}\|\nabla d_{t}\|_{L^{2}}\nonumber\\
\leq&  \frac{1}{8}\|\nabla^{2} d_{t}\|_{L^{2}}^{2}+C(\| d_{t}\|_{L^{2}}^{2}+\|\nabla d_{t}\|_{L^{2}}^{2}),\nonumber\\
M_{12}\leq & \frac{1}{16}\|\nabla^{2} d_{t}\|_{L^{2}}^{2}+C\|\nabla d\|_{L^{4}}^{2}\| \nabla d_{t}\|_{L^{4}}^{2}\leq \frac{1}{8}\|\nabla^{2} d_{t}\|_{L^{2}}^{2}+C\|\nabla d_{t}\|_{L^{2}}^{2}.
\end{align*}
Inserting the estimates of $M_{i} (i=8,9,\cdots, 12)$ into \eqref{eq4.42}, it follows that
\begin{align}\label{eq4.43}
&\frac{d}{dt}\|\nabla d_{t}\|_{L^{2}}^{2}+\|\nabla^{2} d_{t}\|_{L^{2}}^{2}
\leq C_{3} (\|\varrho^{\frac{1}{2}} u_{t}\|_{L^{2}}+\|\nabla u_{t}\|_{L^{2}}^{2})+C
 (\| d_{t}\|_{L^{2}}^{2}+\|\nabla d_{t}\|_{L^{2}}^{2}).
\end{align}

Next, multiplying \eqref{eq4.39} by  ${2C_{3}+1}$ and adding the resulting inequality with \eqref{eq4.40} and  \eqref{eq4.43}, we have
\begin{align*}
&\frac{d}{dt}({(2C_{2}+1)}\|\varrho^{\frac{1}{2}} u_{t}\|_{L^{2}}^{2}+\|d_{t}\|_{H^{1}}^{2})+\|\nabla u_{t}\|_{L^{2}}
+\frac{1}{2} \|\nabla d_{t}\|_{H^{1}}^{2} \nonumber\\
\leq& C(1+\|\varrho^{\frac{1}{2}} u_{t}\|_{L^{2}}^{2}+\|\nabla d_{t}\|_{H^{1}}^{2})
+C(1+\|\nabla^{2} u\|_{L^{2}}^{2}).
\end{align*}
Multiplying the above inequality by $t$, we obtain
\begin{align}\label{eq4.44}
\sup_{0\leq t\leq T} t\left(\|\varrho^{\frac{1}{2}} u_{t}\|_{L^{2}}^{2}+\|\nabla d_{t}\|_{H^{1}}^{2}\right)
+\int_{0}^{T} t\left(\|\nabla u_{t}\|_{L^{2}}^{2}+\|\nabla d_{t}\|_{H^{1}}^{2}\right)\text{d}t\leq C(T),
\end{align}
after by using Gronwall's inequality and \eqref{eq4.08}.

Finally, it follows from H\"{o}lder's and Gagliardo-Nirenberg inequalities,  \eqref{eq5.2}, \eqref{eq5.4}  and \eqref{eq3.14}  that
\begin{align*}
\|\nabla^{3}d\|_{L^{2}}^{2}\leq& C(\|\nabla d_{t}\|_{L^{2}}^{2}+\||\nabla u||\nabla d|\|_{L^{2}}^{2}
+\||u||\nabla^{2}d|\|_{L^{2}}^{2}+\||\nabla d|^{3}\|_{L^{2}}^{2}+\||\nabla^{2}d||\nabla d|\|_{L^{2}}^{2})\nonumber\\
\leq& C(\|\nabla d_{t}\|_{L^{2}}^{2}+\|\nabla u\|_{L^{4}}^{2} \|\nabla d\|_{L^{4}}^{2}+\|u\bar{x}^{-\frac{a}{4}}\|_{L^{8}}^{2}
\|\nabla^{2}d\bar{x}^{\frac{a}{2}}\|_{L^{2}}\|\nabla^{2}d\|_{L^{4}}
+\|\nabla d\|_{L^{6}}^{6}+\||\nabla d||\nabla^{2}d|\|_{L^{2}}^{2})\nonumber\\
\leq& C(\|\nabla d_{t}\|_{L^{2}}^{2}+ \|\nabla^{2} u\|_{L^{2}}^{2}+\|\nabla^{2}d\bar{x}^{\frac{a}{2}}\|_{L^{2}}^{2}+\|u\bar{x}^{-\frac{a}{4}}\|_{L^{8}}^{4}
\|\nabla^{2}d\|_{L^{4}}^{2}+\|\nabla d\|_{L^{2}}^{2}\|\nabla^{2}d\|_{L^{2}}^{4}+\|\nabla^{2} d\|_{L^{3}}^{2}\|\nabla d\|_{L^{6}}^{2})\nonumber\\
\leq& C(\|\nabla d_{t}\|_{L^{2}}^{2}\!+ \!\|\nabla^{2} u\|_{L^{2}}^{2}\!+\!\|\nabla^{2}d\bar{x}^{\frac{a}{2}}\|_{L^{2}}^{2}\!+\!(\|\varrho^{\frac{1}{2}} u\|_{L^{2}}^{2}\!+\!\|\nabla u\|_{L^{2}}^{2})^{2}
\|\nabla^{2}d\|_{L^{2}}\|\nabla^{3}d\|_{L^{2}}\nonumber\\
&+\|\nabla^{2}d\|_{L^{2}}^{4}+\|\nabla^{2}d\|_{L^{2}}\|\nabla^{3}d\|_{L^{2}})\nonumber\\
\leq&\frac{1}{2}\|\nabla^{3}d\|_{L^{2}}^{2} +C(\|\nabla d_{t}\|_{L^{2}}^{2}+ \|\nabla^{2} u\|_{L^{2}}^{2}+\|\nabla^{2}d\|_{L^{2}}^{4}+\|\nabla^{2}d\bar{x}^{\frac{a}{2}}\|_{L^{2}}^{2}
+\|\nabla^{2}d\|_{L^{2}}^{2})\nonumber\\
\leq &\frac{1}{2}\|\nabla^{3}d\|_{L^{2}}^{2} +C(\|\nabla d_{t}\|_{L^{2}}^{2}+ \|\nabla^{2} u\|_{L^{2}}^{2}+\|\nabla^{2}d\bar{x}^{\frac{a}{2}}\|_{L^{2}}^{2}
+\|\nabla^{2}d\|_{L^{2}}^{2}),
\end{align*}
which combined with \eqref{eq5.3}, \eqref{eq4.12}, \eqref{eq4.32} and \eqref{eq4.44} indicates \eqref{eq4.34}.
This completes the proof of Lemma \ref{lem4.7}.
\end{proof}

\subsection{The proof of Theorem \ref{thm1.2}}

In this subsection, with Theorem \ref{thm3.1} and the a priori estimates obtained in Subsections 4.1 and 4.2 at hand, we shall give the proof of Theorem \ref{thm1.2}.

By Theorem  \ref{thm3.1}, we know there exists a $T_{*}>0$ such that the Cauchy problem of system \eqref{eq1.1}--\eqref{eq1.2} admits a unique strong solution $(\varrho, u,P, d)$ on $\mathbb{R}^{2}\times (0,T_{*}]$. In what follows, we shall extend the local solution to all the time.

Set
\begin{align}\label{eq4.45}
T^{*}\!=\sup \left\{ T | (\varrho,u, P, d) \text{ is a strong solution to \eqref{eq1.1}--\eqref{eq1.2} on } \mathbb{R}^{2}\times (0,T]\right\}.
\end{align}
First, for any $0<\tau<T_{*}<T\leq T^{*}$ with $T$ finite, one deduces from \eqref{eq5.2}, \eqref{eq5.4}, \eqref{eq4.12} and \eqref{eq4.34} that for all $q\geq 2$,
\begin{align}\label{eq4.46}
\nabla u, \nabla d,\nabla^{2}d\in C([\tau,T]; L^{2}(\mathbb{R}^{2})\cap L^{q}(\mathbb{R}^{2})),
\end{align}
where one has used the standard embedding
\begin{align*}
L^{\infty}(\tau,T; H^{1}(\mathbb{R}^{2}))\cap H^{1}(\tau,T; H^{-1}(\mathbb{R}^{2})) \hookrightarrow C(\tau,T;L^{q}(\mathbb{R}^{2}))\quad \text{ for all }q\in [2,\infty).
\end{align*}
Moreover, it follows from \eqref{eq4.23}, \eqref{eq4.30} and \cite[Lemma 2.3]{Lions1} that
\begin{align}\label{eq4.47}
\varrho\in C([0,T];L^{1}(\mathbb{R}^{2})\cap H^{1}(\mathbb{R}^{2})\cap W^{1,q}(\mathbb{R}^{2})).
\end{align}

Now, we claim that
\begin{align}\label{eq4.48}
T^{*}=\infty.
\end{align}
Otherwise, if $T^{*}<\infty$, it follows from \eqref{eq4.46}, \eqref{eq4.47}, \eqref{eq5.2}, \eqref{eq5.4}, \eqref{eq4.30} and \eqref{eq4.31} that
\begin{align*}
(\varrho,u,P,d)(x,T^{*})= \lim_{t\rightarrow T^{*}}(\varrho, u,P,d)(x,t)
\end{align*}
satisfies  \eqref{eq1.7} at $t=T^{*}$. Moreover, using \eqref{eq1.5} and \eqref{eq3.6} with $p=1$, it follows that
\begin{align*}
\int_{\mathbb{R}^{2}}\varrho(x,T^{*})\text{d}x=\int_{\mathbb{R}^{2}} \varrho_{0}(x) \text{d}x=1.
\end{align*}
Notice that there exists  $N_{0}>0$, 
it is easy to see that
\begin{align*}
\int_{\mathbb{R}^{2}} \varrho(x,T^{*})\text{d}x\geq \frac{1}{2} \int_{B_{N_{0}}}\varrho(x,T^{*})\text{d}x\geq \frac{1}{2}.
\end{align*}
Thus, we can take $(\varrho,u,P,d)(x,T^{*})$ as the initial data, and Theorem \ref{thm3.1} implies that one could extend the local solutions beyond $T^{*}$. This contradicts the assumption of $T^{*}$ in \eqref{eq4.45}. Hence, we prove \eqref{eq4.48}. Furthermore, from \eqref{eq5.3}, \eqref{eq4.11}, \eqref{eq4.12}  and \eqref{eq4.34}, one obtains that \eqref{eq1.10} holds.
This completes the proof of Theorem \ref{thm1.2}.
\hfill$\Box$
\\
\\
\textbf{Acknowledgments}

The authors would like to thank Professor  Song Jiang for his helpful suggestions.

\end{document}